\documentclass[preprint,12pt]{elsarticle}

\usepackage{comment}
\usepackage{graphicx}
\usepackage{epstopdf}
\usepackage{subcaption} 
\usepackage[hidelinks]{hyperref}
\usepackage{mathrsfs}
\usepackage{amsfonts}
\usepackage{amssymb}
\usepackage{bm}
\usepackage{xcolor}
\usepackage{amsmath}
\usepackage{mathtools}
\usepackage{amsthm}
\newcounter{casenum}

\usepackage{bm}
\usepackage{xspace}

\newcommand{\Tr}{\ensuremath{^{\mr{T}}}}

\newcommand{\mr}[1]{\ensuremath{\mathrm{#1}}}

\newcommand{\fnc}[1]{\ensuremath{\mathit{#1}}}
\newcommand{\bfnc}[1]{\ensuremath{\bm{\mathit{#1}}}}

\newcommand{\xm}[0]{\ensuremath{x_{m}}}
\newcommand{\xone}[0]{x_{1}}
\newcommand{\xtwo}[0]{x_{2}}
\newcommand{\xthree}[0]{x_{3}}
\newcommand{\xil}[0]{\ensuremath{\xi_{l}}}

\newcommand{\Ohat}[0]{\ensuremath{\hat{\Omega}}}

\newcommand{\U}[0]{\ensuremath{\bfnc{U}}}
\newcommand{\Fxm}[0]{\ensuremath{\bfnc{F}}_{\xm}}

\newcommand{\Vone}[0]{\ensuremath{\fnc{V}_{1}}}
\newcommand{\Vtwo}[0]{\ensuremath{\fnc{V}_{2}}}
\newcommand{\Vthree}[0]{\ensuremath{\fnc{V}_{3}}}
\newcommand{\Vm}[0]{\ensuremath{\fnc{V}_{m}}}

\newcommand{\E}[0]{\ensuremath{\fnc{E}}}

\newcommand{\Fxmv}[0]{\ensuremath{\bm{\fnc{F}}_{\xm}^{(v)}}}

\theoremstyle{plain}
\newtheorem{theorem}{Theorem}

\newtheorem{lemma}[theorem]{Lemma}
\theoremstyle{definition}

\newtheorem{remark}{Remark}

\begin{document}

\begin{frontmatter}



\title{High-order Positivity-preserving $L_2$-stable Spectral Collocation Schemes for the 3-D compressible Navier-Stokes equations}


\author{Nail K. Yamaleev\footnote{Corresponding author. Department of Mathematics and Statistics, Tel.: +1 757 683 3423. {\it E-mail address:} nyamalee@odu.edu} and Johnathon Upperman}

\address{Old Dominion University, Norfolk, VA 23529, USA}

\begin{abstract}
This paper extends a new class of positivity--preserving, \ entropy stable spectral collocation schemes developed for the one-dimensional compressible Navier-Stokes equations in \cite{UY_1Dlow, UY_1Dhigh} to three spatial dimensions. The new high-order schemes are provably $L_2$ stable, design--order accurate for smooth solutions, and guarantee the pointwise positivity of thermodynamic variables for 3-D compressible viscous flows. Similar to the 1-D counterpart, the proposed schemes for the 3-D Navier-Stokes equations are constructed by using a flux-limiting technique that  combines a positivity-violating entropy stable method of arbitrary order of accuracy and a novel first-order positivity-preserving entropy stable finite volume-type  scheme discretized on the same Legendre-Gauss-Lobatto grid points used for constructing the high-order discrete operators. The positivity preservation and excellent discontinuity-capturing properties are achieved by adding an artificial dissipation in the form of the low- and high-order Brenner-Navier-Stokes diffusion operators. To our knowledge, this is the first family of positivity-preserving, entropy stable  schemes of arbitrary order of accuracy for the 3-D compressible Navier-Stokes equations. 
\end{abstract}


\end{frontmatter}


\section{Introduction}
A new family of high-order entropy stable spectral collocation schemes that provide pointwise positivity of thermodynamic variables for the 1-D compressible Navier-Stokes equations has recently been introduced by the authors of the present paper in \cite{UY_1Dlow, UY_1Dhigh}. The entropy stability is achieved by using summation-by-parts operators and entropy consistent fluxes for discretizing both the inviscid and viscous terms  of the symmetrized form of the compressible Navier-Stokes equations. We have proven in \cite{UY_1Dhigh} that the new high-order flux-limiting schemes are both pointwise positivity-preserving and $L_2$ stable for the Navier-Stokes equations in one spatial dimension. 

Herein, we generalize and extend the 1-D positivity-preserving entropy stable methodology of  \cite{UY_1Dlow, UY_1Dhigh} to the three-dimensional compressible Navier-Stokes equations
on static unstructured hexahedral grids. Similar to the 1-D high-order schemes in  \cite{UY_1Dhigh}, the new schemes for the 3-D compressible Navier-Stokes equations are constructed by combining
a positivity-violating entropy stable method of arbitrary order of accuracy and a novel first-order positivity-preserving entropy stable method developed
in the companion paper \cite{UY_3Dlow}. In contrast to positivity--preserving methods developed in \cite{Svard, GHKL, GMPT}, which are at most 2nd-order accurate, 
the proposed schemes provide an arbitrary order of accuracy for sufficiently smooth solutions of the 3-D compressible Navier-Stokes equations.
Unlike the positivity-preserving high-order discontinuous Galerkin (DG) method developed for the Navier-Stokes  equations in \cite{Zhang}, the new high-order schemes guarantee not only the so-called weak positivity of thermodynamic variables, 
but also the pointwise positivity at individual collocation points that are directly used for approximation of the governing equations. 
Furthermore, the positivity-preserving DG scheme developed in \cite{Zhang} imposes very severe constraints on the time step which is about an order of magnitude less than that of the baseline method for high-order polynomial bases. Note that the actual time step constraint in  \cite{Zhang} is much stiffer, because the lower bound on  the artificial viscosity coefficient required for positivity may grow dramatically, as the velocity gradients increase. 

Another distinctive feature of the proposed methodology is that the new high-order positivity-preserving schemes satisfy the discrete entropy inequality, thus facilitating a rigorous $L_2$-stability proof for the symmetric form of the discretized Navier-Stokes equations.
To our knowledge, this is the first family of high-order schemes that provide
both the pointwise positivity of thermodynamic variables and entropy stability for the 3-D compressible Navier-Stokes equations.

\section{The 3-D Navier-Stokes and Brenner-Navier-Stokes equations}
\label{BNS}
The 3-D compressible Navier-Stokes equations in curvilinear coordinates $(\xi_1,\xi_2,\xi_3)$ are written in conservation law form as follows:
\begin{equation}
\label{eq:NS_Curvilinear}
\begin{split}
&\frac{\partial J\U}{\partial t} + \sum\limits_{m, l=1}^{3}\frac{\partial}{\partial \xi_l}\left(\bfnc{F}_{\xil}-\bfnc{F}_{\xil}^{(v)}\right)= 0,\\
&\bfnc{F}_{\xil}\equiv \sum\limits_{m=1}^{3}\fnc{J}\frac{\partial \xi_l}{\partial \xm}\Fxm,\quad
\bfnc{F}_{\xil}^{(v)}\equiv\sum\limits_{m=1}^{3}\fnc{J}\frac{\partial \xi_l}{\partial \xm}\Fxm^{(v)},
\end{split}
\end{equation}
where $\Fxm$, and $\Fxmv$ are the inviscid and viscous fluxes associated with the Cartesian coordinates $(x_1, x_2, x_3)$, which are given by
\begin{equation*}
\Fxm = \left[\rho\Vm,\rho\Vm\Vone+\delta_{m,1}\fnc{P},\rho\Vm\Vtwo+\delta_{m,2}\fnc{P},\rho\Vm\Vthree+\delta_{m,3}\fnc{P},\rho\Vm\fnc{H}\right]\Tr, 
\end{equation*}
\begin{equation}\label{eq:Fv}
\Fxmv=\left[0,\tau_{1,m},\tau_{2,m},\tau_{3,m},\sum\limits_{i=1}^{3}\tau_{i,m}\fnc{U}_{i}-\kappa\frac{\partial \fnc{T}}{\partial\xm}\right]\Tr.
\end{equation} 
In the above equations, $\U = \left[\rho,\rho\Vone,\rho\Vtwo,\rho\Vthree,\rho\E\right]\Tr$ is a vector of the conservative variables,  $\fnc{P}$ is the pressure, and $\fnc{H}$ is the specific total enthalpy.
The viscous stresses in Eq.~(\ref{eq:Fv}) are given by
\begin{equation}\label{eq:tau}
\tau_{i,j} = \mu\left(\frac{\partial\fnc{V}_{i}}{\partial x_{j}}+\frac{\partial\fnc{V}_{j}}{\partial x_{i}}
-\delta_{i,j}\frac{2}{3}\sum\limits_{n=1}^{3}\frac{\partial\fnc{V}_{n}}{\partial x_{n}}\right),
\end{equation}
where $\mu(T)$ is the dynamic viscosity, $\kappa(T)$ is the thermal conductivity, and $\delta_{i,j}$ is the Kronecker delta.
The governing equations~(\ref{eq:NS_Curvilinear}) are subject to boundary conditions that are assumed to satisfy the entropy inequality. 

It is well known that there are no theoretical results on positivity of thermodynamic variables for the compressible Navier-Stokes equations. To overcome this problem, we regularize Eq.~\eqref{eq:NS_Curvilinear} by adding artificial dissipation 
in the form of the diffusion operator of the Brenner-Navier-Stokes equations  introduced in \cite{Brenner1}. The Brenner-Navier-Stokes equations 
are given by
\begin{equation}\label{eq:BNS_C}
\frac{\partial\U}{\partial t}
+
\sum\limits_{m=1}^{3}
\frac{\partial \Fxm}{\partial \xm} 
= 
\sum\limits_{m=1}^{3}\frac{\partial {\bm F}_{x_m}^{(B)}}{\partial \xm}
, \quad 
\forall \left(\xone,\xtwo,\xthree\right)\in\Omega,\quad t\ge 0,
\end{equation}
where $\sigma$ is the volume diffusivity and the viscous fluxes, 
${\bm F}_{x_m}^{(B)}, m=1,2,3$, are defined as
\begin{equation}
\label{eq:FB}
\begin{split}
&{\bm F}_{x_m}^{(B)}
=
\Fxmv 
+
\sigma 
\frac{\partial \rho}{\partial \xm}
\left[
\begin{array}{ccc}
1 & \bfnc{V} & \E 
\end{array}
\right]^\top
.
\end{split}
\end{equation} 

Despite that Eqs.~\eqref{eq:NS_Curvilinear} and \eqref{eq:BNS_C} are very similar to each other,  the Brenner-Navier-Stokes equations possess some remarkable properties that are not available for the Navier-Stokes equations. In contrast to Eq.~\eqref{eq:NS_Curvilinear},  the Brenner-Navier-Stokes equations guarantee existence of a weak solution and uniqueness of a strong solution, ensure global-in-time positivity of the density and temperature, satisfy a large class of entropy inequalities, and is compatible with a minimum entropy principle  \cite{FV, 
GUERPOP2014}.  
We rely on these properties of the Brenner-Navier-Stokes equations and regularize the Navier-Stokes equations by adding the following dissipation term to Eq.~\eqref{eq:NS_Curvilinear}: 
\begin{equation}
\label{rns3d}
\frac{\partial\U}{\partial t}
+
\sum\limits_{m=1}^{3}
\frac{\partial \Fxm}{\partial \xm} 
= 
\sum\limits_{m=1}^{3}
\left[
\frac{\partial {\bm F}_{x_m}^{(v)}}{\partial \xm}
+
\frac{\partial {\bm F}_{x_m}^{(AD)}}{\partial \xm}
\right],
\end{equation}
where the the artificial dissipation flux  ${\bm F}_{x_m}^{(AD)}$ can be obtained from the viscous flux of the Brenner-Navier-Stokes equations, ${\bm F}_{x_m}^{(B)}$, by setting $\mu=\mu^{AD}$, $\sigma = c_\rho \mu^{AD}/\rho$, and $\kappa = c_T\mu^{AD}$.
The coefficient $\mu^{AD}$ is an artificial viscosity and $c_T$ and $c_\rho$ are  positive tunable coefficients, which are set equal to $c_\rho = 0.9$ and $c_T = c_\rho \frac{c_{\fnc{P}}}{\gamma}$ for all test problems considered in this paper.

A necessary condition for selecting a unique, physically
relevant solution among possibly many weak solutions of Eqs.~\eqref{eq:NS_Curvilinear} and \eqref{rns3d}  is the entropy
inequality. Both the Navier-Stokes and regularized Navier-Stokes equations are equipped with the same convex scalar entropy function $\mathcal{S}= -\rho s$ and entropy flux $\mathcal{F}=-\rho s \bm{V}$, where $s$ is the thermodynamic entropy. 
It can be shown that the following inequality holds for Eqs. (\ref{eq:NS_Curvilinear}) and (\ref{rns3d}) assuming the corresponding boundary conditions are entropy stable (e.g., see \cite{GCLSS2019}):
\begin{equation}\label{eq:Eineq.3}
\int_{\Ohat} \frac{\partial (J \mathcal{S})}{\partial \tau}\mr{d}\Ohat =  \frac{d}{d \tau} \int_{\Ohat} J \mathcal{S}\mr{d}\Ohat  \leq 0.
\end{equation} 
Along with the entropy inequality given by Eq.~(\ref{eq:Eineq.3}), the regularized Navier-Stokes equations (Eq.~\eqref{rns3d}) preserve some other key properties of the Brenner-Navier-Stokes equations including the positivity of thermodynamic variables.
Herein, we propose to develop new numerical schemes that mimic these properties of the 3-D regularized Navier-Stokes equations at the discrete level.

\section{3-D high-order SBP operators}
\label{HIGHOPERATOR}
\subsection{High--order Diagonal-Norm Summation-by-parts Operators}
\label{SBP}
The derivatives in Eq.~\eqref{rns3d} are discretized by spectral collocation operators that satisfy 
the summation-by-parts (SBP) property \cite{CFNF}. 
In the one-dimensional setting, this mimetic property is achieved by approximating
 the first derivative with a discrete operator, $D$, and
using local mass $\mathcal{P}$ and stiffness $\mathcal{Q}$ matrices satisfying the following properties:
\begin{equation}
\label{SCO4}
\begin{split}
&
D=\mathcal{P}^{-1} \mathcal{Q}, \quad
\mathcal{P} = \mathcal{P}^\top,   \quad {\bf v}^\top \mathcal{P} {\bf v} > 0, \quad \forall {\bf v} \ne {\bf 0}, 
\\
&
\mathcal{Q}=B-\mathcal{Q}^\top, \quad B = {\rm diag}(-1,0,\dots,0,1) .
\end{split}
\end{equation}
Only diagonal-norm SBP operators are considered herein, which is critical for proving the entropy inequality at the discrete level \cite{CFNF}. 

In one spatial dimension, the physical domain is divided into $K$ non-overlapping elements $[x_1^k, x_{N_p}^k
]$, so that $x_1^k=x_{N_p}^{(k-1)}$. The discrete solution inside each element  is defined on 
Legendre-Gauss-Lobatto (LGL) points, ${\bf x}_k = \left[x_1^k, \dots, x_{N_p}^k \right]^\top$.  
These local points ${\bf x}_k$ are referred to as solution points. 
Along with the solution points, we also define a set of intermediate points $\bar{\bf x}_k =
\left[\bar{x}_0^k, \dots, \bar{x}_{N_p}^k \right]^\top$ prescribing bounding control volumes around 
each solution point. These points referred to as flux points form  
a complementary grid whose spacing is precisely equal to the diagonal elements of the positive
definite matrix $\mathcal{P}$ in Eq. (\ref{SCO4}), i.e., $\Delta \bar{\bf x} = \mathcal{P} {\bf 1}$,
where $ \bar{\bf x} = \left [\bar{x}_0, \dots, \bar{x}_{N_p} \right]^\top$ is a vector of flux points, 
${\bf 1}=[1, \dots, 1]^\top$, and $\Delta$ is an $N_p\times (N_p +1)$ matrix corresponding to the two-point backward
difference operator \cite{CFNF,YC3}.
As has been proven in \cite{FCNYS}, 
these discrete SBP derivative operators can be recast in the following telescopic flux form:
\begin{equation}
\label{TF2}
 \mathcal{P}^{-1} \mathcal{Q} {\bf f}  = \mathcal{P}^{-1} \Delta \bar{\bf f},
\end{equation}
where $\bar{\bf f}$ is a $p$th-order flux vector defined at the flux points. 

Using tensor product arithmetic, the above one-dimensional SBP operators naturally extend to two and three spatial dimensions.
The multidimensional tensor product operators are defined as
\begin{equation}
\begin{aligned}
&D_{\xi^1} = \left( D_N \otimes I_N \otimes  I_N \otimes I_5 \right), & 
&\mathcal{P}_{\xi^1} = \left( \mathcal{P}_N \otimes I_N \otimes  I_N  \otimes I_5 \right),
\\
&\mathcal{P}_{\xi^1,\xi^2} =   \left( \mathcal{P}_N \otimes \mathcal{P}_N  \otimes  I_N \otimes I_5 \right), & 
&\mathcal{P} =   \left( \mathcal{P}_N \otimes \mathcal{P}_N  \otimes  \mathcal{P}_N \otimes I_5 \right),
\\
&\widehat{\mathcal{P}} =   \left( \mathcal{P}_N \otimes \mathcal{P}_N  \otimes  \mathcal{P}_N  \right), & 
&\mathcal{P}_{\perp,\xi^1} =   \left( I_N \otimes \mathcal{P}_N  \otimes  \mathcal{P}_N \otimes I_5 \right),
\end{aligned}
\end{equation}
with similar definitions for other directions and operators $\mathcal{Q}_{\xi^i}$, $\Delta_{\xi^i}$ and $B_{\xi^i}$.  We also use the following notation $\mathcal{P}_{ijk} = \mathcal{P}_{i,i} \mathcal{P}_{j,j} \mathcal{P}_{k,k}$ and $\mathcal{P}_{ij} = \mathcal{P}_{i,i} \mathcal{P}_{j,j}$ where $\mathcal{P}_{i,i}$ is the scalar $i$th diagonal entry of $\mathcal{P}_N$.

\section{Baseline 3-D high-order spectral collocation schemes}
\label{SCRNS}
With the 3-D SBP operators discussed in the previous section, a baseline 3-D semi-discrete spectral collocation scheme of arbitrary order of accuracy 
for the 3-D Navier-Stokes equations (Eq.~\eqref{eq:NS_Curvilinear}) can be written as follows:
\begin{equation}
\label{semiDiscCurviBaseline} 
\hat{{\bf U}}_t + 
\sum\limits_{ l=1}^{3} 
\mathcal{P}^{-1}_{\xi^l} \Delta_{\xi^l}\hat{\bar{{\bf f}}}_l - 
D_{\xi^l}\hat{{\bf f}}^{(v)}_l = 
\sum\limits_{ l=1}^{3} \mathcal{P}^{-1}_{\xi^l} \hat{{\bf g}}^{(BC)}_l + \mathcal{P}^{-1}_{\xi^l} \hat{{\bf g}}^{(Int)}_l
\end{equation}  
where $\hat{{\bf U}} = \left[J \right]{\bf U}$,  $\hat{{\bf g}}^{(BC)}_l$  and $\hat{{\bf g}}^{(Int)}_l$ are boundary and interface penalty terms.

The contravariant inviscid fluxes, $\hat{\bar{{\bf f}}}_l$, defined at flux points are given by 
\begin{equation}
\begin{array}{l}
\label{PTH_ECFLUX} 
\hat{\bar{{\bf f}}}_m(\vec{\xi}_{\overline{i}}) = 
\sum\limits_{ j=i+1}^{N} \sum\limits_{ l=1}^{i} 2 q_{l,j}\bar{{ f}}_{(S)}({\bf U}(\vec{\xi}_{l}),{\bf U}(\vec{\xi}_{j})) \frac{\hat{\vec{{\bf a}}}^m(\vec{\xi}_{l}) +\hat{\vec{{\bf a}}}^m(\vec{\xi}_{j})}{2} \ \text{for} \, \, 1 \leq i \leq N-1, 
\\
\hat{\bar{{\bf f}}}_m(\vec{\xi}_{\overline{i}}) =  \bar{{ f}}_{(S)}({\bf U}(\vec{\xi}_{\overline{i}}),{\bf U}(\vec{\xi}_{\overline{i}})) \hat{\vec{{\bf a}}}^m(\vec{\xi}_{\overline{i}}) \ \text{for} \, \, \overline{i} \in \{0,N \},
\end{array}
\end{equation}
where $m=1, 2, 3$ and $\bar{{ f}}_{(S)}(\cdot,\cdot)$ is a two-point, consistent, entropy conservative  flux that satisfies 
\begin{equation}
\label{2PT_EC}
\left(w_1 - w_2 \right)^\top \bar{{f}}_{(S)}(U_1,U_2) =  \vec{\psi}_1 - \vec{\psi}_2  \quad 
\end{equation}
for any two admissible states $U_1$ and $U_2$ \cite{TAD2003}. For all test problems considered, we use the entropy conservative flux developed in \cite{Chand}.
In \cite{UY_3Dlow}, we show that the proposed method for ensuring positivity is  independent of a particular choice of $\bar{{ f}}_{(S)}$.

 In the above equation and hereafter, $\hat{\bf a}^l_m(\vec{\xi}_{ijk})$ is a $p$th-order discrete approximation of $J\frac{\partial \xi^l}{\partial x^m}$ at the solution point $\vec{\xi}_{ijk}$, which is constructed such that it satisfied the geometric conservation law (GCL) equations. These approximations are not unique and the specific formulas used for $\hat{\bf a}^l_m(\vec{\xi}_{ijk})$ in the present work can be found elsewhere (e.g. \citep{TLGCL}).

The contravariant viscous fluxes, $\hat{{\bf f}}^{(v)}_l$, are constructed as follows:
\begin{equation}
\label{physicalVISCFLUX}
\begin{aligned}
\hat{{\bf f}}^{(v)}_l &= \sum\limits_{ m=1}^{3} [\hat{a}^l_m] {\bf f}^{(v)}_{x^m}  
, \quad 
{\bf f}^{(v)}_{x^m} = \sum\limits_{ j=1}^{3} [c^{(v)}_{m,j}] {\bf\Theta}_{x^j}.
\end{aligned}
\end{equation}
For each $1 \leq m, j \leq 3$, $[c^{(v)}_{m,j}]$ is a block-diagonal matrix with $5 \times 5$ blocks, such that  $[\left( c^{(v)}_{m,j} \right)^T] = [c^{(v)}_{j,m}]$, and $\sum\limits_{ m=1}^{3}\sum\limits_{ j=1}^{3} {\bf v}^T[c^{(v)}_{m,j}] {\bf v} \geq 0, \forall {\bf v}$, i.e.,  the full viscous tensor is symmetric positive semi-definite (SPSD). 
 
The gradient of the entropy variables, ${\bf\Theta}_{x^j}$, is discretized by using an approach that closely resembles the local discontinuous Galerkin (LDG) method developed in \cite{CS}, which can be written as
\begin{equation}
\label{gradENTVARS}
\begin{array}{ll}
{\bf\Theta}_{x^j} &= \sum\limits_{ l=1}^{3} [\hat{a}^l_j] [J^{-1}] 
\left( D_{\xi^l} {\bf w} +  \mathcal{P}^{-1}_{\xi^l} \hat{\bf g}^{{\bf\Theta}}_{l}  \right)
\\
\hat{\bf g}^{{\bf\Theta}}_{1}(\vec{\xi}_{ijk}) 
&=
\frac{1}{2} \left(\delta_{1i}\Delta_1 {\bf w}(\vec{\xi}_{i-1jk})  + \delta_{Ni}\Delta_1 {\bf w}(\vec{\xi}_{ijk})  \right)
\\
\Delta_1 {\bf w}(\vec{\xi}_{ijk})  &=  {\bf w}(\vec{\xi}_{i+1jk}) - {\bf w}(\vec{\xi}_{ijk}),
\end{array}
\end{equation}
where $\delta_{ij}$ is the Kronecker delta. Note that similar discretizations are used in each computational direction. 

In \cite{CFNF}, it has been proven that the baseline high-order spectral collocation scheme given by Eqs.~(\ref{semiDiscCurviBaseline}--\ref{gradENTVARS}) satisfies the discrete entropy inequality if the corresponding boundary conditions are entropy stable. 
However, entropy stability alone does not guarantee the positivity of thermodynamic variables, if strong discontinuities are present in the domain. One of the key objectives of this paper is to modify the baseline scheme to address this shortcoming.

\section{Baseline 3-D spectral collocation scheme with high-order artificial dissipation}
\label{SCRNS_highBREN}

To control the amount of entropy production in regions where the discrete solution is under-resolved, we generalize the method developed in \cite{UY} to three spatial dimensions and add artificial dissipation in the form of the Brenner diffusion operator to the baseline high-order scheme (Eq.~\eqref{semiDiscCurviBaseline}) presented in the foregoing section. 
The baseline 3-D spectral collocation scheme with the high-order artificial dissipation is given by
\begin{equation}
\label{semiDiscCurviBaseline_highBREN}
\hat{{\bf U}}_t
+ 
\sum\limits_{ l=1}^{3} 
\mathcal{P}^{-1}_{\xi^l} \Delta_{\xi^l}\hat{\bar{{\bf f}}}_l 
- 
D_{\xi^l}
\left[
\hat{{\bf f}}^{(v)}_l 
+
\hat{{\bf f}}^{(AD_p)}_l 
\right]
= 
\sum\limits_{ l=1}^{3} \mathcal{P}^{-1}_{\xi^l} 
\left[ 
\hat{{\bf g}}_l^{(AD_p)}
+
\hat{{\bf g}}_l
\right],
\end{equation} 
where $\hat{{\bf g}}_l = \hat{{\bf g}}^{(BC)}_l + \hat{{\bf g}}^{(Int)}_l$.  The high-order artificial dissipation terms, $\hat{{\bf f}}^{(AD_p)}_l $ and $\hat{{\bf g}}_l^{(AD_p)}$, are discretized similarly to the viscous terms of Eq.~\eqref{semiDiscCurviBaseline}, as discussed in Section~\ref{SCRNS}.  The Brenner fluxes, $\hat{{\bf f}}^{(AD_p)}_l$, are constructed as follows:
\begin{equation}
\label{brennerVISCFLUX}
\begin{aligned}
\hat{{\bf f}}^{(AD_p)}_l &= \sum\limits_{ m=1}^{3} [\hat{a}^l_m] {\bf f}^{(AD_p)}_{x^m}  
, \quad 
{\bf f}^{(AD_p)}_{x^m} = \sum\limits_{ j=1}^{3} [c^{(B)}_{m,j}] {\bf\Theta}_{x^j},
\end{aligned}
\end{equation}
where $[c^{(B)}_{m,j}]$, $1 \leq m,j \leq 3$ are block--diagonal matrices with $5 \times 5$ blocks, $[\left( c^{(B)}_{m,j} \right)^T] = [c^{(B)}_{j,m}]$, and $\sum\limits_{ m=1}^{3}\sum\limits_{ j=1}^{3} {\bf v}^T[c^{(B)}_{m,j}] {\bf v} \geq 0, \forall {\bf v}$, i.e.,  the full artificial dissipation tensor is symmetric positive semi-definite (SPSD).  

To ensure consistency, maintain design-order accuracy for smooth resolved solutions, and control the amount of dissipation added in regions where the solution is under-resolved or discontinuous, we use the artificial viscosity, $\bfnc{\mu}^{AD}$, that is described in Section~\ref{AV}.  The mass and heat viscosity at each solution point are set as 
${\bm \sigma}(\vec{\xi}_{ijk}) = c_{\rho} \bfnc{\mu}^{AD}(\vec{\xi}_{ijk})/{\bm \rho}(\vec{\xi}_{ijk})$, 
and 
${\bm \kappa}(\vec{\xi}_{ijk}) = c_T\bfnc{\mu}^{AD}(\vec{\xi}_{ijk})$ (see Section~\ref{BNS}).

The high-order spectral collocation scheme given by Eq.~(\ref{semiDiscCurviBaseline_highBREN}) is conservative and stable in the entropy sense. The conservation follows immediately from the telescopic flux form of the inviscid terms and the SBP form of the viscous and artificial dissipation terms. The entropy stability of the discretized Navier-Stokes terms in Eq.~(\ref{semiDiscCurviBaseline_highBREN}) is proven in \cite{CFNF}. The entropy dissipation properties of the artificial dissipation terms follow immediately form Eq.~(\ref{brennerVISCFLUX}) and the symmetric positive semi-definiteness of  the artificial viscous tensor  $[c^{(B)}_{m,j}]$, $1 \leq m,j \leq 3$.

\section{Artificial Viscosity}
\label{AV}  
The scalar artificial viscosity, $\bfnc{\mu}^{AD}$, is used for both the high- and low-order artificial dissipation operators.  
Details on how the artificial viscosity coefficient is constructed are presented in \cite{UY_3Dlow}. Herein, we only briefly outline its key elements. The artificial viscosity coefficient is constructed based on the finite element residual of the entropy equation and the physical properties of the fluid. 
In the $k$-th grid element, $\bfnc{\mu}^{AD}$ is defined as follows:
$$
\bfnc{\mu}^{AD} = Sn^k \mu^k_{\max},
$$
where where $Sn$ is a sensor function ($0 \leq Sn \leq 1$) and $\mu^k_{\max}$ is the magnitude of the artificial viscosity in the $k$-th grid element.

To detect grid elements where the solution loses its regularity or becomes under-resolved, the sensor is constructed as a function of the finite element residual of the entropy equation, which is given by
\begin{equation}
\label{ENTROPYRESIDUALSENSOR} 
Sn^k = 
\left\{
\begin{array}{ll}
 Sn^k_0 , & {\rm if} \ Sn^k_0 \geq \max(0.2,\delta), \\
0,                                                                                                                                                                     & {\rm otherwise}
\end{array}
\right.
Sn^k_0 = \max({\bf r}^k)^{\max(1,\frac{p-1}{p-1.5})},
\end{equation}
where $p$ is the polynomial order and ${\bf r}(\vec{\xi}_{ijk})$ is a pointwise normalized entropy residual.
To take into account the physics of a problem, we also augment the entropy residual-based sensor with compression and pressure gradient sensors. These sensors are introduced to identify those regions where the amount of artificial viscosity can be reduced without sacrificing the solution accuracy. We refer  the reader to \cite{UY_3Dlow} for further details.

In each element, the upper bound of the artificial viscosity, $\mu^k_{\max}$, is set to be proportional to the maximum value of local velocity and
pressure jumps between neighboring solution points \cite{UY_3Dlow}. The result is that we minimize the amount of artificial dissipation
at contact discontinuities and make $\mu^k_{\max}$ proportional to the discontinuity strength, such that the velocity and pressure jumps act
as a limiter, if spurious oscillations are present in the solution.

The globally continuous artificial viscosity $\bfnc{\mu}_k^{AD}$ is then constructed by using the following smoothing procedure.
At each element vertex, we from a unique vertex viscosity coefficient, $\mu^{\rm ver}_i  = \max\limits_{k \in I_i}\mu^k_{\max}$,
where $I_i$ contains indices of all elements that share the $i$th grid vertex.  After that, 
the globally continuous artificial viscosity is obtained by using the tri-linear interpolation of 8 vertex viscosities, $\mu^{\rm ver}_i$, of the given hexahedral element.

\section{High-order positivity--preserving flux-limiting scheme}
\label{POSPRESLIM} 
Following an approach developed in \cite{UY_1Dhigh}, we construct a new high-order positivity--preserving flux-limiting scheme  for the 3-D Navier-Stokes equations by  combining the corresponding positivity-violating high-order spectral collocation scheme (Eq.~\eqref{semiDiscCurviBaseline}) and the first-order positivity--preserving finite volume scheme presented in the companion paper \cite{UY_3Dlow}. This methodology is presented next.  

\subsection{Positivity}
\label{LIMCONSPOS}
We first consider the 1st-order explicit Euler approximation of the time derivative term in Eq.~(\ref{eq:NS_Curvilinear}), so that on a given element 
$$
\begin{array}{l}
\hat{\bf U}^{n+1}_p = \hat{\bf U}^{n} + \tau \left(\frac{d {\bf \hat{\bf U}}}{dt}\right)_p, \\
\hat{\bf U}^{n+1}_1 = \hat{\bf U}^{n} + \tau \left(\frac{d {\bf \hat{\bf U}}}{dt}\right)_1,
\end{array}
$$
where $\hat{\bf U}^{n+1}_p = \left[J \right]{\bf U}^{n+1}_p$ and $\hat{\bf U}^{n+1}_1 = \left[J \right]{\bf U}^{n+1}_1$ are $p$th-order and first-order numerical solutions defined on the same Legendre-Gauss-Lobatto (LGL) grid elements with the same high-order metric terms.
In the above equation, $\hat{\bf U}^{n+1}_1$ is obtained by the first-order  positivity-preserving entropy stable scheme presented in \cite{UY_3Dlow}.
Therefore, at every $i$th solution point of each element $\text{IE}((\hat{\bf U}^{n+1}_1)_i) > 0$ and $(\rho^{n+1}_1)_i > 0$, where $\text{IE}((\hat{\bf U}^{n+1}_1)_i)$ is the internal energy associated with the 1st-order solution $(\hat{\bf U}^{n+1}_1)_i$.

To combine the 1st- and $p$th-order schemes, we use the flux-limiting technique developed in \cite{UY_1Dhigh}, which is in fact equivalent to limiting the low- and high-order solution vectors of the conservative variables:
\begin{equation}
\begin{array}{ll}
\label{LimSol}
\hat{\bf U}^{n+1}(\theta_f) &=  \hat{\bf U}^n + \tau \left[(1-\theta_f)\left(\frac{d\hat{\bf U}}{dt}\right)_1  +   \theta_f \left(\frac{d \hat{\bf U}}{dt}\right)_p\right]  \\
&= (1-\theta_f)\hat{\bf U}^{n+1}_1 + \theta_f \hat{\bf U}^{n+1}_p 
= \hat{\bf U}^{n+1}_1 + \theta_f [ \hat{\bf U}^{n+1}_p - \hat{\bf U}^{n+1}_1], 
\end{array} 
\end{equation}
where the flux limiter $\theta_f$, $0 \leq \theta_f \leq 1$,  is a constant on a given high-order element.

At each solution point, local lower bounds of density and internal energy are defined as follows:
\begin{equation}
\label{ALEPHUSED}
\epsilon^{\rho}_i = ({\rho}_1)^{n+1}_i \aleph,  \hspace{1cm} \epsilon^{\text{IE}}_i = \text{IE}((\hat{\bf U}_1)^{n+1}_i) \aleph,
\end{equation}
where $\aleph$ , $0 < \aleph < 1$, is a function that is bounded from below by a small positive number (e.g., $10^{-8}$), which approaches to its lower bound if the solution is smooth and goes to 1 if the solution loses its regularity.
In the present analysis, $\aleph$ is defined as follows: 
\begin{equation}
\label{ALEPHEXPLICIT}
\aleph^k =  \max(10^{-8},L^k),  \hspace{0.5cm}  L^k = Sn^k \max\limits_i\left(\frac{|\Delta P|}{2P_A}\right),
\end{equation}         
where $0 \leq Sn^k \leq 1$ is the residual-based sensor given by  Eq.~(\ref{ENTROPYRESIDUALSENSOR}) and $0 \leq \max\limits_i\left(\frac{|\Delta P|}{2P_A}\right) < 1$ is one half of the maximum relative two--point pressure jump (including jumps at the interfaces) on the $k$th element.
Note that $0 < \epsilon^{\rho}_i < ({\rho}_1)^{n+1}_i$ and 
$0 < \epsilon^{\text{IE}}_i < \text{IE}((\hat{\bf U}_1)^{n+1}_i)$ because $0 \leq L^k < 1$.

We now prove the following two lemmas.
\begin{lemma}
\label{EPSRHO}
For every $i$-th solution point, define a set 
$$
H^{\rho}_{i} = \{ \theta_f \in [0,1] \, \, | \, \, \rho^{n+1}_i(\theta_f) \geq \epsilon^{\rho}_i \}.
$$
Then, the set $H^{\rho}_{i}$ can be written as $H^{\rho}_{i} = [0,\theta^{\rho}_i]$ where $0 < \theta^{\rho}_i \leq 1$. Furthermore, the following statements hold: 1) if  $0 \leq \theta_f < \theta^{\rho}_i$, then $\rho^{n+1}_i(\theta_f) > \epsilon^{\rho}_i$ and 2) if $\theta^{\rho}_i < 1$, then $\rho^{n+1}_i(\theta^{\rho}_i) = \epsilon^{\rho}_i$.
\end{lemma}
\begin{proof}
This follows directly from the fact that $\rho^{n+1}_i(\theta_f)$ given by Eq.~(\ref{LimSol}) is a linear equation in the variable $\theta_f$ with $\rho^{n+1}_i(0) > \epsilon^{\rho}_i$.  
\end{proof}

\begin{lemma}
\label{EPSIE}
For every $i$-th solution point, define a set
$$
H^{\text{IE}}_{i} = \{ \theta_f \in H^{\rho}_{i} 
\, \, | \, \,
\text{IE}(\hat{\bf U}^{n+1}_i(\theta_f)) \geq \epsilon^{\text{IE}}_i \},
$$
where $H^{\rho}_{i} = [0,\theta^{\rho}_i]$ is defined in Lemma~\ref{EPSRHO}.
Then, the set $H^{\text{IE}}_{i}$ can be written as $H^{\text{IE}}_{i} = [0,\theta^{\text{IE}}_i]$ where $0 < \theta^{\text{IE}}_i \leq \theta^{\rho}_i$.
Furthermore, the following statements hold: 1) if $0 \leq \theta_f < \theta^{\text{IE}}_i$, then $\text{IE}(\hat{\bf U}^{n+1}_i(\theta_f)) > \epsilon^{\text{IE}}_i$ and 2) if $\theta^{\text{IE}}_i < \theta^{\rho}_i$, then $\text{IE}(\hat{\bf U}^{n+1}_i(\theta^{\text{IE}}_i)) = \epsilon^{\text{IE}}_i$.
\end{lemma}
\begin{proof}
For each $i$th solution point, if $\text{IE}(\hat{\bf U}^{n+1}_i(\theta^{\rho}_i)) \geq \epsilon^{\text{IE}}_i$, then we set $\theta^{\text{IE}}_i = \theta^{\rho}_i$.  Assume that there is a solution point such that $\text{IE}(\hat{\bf U}^{n+1}_i(\theta^{\rho}_i)) < \epsilon^{\text{IE}}_i$.  Since $\rho^{n+1}_i(\theta_f) \geq \epsilon^{\rho}_i > 0$ $\forall \theta_f \in [0, \theta^{\rho}_i]$, $\text{IE}(\hat{\bf U}^{n+1}_i(\theta_f))$ is a continuous function with respect to $\theta_f$ for $\theta_f \in [0, \theta^{\rho}_i]$.  Taking into account that $\text{IE}(\hat{\bf U}^{n+1}_i(0)) = \text{IE}((\hat{\bf U}^{n+1}_1)_i) > \epsilon^{\text{IE}}_i$ and  $\text{IE}(\hat{\bf U}^{n+1}_i(\theta^{\rho}_i)) < \epsilon^{\text{IE}}_i$, it follows by the intermediate value theorem that there exists  $\theta^*_i \in (0, \theta^{\rho}_i)$ such that $\text{IE}(\hat{\bf U}^{n+1}_i(\theta^*_i)) = \epsilon^{\text{IE}}_i$.  Let $\theta^{\text{IE}}_i = \theta^*_i$ (note that there is only one $\theta^*_i \in (0, \theta^{\rho}_i)$ such that $\text{IE}(\hat{\bf U}^{n+1}_i(\theta^*_i)) = \epsilon^{\text{IE}}_i$).
Now we show that for  all $0 \leq \theta_f < \theta^{\text{IE}}_i$, we have $\text{IE}(\hat{\bf U}^{n+1}_i(\theta_f)) > \epsilon^{\text{IE}}_i$.  By definition of $\epsilon^{\text{IE}}_i$, $\text{IE}(\hat{\bf U}^{n+1}_i(0)) > \epsilon^{\text{IE}}_i$.  For $0 < \theta_f < \theta^{\text{IE}}_i$, we have
\begin{align}
\label{IE_convex}
\hat{\bf U}^{n+1}_i(\theta_f) &= (1-\theta_f)(\hat{\bf U}_1)^{n+1}_i + \theta_f (\hat{\bf U}_p)^{n+1}_i \\
 &= \frac{\theta_f}{\theta^{\text{IE}}_i}\biggl[\theta^{\text{IE}}_i \left((\hat{\bf U}_p)^{n+1}_i-(\hat{\bf U}_1)^{n+1}_i\right) + (\hat{\bf U}_1)^{n+1}_i\biggr]                  
 + \left(1-\frac{\theta_f}{\theta^{\text{IE}}_i}\right)(\hat{\bf U}_1)^{n+1}_i \nonumber \\ 
 &= 
 \frac{\theta_f}{\theta^{\text{IE}}_i} \hat{\bf U}^{n+1}_i(\theta^{\text{IE}}_i) 
 + 
 \left(1-\frac{\theta_f}{\theta^{\text{IE}}_i}\right)(\hat{\bf U}_1)^{n+1}_i \nonumber .
\end{align} 
Hence, due to the concavity of internal energy 
\begin{equation}
\begin{array}{ll}
\label{IE_convex_cocave}
\text{IE}(\hat{\bf U}^{n+1}_i(\theta_f)) 
&\geq
  \frac{\theta_f}{\theta^{\text{IE}}_i} \text{IE}(\hat{\bf U}^{n+1}_i(\theta^{\text{IE}}_i)) 
  + 
  \left(1-\frac{\theta_f}{\theta^{\text{IE}}_i}\right)\text{IE}((\hat{\bf U}_1)^{n+1}_i) 
  \\
&>
  \frac{\theta_f}{\theta^{\text{IE}}_i} \epsilon^{\text{IE}}_i
  + 
  \left(1-\frac{\theta_f}{\theta^{\text{IE}}_i}\right)\epsilon^{\text{IE}}_i
 =
  \epsilon^{\text{IE}}_i   .
\end{array}
\end{equation} 
\end{proof}
\begin{remark}
Note that $\theta^{\text{IE}}_i$ in Lemma~\ref{EPSIE} can readily be found by solving the quadratic equation for internal energy, which is 
analogous to the one presented in the companion paper \cite{UY_3Dlow}.
\end{remark}
For a given element, we define $\theta_{\text{IE}} = \text{min}_i\{\theta^{\text{IE}}_i\} > 0$.  By construction, $\text{IE}(\hat{\bf U}^{n+1}_i(\theta_{\text{IE}})) \geq \epsilon^{\text{IE}}_i$  and 
$\rho(\hat{\bf U}^{n+1}_i(\theta_{\text{IE}})) \geq \epsilon^{\rho}_i$ for every solution point on the element.  The solution at the $(n+1)$th time level is set equal to $\hat{\bf U}^{n+1}(\theta_{\text{IE}})$, which preserves the pointwise positivity of both density and internal energy. 
\begin{remark}
The above limiting is not immediately conservative for general $\hat{\bf U}^{n+1}_1$ and $\hat{\bf U}^{n+1}_p$.  We refer the reader to Section~\ref{SEMIDISC_LIMSCHEME} which presents an implementation  of this limiting procedure in a way that preserves conservation.
\end{remark}

\subsection{Design order of accuracy}
\label{LIMACCURACY}
In this section, we prove that the proposed limiting scheme is design-order accurate for smooth solutions and sufficient grid resolutions.  Without loss of generality, we assume that the grid resolution depends on a single parameter $0 < h^x \leq 1$, such that all element edges are directly proportional to $h^x$.
In this section, $\| \cdot \|$ denotes the Euclidean norm.  
Let  $\hat{\bf U}^{\text{ex}}_i(t_{n+1})$ be the smooth exact solution at the $i$th solution point at $t = t_{n+1}$. 
For each solution point, we define a local admissible set
$$
\mathcal{A}^{\epsilon}_i = \{ {\bm u}_i 
= 
\left[
\begin{array}{ccc}
\rho & \rho \vec{\bfnc{V}} & \rho E
\end{array}
\right]^\top
 \ \left| \right. \ \text{IE}({\bm u}_i) \geq \epsilon^{\text{IE}}_i \ , \rho_i \geq \epsilon^{\rho}_i \}
$$      
and assume that $\hat{\bf U}^{\text{ex}}_i(t_{n+1}) \in \mathcal{A}^{\epsilon}_i$.  Note that  $\epsilon^{\text{IE}}_i$ and $\epsilon^{\rho}_i$ are positive user-defined parameters that can be made arbitrarily small by selecting a sufficiently small value of the parameter $\aleph$ for a given element.  In the present analysis, $\aleph$, which  is given by Eq.~(\ref{ALEPHEXPLICIT}), is set such that it becomes smaller when the regularity of the numerical solution increases. 
  We also assume that the solution is sufficiently smooth, so that 
$
\| (\hat{\bf U}^{n+1}_1)_i - (\hat{\bf U}^{n+1}_p)_i  \| \le \| (\hat{\bf U}^{n+1}_1)_i - \hat{\bf U}^{\text{ex}}_i(t_{n+1}) \| 
+ \| \hat{\bf U}^{\text{ex}}_i(t_{n+1}) - (\hat{\bf U}^{n+1}_p)_i  \| 
=
 \mathcal{O}(h^x)$, as $h^x \rightarrow 0$.  

Let us show that $\| \hat{\bf U}_i^{n+1}(\theta_{\text{IE}}) - \hat{\bf U}^{\text{ex}}_i(t_{n+1})\| =  \mathcal{O}((h^x)^p)$ for all solution points. If $\theta_i^{\text{IE}} = 1 \, \, \forall i$ on a given element, then 
$\theta_{\text{IE}} = \text{min}_i\{\theta^{\text{IE}}_i\}=1$,
$\hat{\bf U}^{n+1}(\theta_{\text{IE}})= \hat{\bf U}_p^{n+1}$ and the result follows.

We now assume that $\theta_{\text{IE}} < 1$. 
In this case, to prove the consistency of the limiting procedure, it is sufficient to show that $1- \theta_{\text{IE}} =  \mathcal{O}((h^x)^{p-1})$.
Indeed, if $1- \theta_{\text{IE}} =  \mathcal{O}((h^x)^{p-1})$, then for every solution point  we have
\begin{equation}
\label{p-order}
\begin{array}{ll}
\| 
\hat{\bf U}_i^{n+1}(\theta_{\text{IE}}) - \hat{\bf U}^{\text{ex}}_i(t_{n+1})
\| 
&
\le  
(1-\theta_{\text{IE}}) \| (\hat{\bf U}_1)^{n+1}_i - \hat{\bf U}^{\text{ex}}_i(t_{n+1})
\|
\\
&
\quad \quad \ \  +
\theta_{\text{IE}} 
\| 
(\hat{\bf U}_p)^{n+1}_i - \hat{\bf U}^{\text{ex}}_i(t_{n+1})
\| 
\\
&
= (1-\theta_{\text{IE}}) O(h^x) + \theta_{\text{IE}} O((h^x)^p) = O((h^x)^p).
\end{array}
\end{equation}
 To prove that $1- \theta_{\text{IE}} =1- \min\limits_i\{\theta^{\text{IE}}_i\} =  \mathcal{O}((h^x)^{p-1})$, it is sufficient to show that if 
$\theta^{\text{IE}}_i < 1$
(which is only possible if
$(\hat{\bf U}_p)^{n+1}_i \not\in \mathcal{A}^{\epsilon}_i$),
then
$1 - \theta^{\text{IE}}_i =  \mathcal{O}((h^x)^{p-1}) \, \, \forall i$.
Assume that at the $i$th solution point $\theta^{\text{IE}}_i < 1$.
Since $\theta^{\text{IE}}_i \leq \theta^{\rho}_i$, we only have to consider the following two cases: 
1)  $\theta^{\text{IE}}_i = \theta^{\rho}_i$ and $\theta^{\rho}_i < 1$, 2) $\theta^{\text{IE}}_i < \theta^{\rho}_i$.

 Case 1.  For $0<\theta^{\rho}_i < 1$, the following inequalities hold $(\rho_p)^{n+1}_i < \epsilon_i^{\rho}  \leq \rho^{\text{ex}}_i(t_{n+1})$, which lead to  $(\rho_p)^{n+1}_i = \epsilon_i^{\rho} + \mathcal{O}((h^x)^p)$.  From Lemma~\ref{EPSRHO} it follows that  $\theta^{\rho}_i$ satisfies
\begin{equation}
\label{rhoTheta}
\rho_i^{n+1}(\theta^{\rho}_i) 
= (\rho_1)^{n+1}_i + \theta^{\rho}_i((\rho_p)^{n+1}_i -(\rho_1)^{n+1}_i )  = \epsilon_i^{\rho}.
\end{equation} 
Thus,
\begin{equation}
\label{theta_rho_orderpm1}
1- \theta^{\rho}_i  
= 
\frac{\epsilon_i^{\rho} - (\rho_p)^{n+1}_i}{(\rho_1)^{n+1}_i-(\rho_p)^{n+1}_i}
=
\frac{\mathcal{O}((h^x)^p)}{\mathcal{O}(h^x)}
=
\mathcal{O}((h^x)^{p-1}).
\end{equation} 
Taking into account that 
$\theta^{\text{IE}}_i = \theta^{\rho}_i$, we also have
$1- \theta^{\text{IE}}_i = \mathcal{O}((h^x)^{p-1})$.
Using $0 < \theta^{\rho}_i < 1$ and Eq.~(\ref{theta_rho_orderpm1}) yield
\begin{equation}
\label{uTildRho_ordP}
\begin{split}
\| \hat{\bf U}_i^{n+1}(\theta^{\rho}_i) - \hat{\bf U}^{\text{ex}}_i(t_{n+1})\| 
& \le
\| \hat{\bf U}_i^{n+1}(\theta^{\rho}_i) - (\hat{\bf U}_p)^{n+1}_i \| 
+
\| (\hat{\bf U}_p)^{n+1}_i - \hat{\bf U}^{\text{ex}}_i(t_{n+1}) \| 
\\
&=
(1- \theta^{\rho}_i) \|( (\hat{\bf U}_1)^{n+1}_i - (\hat{\bf U}_p)^{n+1}_i \| 
+
  \mathcal{O}((h^x)^p)
\\
&=
\mathcal{O}((h^x)^p).
\end{split}
\end{equation}

Case 2. We now assume that  $\theta^{\text{IE}}_i < \theta^{\rho}_i$. Note that the internal energy $\text{IE}(\hat{\bf U}_i^{n+1}(\theta^{\rho}_i))$ is defined at $i$, because ${\rho}_i^{n+1}(\theta^{\rho}_i) \ge \epsilon^{\rho}_i > 0$. 
As in Case 1,  Eq.~\eqref{uTildRho_ordP} holds, because $\theta^{\rho}_i < 1$.  Furthermore, if $\theta^{\rho}_i = 1$, then $\hat{\bf U}_i^{n+1}(\theta^{\rho}_i) =(\hat{\bf U}_p)^{n+1}_i $, which again implies that Eq.~\eqref{uTildRho_ordP} holds.
 Using  Eq.~\eqref{uTildRho_ordP} yields
 \begin{equation}
 \label{IEordp}
\begin{split}
\text{IE}(\hat{\bf U}_i^{n+1}(\theta^{\rho}_i)) 
&=
\rho_i^{n+1}(\theta^{\rho}_i)
\E_i^{n+1}(\theta^{\rho}_i)
-
\frac{\rho_i^{n+1}(\theta^{\rho}_i)}{2}
\| \vec{\bfnc{V}}_i^{n+1}(\theta^{\rho}_i)  \|^2
\\
&=
\text{IE}(\hat{\bf U}^{\text{ex}}_i(t_{n+1})) + \mathcal{O}((h^x)^p),
\end{split}
\end{equation}
where $\E_i^{n+1}(\theta^{\rho}_i)$ is the specific total energy of $\hat{\bf U}_i^{n+1}(\theta^{\rho}_i)$.
Since
$\theta^{\text{IE}}_i < \theta^{\rho}_i$,
$\text{IE}(\hat{\bf U}_i^{n+1}(\theta^{\rho}_i)) 
< 
\epsilon^{\text{IE}}_i
\leq 
\text{IE}(\hat{\bf U}^{\text{ex}}_i(t_{n+1})).
$
Therefore, from Eq.~\eqref{IEordp} it follows that  
$
\text{IE}(\hat{\bf U}_i^{n+1}(\theta^{\rho}_i)) 
= 
\epsilon^{\text{IE}}_i + \mathcal{O}((h^x)^p)
$.
Using Eq.~\eqref{IE_convex} for $0<\theta < \theta^{\rho}_i$, we have
\begin{equation}
\hat{\bf U}^{n+1}_i(\theta) 
= \frac{\theta}{\theta^{\rho}_i} \hat{\bf U}^{n+1}_i(\theta^{\rho}_i) 
+ \left(1-\frac{\theta}{\theta^{\rho}_i}\right)(\hat{\bf U}_1)^{n+1}_i.
\end{equation} 
Again, $\hat{\bf U}^{n+1}_i(\theta)$ may have non-positive internal energy, but it has positive density.  Hence, for all $\theta \in (0, \theta^{\rho}_i)$, $\text{IE}(\hat{\bf U}^{n+1}_i(\theta))$ is defined at $i$ and the following bound holds:
\begin{equation}
\label{concave_IE_L}
\begin{array}{ll}
\text{IE}(\hat{\bf U}^{n+1}_i(\theta))
&=
\frac{\theta}{\theta^{\rho}_i} 
\text{IE}(\hat{\bf U}^{n+1}_i(\theta^{\rho}_i))
+
\left(1-\frac{\theta}{\theta^{\rho}_i}\right)
\text{IE}((\hat{\bf U}^{n+1}_1)_i)
\\
&+
\frac{
\rho^{n+1}_i(\theta^{\rho}_i) ({\rho}^{n+1}_1)_i
\left \|   
(\vec{\bfnc{V}}^{n+1}_1)_i - \vec{\bfnc{V}}^{n+1}_i(\theta^{\rho}_i) 
 \right \|^2
\frac{\theta}{\theta^{\rho}_i}\left(1-\frac{\theta}{\theta^{\rho}_i}\right)
}
{2\rho^{n+1}_i(\theta)}
\\
&
\geq
\frac{\theta}{\theta^{\rho}_i} 
\text{IE}(\hat{\bf U}^{n+1}_i(\theta^{\rho}_i))
+
\left(1-\frac{\theta}{\theta^{\rho}_i}\right)
\text{IE}((\hat{\bf U}^{n+1}_1)_i)
.
\end{array}
\end{equation}
 Note that there exists a unique 
$\theta^*_i \in (0,\theta^{\rho}_i)$ such that
\begin{equation}
\label{IEstar}
\frac{\theta^*_i}{\theta^{\rho}_i} 
\text{IE}(\hat{\bf U}^{n+1}_i(\theta^{\rho}_i))
+ \left(1-\frac{\theta^*_i}{\theta^{\rho}_i}\right)
\text{IE}((\hat{\bf U}_1)^{n+1}_i)
= \epsilon^{\text{IE}}_i.
\end{equation}
From Eq.~\eqref{concave_IE_L} it follows that
$\text{IE}(\hat{\bf U}^{n+1}_i(\theta^*_i)) \geq \epsilon^{\text{IE}}_i$ and according to Lemma~\ref{EPSIE}, 
$\theta^*_i \leq \theta^{\text{IE}}_i$.  
Using Eq.~\eqref{IEstar} and Eq.~\eqref{IEordp}, we have
\begin{equation}
\label{THETAstar}
1 - \frac{\theta^*_i}{\theta^{\rho}_i}
=
\frac{
\epsilon^{\text{IE}}_i - \text{IE}(\hat{\bf U}^{n+1}_i(\theta^{\rho}_i))
}
{
\text{IE}((\hat{\bf U}_1)^{n+1}_i) - \text{IE}(\hat{\bf U}^{n+1}_i(\theta^{\rho}_i))
}
=
\frac{\mathcal{O}((h^x)^p)
}
{\mathcal{O}((h^x))}
=\mathcal{O}((h^x)^{p-1}).
\end{equation}
Equations~\eqref{theta_rho_orderpm1} and (\ref{THETAstar}) yield
$
1 - \theta^*_i = \mathcal{O}((h^x)^{p-1})
$. 
Since $\theta^*_i \leq \theta^{\text{IE}}_i < 1$, it follows that $1-\theta^{\text{IE}}_i = \mathcal{O}((h^x)^{p-1}) \ \forall i$ and Eq.~(\ref{p-order}) holds.

\subsection{High-order positivity-preserving flux-limiting scheme}
\label{SEMIDISC_LIMSCHEME}    

We now present the semi-discrete form of the high-order positivity-preserving flux limiting scheme which is given by 
\begin{equation}  
\label{EQN_LIMSCHEME} 
\begin{array}{ll}
\frac{d {\bf \hat{\bf U}}}{dt} &= 
\theta^k_f    \left(\frac{d {\bf \hat{\bf U}}}{dt}\right)_p 
+ 
(1-\theta^k_f)\left(\frac{d {\bf \hat{\bf U}}}{dt}\right)_1
+
\left(\frac{d {\bf \hat{\bf U}}}{dt}\right)_{AD},
\\
\left(\frac{d {\bf \hat{\bf U}}}{dt}\right)_p           
&=  
\sum\limits_{ l=1}^{3} 
-
\mathcal{P}^{-1}_{\xi^l} \Delta_{\xi^l}\hat{\bar{{\bf f}}}_l 
+ 
D_{\xi^l}\hat{{\bf f}}^{(v)}_l 
+ 
\mathcal{P}^{-1}_{\xi^l} \hat{{\bf g}}_l ,
\\
\left(\frac{d {\bf \hat{\bf U}}}{dt}\right)_1 
&=  
\sum\limits_{ l=1}^{3} 
-
\mathcal{P}^{-1}_{\xi^l} \Delta_{\xi^l}
\hat{\bar{{\bf f}}}^{(MR)}_l 
+
D_{\xi^l}\hat{{\bf f}}^{(v)}_l 
+ 
\mathcal{P}^{-1}_{\xi^l} 
\hat{{\bf g}}_l ,
\\
\left(\frac{d {\bf \hat{\bf U}}}{dt}\right)_{AD} 
&= 
\sum\limits_{ l=1}^{3}
\mathcal{P}^{-1}_{\xi^l} \Delta_{\xi^l}
\left[	
(1-\theta_f^k)
\hat{\bar{{\bf f}}}^{(AD_1)}_{\hat{\bar{\sigma}},l}
+
\hat{\bar{{\bf f}}}^{(AD_1)}_l
\right]
+
D_{\xi^l}\hat{{\bf f}}^{(AD_p)}_l
\\
&
\quad	\quad \quad	\quad
+ 
\mathcal{P}^{-1}_{\xi^l}
\left[
\hat{{\bf g}}_l^{(AD_1)}
+ 
\hat{{\bf g}}_l^{(AD_p)}
\right],
\end{array}
\end{equation} 
where the flux limiter $\theta^k_f $ $(0 \le \theta^k_f \le 1)$ is a constant computed independently in each element and $\hat{\bar{{\bf f}}}^{(MR)}_l $ is the first-order Merriam-Roe entropy dissipative flux \cite{MER, UY_3Dlow}.  Note that the flux limiting is only applied to the inviscid terms and the mass diffusion term required for positivity of density.  The term $\left(\frac{d {\bf \hat{\bf U}}}{dt}\right)_p $ is the baseline high-order scheme with no artificial dissipation, where $\hat{{\bf g}}_l$ represents both the inviscid and viscous penalties (see Section~\ref{SCRNS}).  The $AD_p$ terms are discussed in Section~\ref{SCRNS_highBREN} and the remaining first-order terms are presented in the companion paper \cite{UY_3Dlow}.

\subsection{Conservation}
\label{Cons}
Since $ \theta^k_f $ is computed independently on each element, it is not immediately clear that the scheme given by Eq.~\eqref{EQN_LIMSCHEME} is conservative  for all $0 \leq  \theta^k_f \leq 1 $.   Let us show that the scheme is indeed conservative.
\begin{theorem}
\label{SEMIDISC_LIMSCHEME_CONSERVATIVE}
The high-order positivity--preserving flux-limiting scheme given by Eq.~\eqref{EQN_LIMSCHEME} is conservative for all $0 \leq  \theta^k_f \leq 1 $.
\end{theorem}
\begin{proof}
Collecting like terms in Eq.~\eqref{EQN_LIMSCHEME} shows that $\theta^k_f$ only affects the values of
$\hat{\bar{{\bf f}}}_l$,   $\hat{\bar{{\bf f}}}^{(MR)}_l$, and  $\hat{\bar{{\bf f}}}^{(AD_1)}_{\hat{\bar{\sigma}},l}$
on the $k$th element.  The artificial dissipation flux $\hat{\bar{{\bf f}}}^{(AD_1)}_{\hat{\bar{\sigma}},l}$ is only defined at the interior flux points, which immediately implies that the corresponding telescopic flux differencing term is globally conservative.  Therefore, we have
\begin{equation}
\begin{array}{ll} 
&
\sum\limits_{ l=1}^{3}  
{\bf 1}_1^\top
\mathcal{P}
\mathcal{P}^{-1}_{\xi^l}\Delta_{\xi^l}
\left[   
\theta^k_f 
\hat{\bar{{\bf f}}}_l
+
(1-\theta^k_f)
\hat{\bar{{\bf f}}}^{(MR)}_l
\right]
\\
&=
\sum\limits_{ j,k=1}^{N}
\mathcal{P}_{jk}
\sum\limits_{ i=1}^{N}
\left[   
\theta^k_f 
\left( 
\hat{\bar{{\bf f}}}_1(\vec{\xi}_{\bar{i}jk})
-
\hat{\bar{{\bf f}}}_1(\vec{\xi}_{\bar{i}-1jk})
\right)  \right.
\\
&  \quad \quad   \quad \quad  \left.
+
(1-\theta^k_f)
\left(
\hat{\bar{{\bf f}}}^{(MR)}_1(\vec{\xi}_{\bar{i}jk})
-
\hat{\bar{{\bf f}}}^{(MR)}_1(\vec{\xi}_{\bar{i}-1jk})
\right)
\right]  + \cdots
\\
&=
\sum\limits_{ j,k=1}^{N}
\mathcal{P}_{jk}
\left[   
\theta^k_f 
\left( 
\hat{\bar{{\bf f}}}_1(\vec{\xi}_{Njk})
-
\hat{\bar{{\bf f}}}_1(\vec{\xi}_{1jk})
\right)  \right.
\\
&  \quad \quad   \quad \quad  \left.
+
(1-\theta^k_f)
\left(
\hat{\bar{{\bf f}}}^{(MR)}_1(\vec{\xi}_{Njk})
-
\hat{\bar{{\bf f}}}^{(MR)}_1(\vec{\xi}_{1jk})
\right)
\right]  + \cdots
\\
&=
\sum\limits_{ j,k=1}^{N}
\mathcal{P}_{jk}
\left[   
\hat{\bar{{\bf f}}}_1(\vec{\xi}_{Njk})
-
\hat{\bar{{\bf f}}}_1(\vec{\xi}_{1jk})
\right]  
\\
&   
+ 
\sum\limits_{ i,k=1}^{N}
\mathcal{P}_{ik} 
\left[   
\hat{\bar{{\bf f}}}_2(\vec{\xi}_{iNk})
-
\hat{\bar{{\bf f}}}_2(\vec{\xi}_{i1k})
\right]
+ 
\sum\limits_{ i,j=1}^{N}
\mathcal{P}_{ij}
\left[   
\hat{\bar{{\bf f}}}_3(\vec{\xi}_{ijN})
-
\hat{\bar{{\bf f}}}_3(\vec{\xi}_{ij1})
\right]
\\
&
=
\sum\limits_{ l=1}^{3}  
{\bf 1}_1^\top
\mathcal{P}
\mathcal{P}^{-1}_{\xi^l}\Delta_{\xi^l}
\hat{\bar{{\bf f}}}_l.
\end{array}
\end{equation}
Hence, conservation of the flux-limiting scheme given by Eq.~\eqref{EQN_LIMSCHEME} follows directly from conservation of the baseline scheme given by Eq.~\eqref{semiDiscCurviBaseline}.
\end{proof}

\subsection{Artificial viscosity for the flux-limiting scheme}
\label{VISC_LIMSCHEME}   

We now present how the artificial viscosity is constructed for the flux-limiting scheme given by Eq.~\eqref{EQN_LIMSCHEME}.
If an element is flagged for flux limiting (i.e, $\theta^k_f < 1$), then only the first-order dissipation is used for this element, even if it is later determined that $\theta^k_f = 1$.  Any element flagged for  flux limiting is herein referred to as `` a limited element."

The artificial viscosity coefficient $\bfnc{\mu}^{AD}$ presented in  Chapter~\ref{AV} is used to construct the first-order Brenner dissipation defined at the flux points, $\bar{\bfnc{\mu}}^{AD}_1$, and the $p$th-order Brenner dissipation calculated at the solution points, $\bfnc{\mu}^{AD}_p$.  Let $ V^k_1,V^k_2, \ldots, V^k_8$ be 8 vertices of the $k$th hexahedral element. Define an indicator function, $\chi(\cdot)$, such that $\chi(V^k_l) = 1$ if $V^k_l$ is collocated with any limited elements, otherwise,  $\chi(V^k_l) = 0$.  Then, set $\bfnc{\mu}^{AD}_p(V^k_l) = \bfnc{\mu}^{AD}(V^k_l)(1-\chi(V^k_l))$ and use the tri-linear interpolation to obtain $\bfnc{\mu}^{AD}_p$ at the remaining solution points.   Note that for elements flagged for limiting,  $\bfnc{\mu}^{AD}_p = 0$ and only the first-order Brenner artificial dissipation is used.  For all fixed $1 \leq j,k \leq N$, the first-order dissipation at $\vec{\xi}_{i} = \vec{\xi}_{ijk}$ is formed as follows:
\begin{equation}
\begin{array}{l}
\bar{\bfnc{\mu}}^{AD}_1(\vec{\xi}_{\bar{i}}) 
=
\frac 12 \left(
\bfnc{\mu}^{AD}(\vec{\xi}_{i})  + \bfnc{\mu}^{AD}(\vec{\xi}_{i+1})
- \bfnc{\mu}^{AD}_p(\vec{\xi}_{i})  - \bfnc{\mu}^{AD}_p(\vec{\xi}_{i+1}) \right), \, 1 \leq i \leq N-1,
\\
\bar{\bfnc{\mu}}^{AD}_1(\vec{\xi}_{\bar{0}}) = 
\bfnc{\mu}^{AD}(\vec{\xi}_1)
-
\bfnc{\mu}^{AD}_p(\vec{\xi}_1)
,
\quad
\bar{\bfnc{\mu}}^{AD}_1(\vec{\xi}_{\bar{N}}) = 
\bfnc{\mu}^{AD}(\vec{\xi}_N)
-
\bfnc{\mu}^{AD}_p(\vec{\xi}_N).
\end{array}
\end{equation}
Identical formulas are used for the other spatial directions.
For the first-order artificial dissipation, the $c_{\rho}$ and $c_T$ coefficients are set equal to those of the $p$th-order counterpart (see Section~\ref{BNS}).  
  
The mass diffusion coefficient for the first-order artificial dissipation flux is set to be proportional to $\bar{\bfnc{\sigma}}^{AD}_1$.
Thus, for all fixed $1 \leq j,k \leq N$, the mass diffusion at $\vec{\xi}_{i} = \vec{\xi}_{ijk}$ is given by
\begin{equation}
\begin{array}{l}
\label{sigma_LIM}
\bar{\bfnc{\sigma}}^{AD}_1(\vec{\xi}_{\bar{i}}) = 
\max
\left(
\chi(\vec{\xi}_{\bar{i}})
\left( \delta_{0,i} +  \delta_{N,i} \right)
\bar{\bfnc{\sigma}}_{\min}(\vec{\xi}_{\bar{i}})
,
c_{\rho}
\frac{
\bar{\bfnc{\mu}}^{AD}_1(\vec{\xi}_{\bar{i}}) 
}
{
\sqrt{\bfnc{\rho}(\vec{\xi}_{i}) \bfnc{\rho}(\vec{\xi}_{i+1})}
}
\right),  \text{for} \, \, 0 \leq i \leq N,
\end{array}
\end{equation}
where $\bfnc{\rho}(\vec{\xi}_{0})$ and $\bfnc{\rho}(\vec{\xi}_{N+1})$ are densities at element interfaces or physical boundaries.  Identical definitions are used for the other spatial directions.  
As follows from Eq.~(\ref{sigma_LIM}),  the first-order artificial mass viscosity at every interface collocated with a limited element is always greater than or equal to the minimum mass diffusion,  $\bar{\bfnc{\sigma}}_{\min}$, required to guarantee the density positivity for the first-order scheme with the explicit Euler discretization in time. The exact formula for $\bar{\bfnc{\sigma}}_{\min}$ is presented  in the companion paper \cite{UY_3Dlow}.

If there exists at least one solution point on a given element, which would otherwise not have positive density, we also require that the mass diffusion coefficient for all interior flux points be sufficient for positivity.  This is achieved by increasing the mass diffusion used for $\hat{\bar{{\bf f}}}^{(AD_1)}_{\hat{\bar{\sigma}},l}$ in Eq.~\eqref{EQN_LIMSCHEME} as follows: 
$
\hat{\bar{\bfnc{\sigma}}}_1(\vec{\xi}_{\bar{i}}) = \max
\left(
\bar{\bfnc{\sigma}}_{\min}(\vec{\xi}_{\bar{i}})
-
\bar{\bfnc{\sigma}}^{AD}_1(\vec{\xi}_{\bar{i}})
, 0 \right).
$

\subsection{Entropy stability}
Let us show that the high-order positivity--preserving flux-limiting semi-discrete scheme given by Eq.~\eqref{EQN_LIMSCHEME} is entropy stable.
\begin{theorem}
\label{TH_ENT_STAB}
The high-order positivity--preserving flux-limiting semi-discrete scheme given by Eq.~\eqref{EQN_LIMSCHEME} is entropy stable.
\end{theorem}
\begin{proof}
Entropy stability of the high-order viscous terms is proven in \cite{CFNF}.  It has been proven in our companion paper \cite{UY_3Dlow} that
the first-order artificial dissipation terms are entropy dissipative.
However, the high- and low-order inviscid entropy conservative terms must be considered together 
to account for the contribution of $ \theta^k_f $.
Lemma~1 in \cite{UY_3Dlow}
equates the entropy contributions of $\hat{\bar{{\bf f}}}_l$  and $\hat{\bar{{\bf f}}}^{(EC)}_l$
where $\hat{\bar{{\bf f}}}^{(MR)}_l = \hat{\bar{{\bf f}}}^{(EC)}_l - \hat{\bar{{\bf f}}}^{(ED)}_l$.  Therefore,
\begin{equation}
\begin{aligned}
\sum\limits_{ l=1}^{3}  
{\bf w}^\top
\mathcal{P}
\mathcal{P}^{-1}_{\xi^l}\Delta_{\xi^l} 
\left[   
\theta^k_f 
\hat{\bar{{\bf f}}}_l
+
(1-\theta^k_f)
\hat{\bar{{\bf f}}}^{(EC)}_l
\right]
=
{\bf w}^\top
\mathcal{P}
\mathcal{P}^{-1}_{\xi^l}\Delta_{\xi^l}\hat{\bar{{\bf f}}}_l  
\end{aligned}
\end{equation}
for all  $0 \leq  \theta^k_f \leq 1$.
 Thus, 
the $\theta^k_f\hat{\bar{{\bf f}}}_l + (1-\theta^k_f)\hat{\bar{{\bf f}}}^{(EC)}_l$ flux is entropy conservative, 
which follows directly from the fact that the high-order flux $\hat{\bar{{\bf f}}}_l$ is entropy conservative, which is proven in \cite{CFNF, CFNPSY}.
\end{proof}

\subsection{Freestream Preservation}
For curvilinear meshes, freestream preservation is an important property that is not guaranteed automatically.    
\begin{theorem}
\label{thm:freePresMix}
The high-order positivity--preserving flux-limiting scheme given by Eq.~\eqref{EQN_LIMSCHEME} is freestream preserving.  
\end{theorem}
\begin{proof}
Let us consider a globally constant state with the consistent Dirichlet boundary conditions and show that  
$\frac{d {\bf \hat{\bf U}}}{dt} = {\bf 0}_5$.
Note that all artificial dissipation and viscous terms including the corresponding penalties depend directly on two-point jumps and high-order computational derivatives of the solution, respectively.  Hence, all viscous terms are identically equal to zero. 

Let us show that all inviscid terms are also exactly equal to zero. 
Indeed, the inviscid penalty terms are equal to zero, 
because of the consistency of the Merriam-Roe flux.
Finally,  $\hat{\bar{{\bf f}}}_l$ and $\hat{\bar{{\bf f}}}^{(EC)}_l$ have been proven to be 
freestream preserving in \cite{CFNPSY} and Lemma~1 in the companion paper \cite{UY_3Dlow}, respectively.
\end{proof}

\subsection{$L_1$ stability}  
Let us now show that the high-order positivity-preserving flux-limiting scheme given by Eq.~\eqref{EQN_LIMSCHEME} is $L_1$ stable. 
Since the proposed scheme is conservative (see Section~\ref{Cons}), the global integrals of density and total energy at the $n$th time level can be recast in the 
following form:
\begin{equation}
\label{conservationIMPLIES}
\begin{split}
\sum\limits^K_{k=1}
{\bf 1}_1^\top
\widehat{\mathcal{P}}
\hat{\bfnc{\rho}}_{k}^n
&
= 
\sum\limits^K_{k=1}
{\bf 1}_1^\top
\widehat{\mathcal{P}}
\hat{\bfnc{\rho}}_{k}^0
+
\sum\limits^{n-1}_{i=0}
B_{\rho}^i,
\\
\sum\limits^K_{k=1}
{\bf 1}_1^\top
\widehat{\mathcal{P}}
\widehat{\bf Et}_{k}^n
&
= 
\sum\limits^K_{k=1}
{\bf 1}_1^\top
\widehat{\mathcal{P}}
\widehat{\bf Et}_{k}^0
+
\sum\limits^{n-1}_{i=0}
B_{Et}^i,
\end{split}
\end{equation}
where $\hat{\bfnc{\rho}}_{k}^n$ and $\widehat{\bf Et}_{k}^n$ are the density and total energy scaled by the Jacobian on the $k$th element at time level $n$ and the $B^i$ terms represent the contribution from the boundaries.  In particular, if the boundary conditions are periodic, then $B^i = 0$ for all $i$.  Taking into account that Eq.~\eqref{EQN_LIMSCHEME} is a pointwise positivity-preserving scheme, we now prove the following theorem. 
\begin{theorem}
Assume that the initial condition is in the admissible set, i.e.,  the initial density and temperature at every solution point are positive.  Furthermore, assume for all $n \in \mathbb{N}$, the initial density and total energy satisfy the following bounds:
$$
c^{\rho}_{\min}
\sum\limits^K_{k=1}
{\bf 1}_1^\top
\widehat{\mathcal{P}}
\hat{\bfnc{\rho}}_{k}^0
\leq
\sum\limits^K_{k=1}
{\bf 1}_1^\top
\widehat{\mathcal{P}}
\hat{\bfnc{\rho}}_{k}^0
+
\sum\limits^{n-1}_{i=0}
B_{\rho}^i 
\leq  
c^{\rho}_{\max}
\sum\limits^K_{k=1}
{\bf 1}_1^\top
\widehat{\mathcal{P}}
\hat{\bfnc{\rho}}_{k}^0,
$$
$$
c^{Et}_{\min}
\sum\limits^K_{k=1}
{\bf 1}_1^\top
\widehat{\mathcal{P}}
\widehat{\bf Et}_{k}^0
\leq
\sum\limits^K_{k=1}
{\bf 1}_1^\top
\widehat{\mathcal{P}}
\widehat{\bf Et}_{k}^0
+
\sum\limits^{n-1}_{i=0}
B_{Et}^i
\leq
c^{Et}_{\max}
\sum\limits^K_{k=1}
{\bf 1}_1^\top
\widehat{\mathcal{P}}
\widehat{\bf Et}_{k}^0,
$$
where $c^{Et}_{\min}$, $c^{Et}_{\max}$, $c^{\rho}_{\min}$, $c^{\rho}_{\max}$ are positive constants, such that
$0< c^{Et}_{\min} \leq c^{Et}_{\max}$ and $0< c^{\rho}_{\min} \leq c^{\rho}_{\max}$.
Then, the discrete solution obtained using the high-order positivity-preserving flux-limiting scheme given by Eq.~\eqref{EQN_LIMSCHEME} satisfies the following $L_1$ bounds:  
\begin{equation}
\begin{array}{ccc}
c^{\rho}_{\min}
\sum\limits^K_{k=1}
{\bf 1}_1^\top
\widehat{\mathcal{P}}
\hat{\bfnc{\rho}}_{k}^0
\leq
&
\sum\limits^K_{k=1}
{\bf 1}_1^\top
\widehat{\mathcal{P}}
\left| \hat{\bfnc{\rho}}_{k}^n  \right| 
&
\leq  
c^{\rho}_{\max}
\sum\limits^K_{k=1}
{\bf 1}_1^\top
\widehat{\mathcal{P}}
\hat{\bfnc{\rho}}_{k}^0,
\\
c^{Et}_{\min}
\sum\limits^K_{k=1}
{\bf 1}_1^\top
\widehat{\mathcal{P}}
\widehat{\bf Et}_{k}^0
\leq
&
\sum\limits^K_{k=1}
{\bf 1}_1^\top
\widehat{\mathcal{P}}
\left| \widehat{\bf Et}_{k}^n  \right|
&
\leq
c^{Et}_{\max}
\sum\limits^K_{k=1}
{\bf 1}_1^\top
\widehat{\mathcal{P}}
\widehat{\bf Et}_{k}^0,
\\
c^{Et}_{\min}
\sum\limits^K_{k=1}
{\bf 1}_1^\top
\widehat{\mathcal{P}}
\widehat{\bf Et}_{k}^0
\leq
&
\sum\limits^K_{k=1}
{\bf 1}_1^\top
\widehat{\mathcal{P}}
\left| \widehat{\bf IE}_{k}^n  \right|
&
\leq
c^{Et}_{\max}
\sum\limits^K_{k=1}
{\bf 1}_1^\top
\widehat{\mathcal{P}}
\widehat{\bf Et}_{k}^0,
\\
c^{Et}_{\min}
\sum\limits^K_{k=1}
{\bf 1}_1^\top
\widehat{\mathcal{P}}
\widehat{\bf Et}_{k}^0
\leq
&
\sum\limits^K_{k=1}
{\bf 1}_1^\top
\widehat{\mathcal{P}}
\left| \widehat{\bf KE}_{k}^n  \right|
&
\leq
c^{Et}_{\max}
\sum\limits^K_{k=1}
{\bf 1}_1^\top
\widehat{\mathcal{P}}
\widehat{\bf Et}_{k}^0,
\end{array}
\end{equation}
where $\left| \hat{\bfnc{\rho}}_{k}^n  \right| $, $\left| \widehat{\bf Et}_{k}^n  \right|$, $\left| \widehat{\bf IE}_{k}^n  \right|$, and $\left| \widehat{\bf KE}_{k}^n  \right|$ are arrays of the absolute values of the density, total energy, internal energy, and kinetic energy at all solution points on the $k$th element at time level $n$.
\end{theorem}
\begin{proof}
The pointwise positivity implies that $\left| \hat{\bfnc{\rho}}_{k}^n  \right| = \hat{\bfnc{\rho}}_{k}^n$, $\left| \widehat{\bf Et}_{k}^n  \right| = \widehat{\bf Et}_{k}^n$, $\left| \widehat{\bf IE}_{k}^n  \right| = \widehat{\bf IE}_{k}^n$, and $\left| \widehat{\bf KE}_{k}^n  \right| = \widehat{\bf KE}_{k}^n$  on every element.
Hence, the bounds for density and total energy are an immediate consequence of Eq.~\eqref{conservationIMPLIES}.  Furthermore, the positivity of the internal and kinetic energy functions at every solution point in the domain guarantees that the remaining bounds hold, since $IE_{k,i}$ and $KE_{k,i}$ are bounded from above by $ET_{k,i}$ at each solution point $i$.
\end{proof}

\subsection{$L_2$ stability}  
We now show that the proposed flux limiting scheme is also $L_2$  stable. 
Define a new convex entropy 
$
\bar{S} = S - S({\bm u}_0) - S_{\U}({\bm u}_0)^\top({\bm u} - {\bm u}_0)$, where ${\bm u}_0$ is a constant non-zero state with zero velocity and the associated entropy variables
\begin{equation}
\label{eq:Evariables_U_0_DISCRETE}
\bar{\bm{w}}
\equiv 
\bar{S}_{\U}
= 
S_{\U}
-
S_{\U}({\bm u}_0)
=
\bm{w}
-
\bm{w}_0.
\end{equation}
For this new entropy $\bar{S}$, we form the corresponding discrete entropy variables $\bar{{\bf w}}_k = {\bf w}_k - {\bf w}_0$ on the $k$th element, where ${\bf w}_0(\vec{\xi}_{ijl}) = \bm{w}_0$ for every solution point of each element.
Contracting Eq.~\eqref{EQN_LIMSCHEME} with the new entropy variables, $\bar{{\bf w}}_k$, yields
\begin{equation}
\label{eq:Sint}
\begin{array}{ll}
\sum\limits^{K}_{k=1}
\bar{{\bf w}}_k^\top \mathcal{P} \frac{d {\bf \hat{\bf U}}_k}{dt}
=
\sum\limits^{K}_{k=1}
{\bf w}_k^\top \mathcal{P} \frac{d {\bf \hat{\bf U}}_k}{dt}
-
{\bf w}_0^\top \mathcal{P} \frac{d {\bf \hat{\bf U}}_k}{dt}.
\end{array}
\end{equation}
From Theorem~\ref{TH_ENT_STAB}, it follows that 
$$
\sum\limits^{K}_{k=1}
{\bf w}_k^\top \mathcal{P} \frac{d {\bf \hat{\bf U}}_k}{dt}
=
B
-
\mathcal{D},
$$
where $B$ contains all boundary contributions and $\mathcal{D} \geq 0$ is the total entropy dissipation.
Assuming that boundary conditions are periodic 
yields 
$$
\sum\limits^{K}_{k=1}
{\bf w}_k^\top \mathcal{P} \frac{d {\bf \hat{\bf U}}_k}{dt}
\leq
-
\mathcal{D}.
$$
Taking into account that 
\begin{equation}
\sum\limits^{K}_{k=1}
{\bf w}_0^\top \mathcal{P} \frac{d {\bf \hat{\bf U}}_k}{dt}
=
\bm{w}_0
\sum\limits^{K}_{k=1}
{\bf 1}_1^\top \mathcal{P} \frac{d {\bf \hat{\bf U}}_k}{dt}
\end{equation}
and the boundary conditions are periodic, i.e.,  ${\bf 1}_1^\top \mathcal{P} \frac{d {\bf \hat{\bf U}}_k}{dt}=0$ , 
we can obtain the following estimate from Eq.~(\ref{eq:Sint}): 
\begin{equation}
\label{eq:wdUdtBounds}
\sum\limits^{K}_{k=1}
\bar{{\bf w}}_k^\top \mathcal{P} \frac{d {\bf \hat{\bf U}}_k}{dt}
\leq
\sum\limits^{K}_{k=1}
{\bf w}_k^\top \mathcal{P} \frac{d {\bf \hat{\bf U}}_k}{dt}
\leq
- \mathcal{D}.
\end{equation}
Using
$\bar{{\bf w}}_k^\top \mathcal{P} \frac{d {\bf \hat{\bf U}}_k}{dt} = {\bf 1}_1^\top \widehat{\mathcal{P}} \frac{d \hat{\bar{{\bf S}}}}{dt}$ and
integrating Eq.~(\ref{eq:wdUdtBounds}) in time from $t^n$ to $t^{n+1}$  yield
\begin{equation}
\label{eq:Sn+1bound}
\begin{split}
\sum\limits^{K}_{k=1}
{\bf 1}_1^\top \widehat{\mathcal{P}} \hat{\bar{{\bf S}}}^{n+1}_k
\leq
\sum\limits^{K}_{k=1}
{\bf 1}_1^\top \widehat{\mathcal{P}} \hat{\bar{{\bf S}}}^n_k
-
\tau^n
\bar{\mathcal{D}}^n,
\end{split}
\end{equation}
where $\bar{\mathcal{D}}^n \geq 0$ is the time-averaged entropy dissipation over the time interval $[t^n, t^{n+1}]$.
Applying the above inequality $n$ times, we have
\begin{equation}
\label{eq:SnboundD}
\begin{split}
\sum\limits^{K}_{k=1}
{\bf 1}_1^\top \widehat{\mathcal{P}} \hat{\bar{{\bf S}}}^{n}_k
\leq
\sum\limits^{K}_{k=1}
{\bf 1}_1^\top \widehat{\mathcal{P}} \hat{\bar{{\bf S}}}^1_k
-
\sum\limits^{n-1}_{i=1}
\tau^i
\bar{\mathcal{D}}^i.
\end{split}
\end{equation}   

Expanding $S$ in the Taylor series  about ${\bm u}_0$ gives
\begin{equation}
\begin{array}{ll}
S({\bf U}^n_k(\vec{\xi}_{a})) 
& = 
S({\bm u}_0)
+
S_{\U}({\bm u}_0)^\top
\left(
{\bf U}^n_k(\vec{\xi}_{a})
-
{\bm u}_0
\right)
\\
&
+
\frac{1}{2}
\left(
{\bf U}^n_k(\vec{\xi}_{a})
-
{\bm u}_0
\right)^\top
S_{\U\U}(\widetilde{{\bf U}}^n_k(\vec{\xi}_{a}))
\left(
{\bf U}^n_k(\vec{\xi}_{a})
-
{\bm u}_0
\right),
\end{array}
\end{equation}
where the state $\widetilde{{\bf U}}^n_k(\vec{\xi}_{a})$ has positive density and temperature for every solution point  $\vec{\xi}_a$, since ${\bm u}_0$ and ${\bf U}^n_k(\vec{\xi}_{a})$ are both in the admissible set. Note that
$$
\begin{array}{l}
\bar{{\bf S}}^n_k(\vec{\xi}_{a}) = 
{\bf S}^n_k(\vec{\xi}_{a}) - S({\bm u}_0) - S_{\U}({\bm u}_0)^\top({\bf U}^n_k(\vec{\xi}_{a}) - {\bm u}_0)
= \\
\frac{1}{2}
\left(
{\bf U}^n_k(\vec{\xi}_{a})
-
{\bm u}_0
\right)^\top
S_{\U\U}(\widetilde{{\bf U}}^n_k(\vec{\xi}_{a}))
\left(
{\bf U}^n_k(\vec{\xi}_{a})
-
{\bm u}_0
\right).
\end{array}
$$  
Let $\lambda^{\min,n}_{S_{\U\U}}$ be the minimum eigenvalue of all Hessian matrices $S_{\U\U}(\widetilde{{\bf U}}^n_k(\vec{\xi}_{a}))$ in the domain. 
Then, defining $\mathcal{C}^{n-1}$ as
$$
\mathcal{C}^{n-1} = \sum\limits^{K}_{k=1}
{\bf 1}_1^\top \widehat{\mathcal{P}} \hat{\bar{{\bf S}}}^1_k
-
\sum\limits^{n-1}_{i=1}
\tau^i
\bar{\mathcal{D}}^i
\leq
\sum\limits^{K}_{k=1}
{\bf 1}_1^\top \widehat{\mathcal{P}} \hat{\bar{{\bf S}}}^1_k,
$$ 
we have
\begin{equation}
\label{eq:SUUbound}
2\lambda^{\min,n}_{S_{\U\U}}
\sum\limits^{K}_{k=1}
\left(
{\bf U}^n_k
-
{\bf U}_0
\right)^\top
\mathcal{P}
\left[J \right]_k
\left(
{\bf U}^n_k
-
{\bf U}_0
\right)
\leq
4
\mathcal{C}^{n-1}.
\end{equation}
Using the following inequality: ${\bf Y}^\top{\bf Y} \leq 2({\bf Y}-{\bf Y}_0)^\top({\bf Y}-{\bf Y}_0) + 2{\bf Y}_0^\top{\bf Y}_0$ that holds 
$\forall {\bf Y}, {\bf Y}_0 \in \mathbb{R}^m$, Eq.~(\ref{eq:SUUbound}) can be recast as follows:
\begin{equation}
\sum\limits^{K}_{k=1}
\left(
{\bf U}^n_k
\right)^\top
\mathcal{P}
\left[J \right]_k
{\bf U}^n_k
\leq
4
\frac{\mathcal{C}^{n-1}}
{\lambda^{\min,n}_{S_{\U\U}}}
+
2
\sum\limits^{K}_{k=1}
{\bf U}_0^\top
\mathcal{P}
\left[J \right]_k
{\bf U}_0,
\end{equation}
thus providing a discrete $L_2$ bound on the solution.  

An alternative $L_2$ bound can be obtained by using the Cholesky  decomposition of $\mathcal{S}_{\U\U} = \mathscr{L}\mathscr{D}\mathscr{L}^\top$,
where the exact expressions of the matrices $\mathscr{L}$ and $\mathscr{D}$ are presented in \cite{JUthesis}.
Based on the Cholesky  decomposition of $\mathcal{S}_{\U\U}$, it has been shown in \cite{JUthesis} that the following bounds hold:
 \begin{equation}
\begin{array}{ll}
&
\left(
{\bf U}^n_k(\vec{\xi}_{j})
-
{\bm u}_0
\right)^\top
S_{\U\U}(\widetilde{{\bf U}}^n_k(\vec{\xi}_{j}))
\left(
{\bf U}^n_k(\vec{\xi}_{j})
-
{\bm u}_0
\right)
\geq
\frac{(\bfnc{\rho}^n_k(\vec{\xi}_{j})-\rho_0)^2}
{b_1(\widetilde{{\bf U}}^n_k(\vec{\xi}_{j}))},
\\
&
\left(
{\bf U}^n_k(\vec{\xi}_{j})
-
{\bm u}_0
\right)^\top
S_{\U\U}(\widetilde{{\bf U}}^n_k(\vec{\xi}_{j}))
\left(
{\bf U}^n_k(\vec{\xi}_{j})
-
{\bm u}_0
\right)
\geq
\frac{
((\bfnc{m}_i)^n_k(\vec{\xi}_{j}))^2}
{b_{i+1}(\widetilde{{\bf U}}^n_k(\vec{\xi}_{j}))}, \quad \, i=1,2,3,
\\
&
\left(
{\bf U}^n_k(\vec{\xi}_{j})
-
{\bm u}_0
\right)^\top
S_{\U\U}(\widetilde{{\bf U}}^n_k(\vec{\xi}_{j}))
\left(
{\bf U}^n_k(\vec{\xi}_{j})
-
{\bm u}_0
\right)
\geq
\frac{({\bf Et}^n_k(\vec{\xi}_{j})-\E t_0)^2}
{b_5(\widetilde{{\bf U}}^n_k(\vec{\xi}_{j}))},
\\
&
b_1({\bm u}_j)
=
\frac{\rho_j}{R},
\quad
b_{i+1}({\bm u}_j)
=
\frac{\fnc{P}_j + \rho_j (V_{i}^2)_j}{R},
\quad \, i=1,2,3,
\\
&
b_5({\bm u}_j)
=
\frac{\fnc{P}^2_j\gamma 
+ 
\fnc{P}_j\rho_j \|\bfnc{V}_j\|^2\gamma
+
\left(
\rho_j
\frac{\|\bfnc{V}_j\|^2}{2}
\right)^2
}{R\rho_j},
\end{array}
\end{equation}
where $(\bfnc{m}_i)^n_k(\vec{\xi}_{j}) = \bfnc{\rho}^n_k(\vec{\xi}_{j})(\bfnc{V}_i)^n_k(\vec{\xi}_{j})  $ is the $i$th component of momentum and ${\bf Et}^n_k(\vec{\xi}_{j}) = \bfnc{\rho}^n_k(\vec{\xi}_{j}){\bf E}^n_k(\vec{\xi}_{j})$ is the total energy.  Let $b^{\max,n}_i = \max\limits_{1 \leq k \leq K}\max\limits_{1 \leq j \leq N_p} b_i(\widetilde{{\bf U}}^n_k(\vec{\xi}_{j}))$,  then we obtain the following $L_2$
bounds on the solution:
\begin{equation}
\begin{array}{ll}
\sum\limits^{K}_{k=1}
\left(
\bfnc{\rho}^n_k
\right)^\top
\widehat{\mathcal{P}}
\hat{\bfnc{\rho}}^n_k
&
\leq
4
b^{\max,n}_1
\mathcal{C}^{n-1}
+
2
\sum\limits^{K}_{k=1}
\bfnc{\rho}_0^\top
\widehat{\mathcal{P}}
\hat{\bfnc{\rho}}_0,
\\
\sum\limits^{K}_{k=1}
\left(
(\bfnc{m}_i)^n_k
\right)^\top
\widehat{\mathcal{P}}
(\hat{\bfnc{m}}_i)^n_k
&
\leq
2
b^{\max,n}_{i+1}
\mathcal{C}^{n-1}, \quad  \quad \quad i=1,2,3,
\\
\sum\limits^{K}_{k=1}
\left(
{\bf Et}^n_k
\right)^\top
\widehat{\mathcal{P}}
\widehat{{\bf Et}}^n_k
&
\leq
4
b^{\max,n}_5
\mathcal{C}^{n-1}
+
2
\sum\limits^{K}_{k=1}
{\bf Et}_0^\top
\widehat{\mathcal{P}}
\widehat{{\bf Et}}_0.
\end{array}
\end{equation}

\section{Numerical Results}
\label{results}
We now assess the accuracy, discontinuity-capturing, and positivity-pre-serving properties of the proposed family of high-order entropy stable spectral 
collocation schemes for the 3-D compressible Navier-Stokes equations on standard benchmark problems with smooth and discontinuous solutions. 
In all numerical experiments  presented herein, the 3rd-order strong stability preserving (SSP) Runge-Kutta
scheme developed in \cite{SSP} is used to advance the semi-discretization in time. Note
that this scheme violates the entropy stability property of the semi-discrete
operator by a factor proportional to the local temporal truncation error. 
%
%
\begin{table}[!h]
\begin{center}
\begin{tabular}{ccccccccc}
\hline
\quad  &  ESSC &\quad & \quad & \quad  & PPESAD & \quad & \quad  & \quad  \\
$K$   & $L_{\infty}$ error  &  rate  & $L_2$ error  &  rate &\quad $L_{\infty}$ error &  rate & $L_2$ error  &  rate \\
\hline
&&&& $p = 4$ &&&& \\
\hline
  3  & $\bm{ 1.24}$   & --       & $\bm{ 4.75\textbf{e-}2}$  & --   & $\bm{ 0.68}$  & --     & $\bm{4.03\textbf{e-}2}$ &  --     \\
  6  & $\bm{ 0.80}$   & 0.63     & $\bm{ 8.50\textbf{e-}3}$  & 2.48 & $\bm{ 0.56}$  & 0.27   & $\bm{8.20\textbf{e-}3}$ &  2.30   \\
 12  & 0.11           & 2.89     & 8.51e-4          & 3.32 &  0.11         & 2.37   & 8.51e-4        &  3.27   \\
 24  & 6.93e-3        & 3.96     & 4.96e-5          & 4.10 &  6.93e-3      & 3.96   & 4.96e-5        &  4.10   \\
 48  & 3.09e-4        & 4.49     & 1.54e-6          & 5.01 &  3.09e-4      & 4.49   & 1.54e-6        &  5.01   \\
\hline
&&&& $p = 5$ &&&& \\
\hline
  3  & $\bm{ 3.15}$   & --       & $\bm{ 3.30\textbf{e-}2}$  & --   & $\bm{0.88}$   & --     & $\bm{ 2.99\textbf{e-}2}$ &  --      \\
  6  & 0.34           & 3.20     & 4.13e-3          & 3.00 &  0.34         & 1.36   & 4.13e-3         & 2.86   \\
 12  & 4.37e-2        & 2.97     & 2.49e-4          & 4.05 &  4.37e-2      & 2.97   & 2.49e-4         & 4.05   \\
 24  & 2.30e-3        & 4.25     & 7.77e-6          & 5.00 &  2.30e-3      & 4.25   & 7.77e-6         & 5.00   \\
 48  & 3.50e-5        & 6.04     & 9.99e-8          & 6.28 &  3.50e-5      & 6.04   & 9.99e-8         & 6.28   \\
\hline
&&&& $p = 6$ &&&& \\
\hline
  3  & $\bm{ 1.27}$   & --       & $\bm{ 2.11\textbf{e-}2}$  & --   & $\bm{ 0.52}$  & --     & $\bm{ 1.99\textbf{e-}2}$ & --       \\
  6  & 0.12           & 3.35     & 1.92e-3          & 3.46 & 0.12          & 2.07   & 1.92e-3    & 3.38   \\
 12  & 1.44e-2        & 3.11     & 7.33e-5          & 4.71 & 1.44e-2       & 3.11   & 7.33e-5    & 4.71   \\
 24  & 3.27e-4        & 5.46     & 1.20e-6          & 5.94 & 3.27e-4       & 5.46   & 1.20e-6    & 5.94   \\
 48  & 3.06e-6        & 6.74     & 7.56e-9          & 7.31 & 3.06e-6       & 6.74   & 7.56e-9    & 7.31   \\
\hline
\end{tabular}
\end{center}
\caption{\label{tab2} $L_{\infty}$ and $L_2$ errors and their convergence rates obtained 
with the ESSC and PPESAD schemes for $p = 4, 5, 6$ for the viscous shock problem on 3-D nonuniform grids at $t=0.1$. 
}
\end{table}
The time step in our numerical experiments is selected by using the Courant-Friedrich-Levy (CFL)-type condition 
and the density and temperature positivity constraints presented in the companion paper \cite{UY_3Dlow}.
The following acronyms are used for numerical schemes in this section.
The baseline high-order entropy stable spectral collocation scheme with polynomial order ``\#'' given by Eq.~\eqref{semiDiscCurviBaseline} is denoted as ESSC-p\#.  The new positivity preserving entropy stable artificial dissipation scheme (Eq.~\eqref{EQN_LIMSCHEME}) is denoted as PPESAD-p\#. The PPESAD-p\# scheme with $\bfnc{\mu}^{AD}$ set to zero is  denoted as PPES-p\#.
\begin{figure}[!h] 
    \begin{center}
	\begin{subfigure}{0.32\textwidth}
		\includegraphics[width=0.9\linewidth]{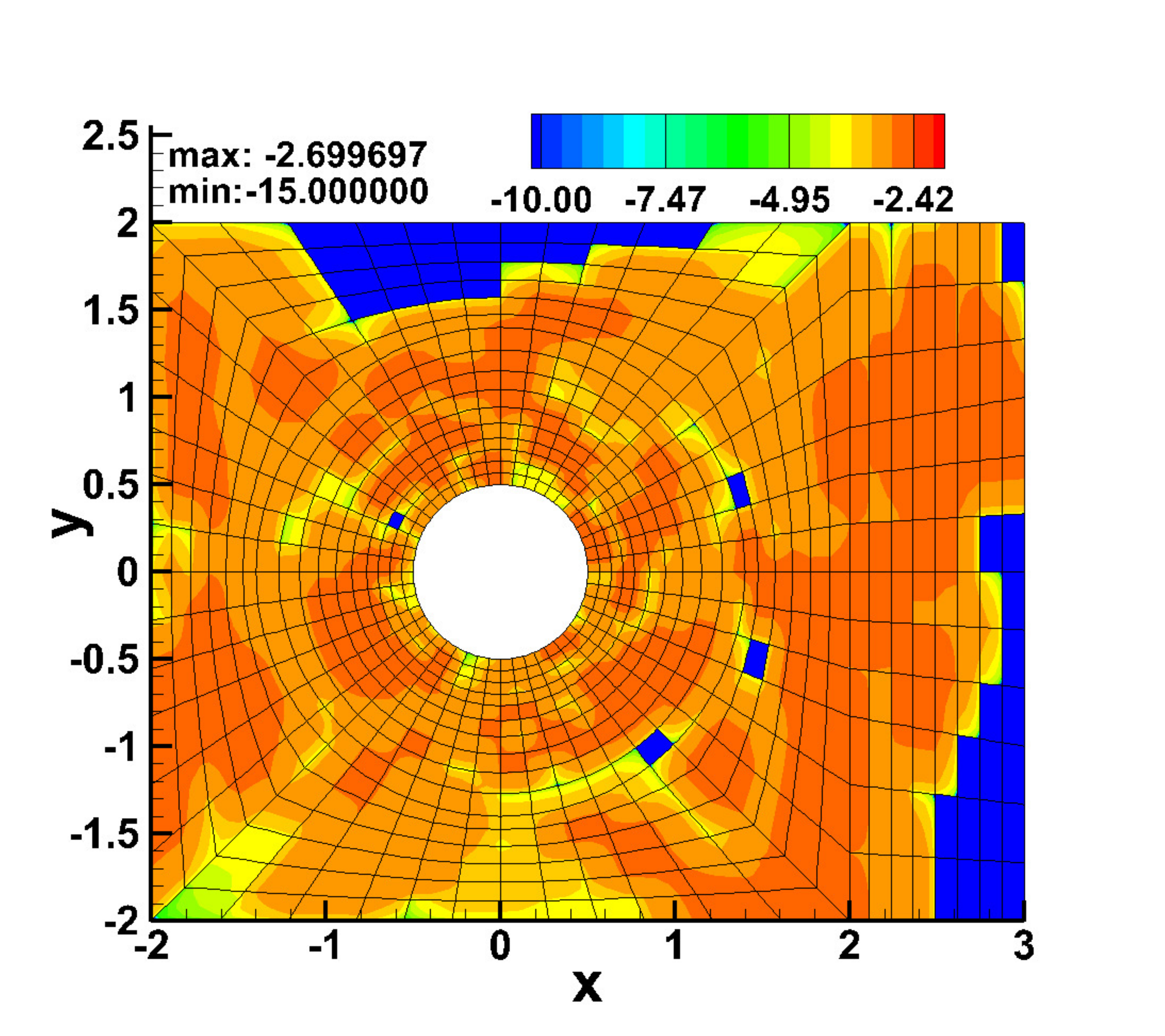} 
		\caption{}
	\end{subfigure}
	\begin{subfigure}{0.32\textwidth}
		\includegraphics[width=0.9\linewidth]{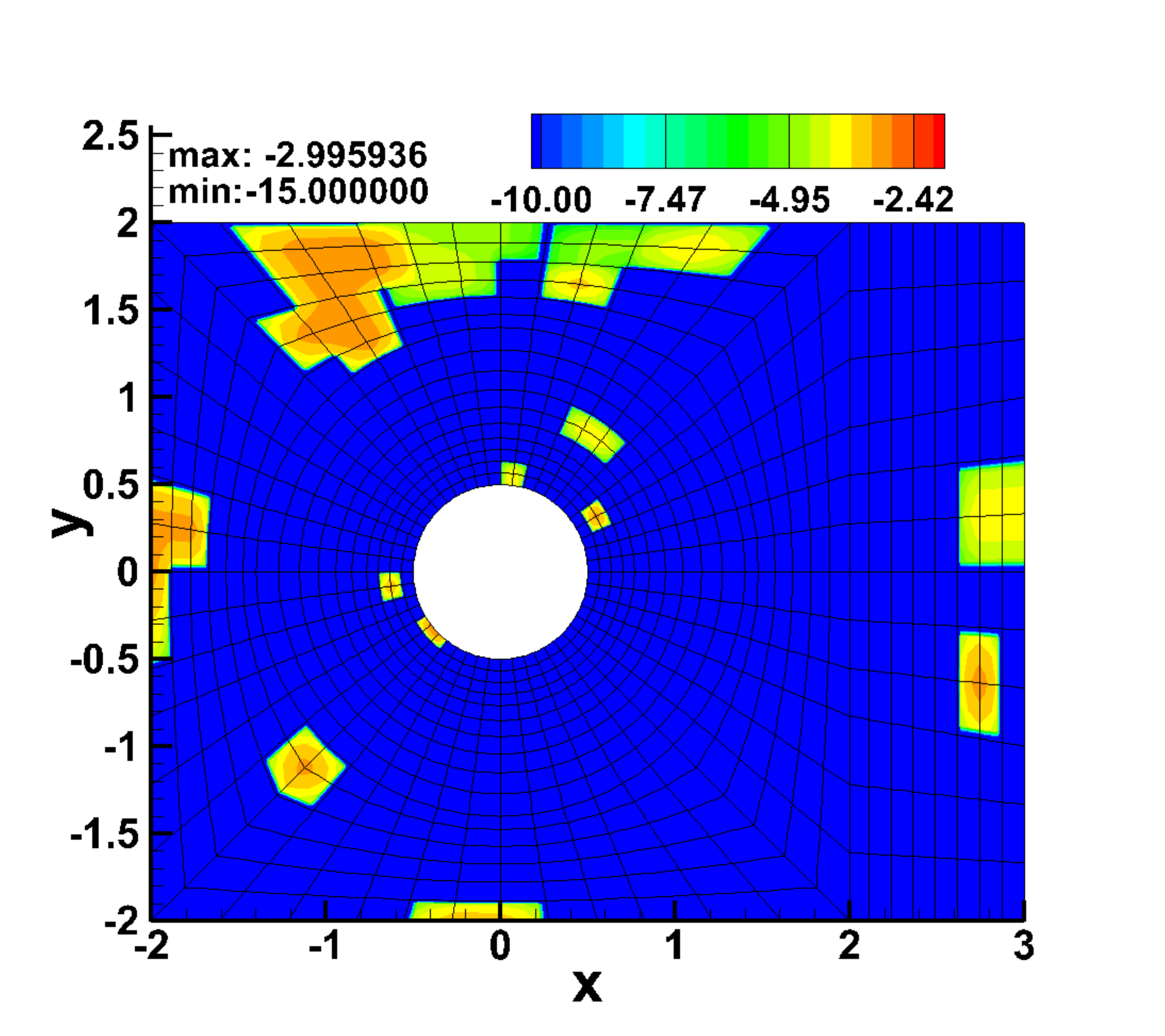}
		\caption{}
	\end{subfigure}
	\begin{subfigure}{0.32\textwidth}
		\includegraphics[width=0.9\linewidth]{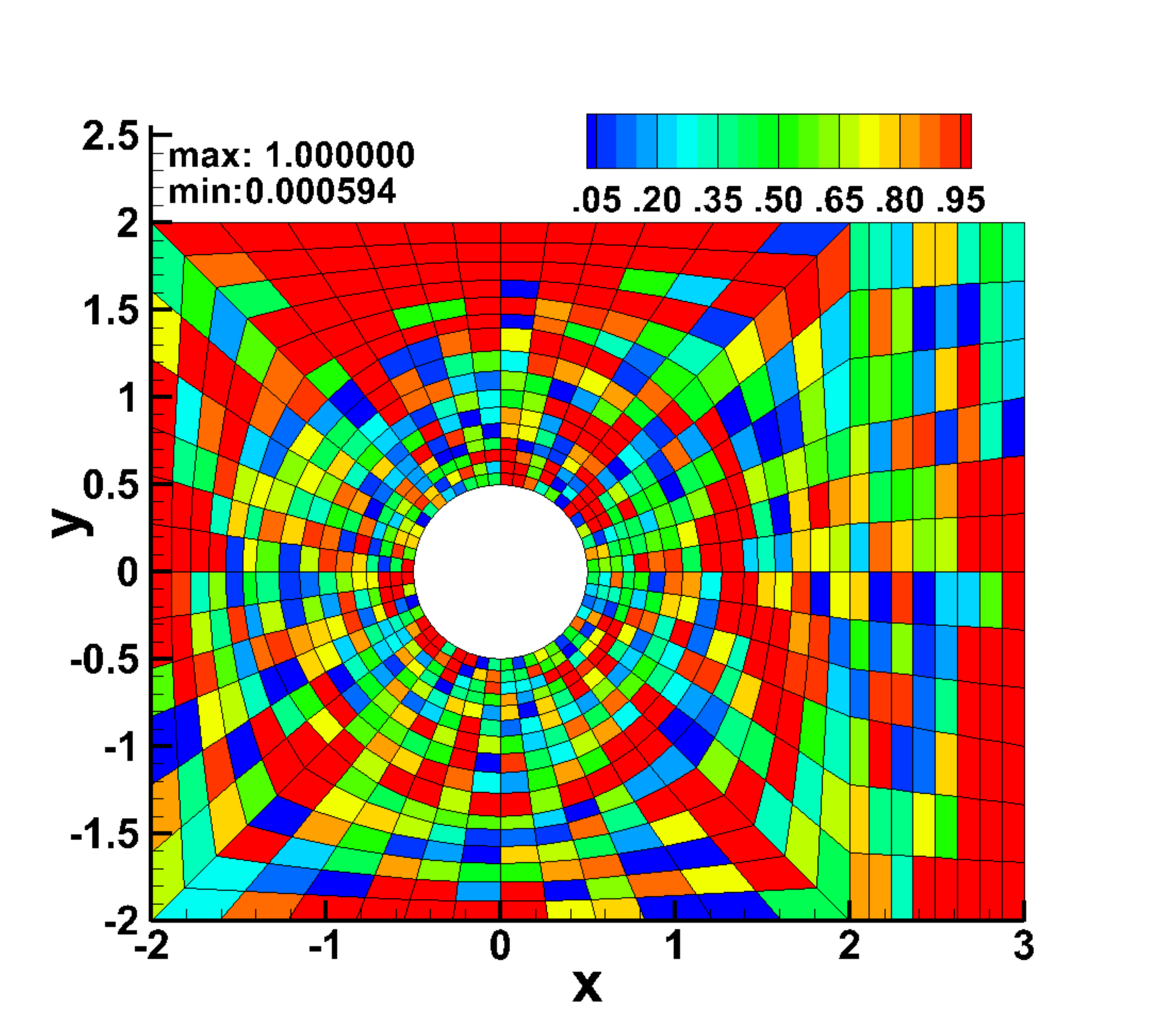}
		\caption{}
	\end{subfigure}
	\end{center}
	 \caption{Contours of randomly generated low-order artificial viscosity (left panel), high-order artificial viscosity (middle panel),  and flux limiter (right panel) obtained with the PPESAD-p4 scheme for the freestream preservation problem at $t = 10$.}  
\label{freeStreamFig}
\end{figure}

 \subsection{$3$-D Viscous Shock}
%
To validate that the proposed schemes are design-order accurate, we consider the propagation of a $3$-D viscous shock on a sequence of randomly perturbed nonuniform grids.
The 1-D viscous shock, which  possesses a smooth analytical solution at the Prandtl number  $Pr = 3/4$, is rotated so that it propagates along the direction $[1, 1, 1]^\top$ and is initially centered at the origin. The Reynolds and Mach numbers are set 
as follows: $Re = 50$ and $Ma = 2.5$.  The governing equations are integrated until $t_{\text{final}} = 0.1$.   
For all polynomial orders presented in Table~\ref{tab2}, the proposed PPESAD scheme outperforms the corresponding baseline ESSC scheme in terms of accuracy on coarse grids, for which the discrete solution is under-resolved (see the results shown in bold). As the grid is refined and the viscous shock becomes fully resolved, the artificial viscosity coefficient 
$\bfnc{\mu}^{AD}$ becomes identically equal to zero and the PPESAD schemes demonstrates the same design-order error convergence as the ESSC scheme. Based on these results, we can conclude that the proposed PPESAD scheme dissipates under-resolved flow features in such a manner that reduces the error, while providing the same accuracy as the underlying ESSC scheme when the solution is sufficiently smooth and fully resolved. 
\begin{figure}[!h] 
	\begin{subfigure}{0.5\textwidth}
		\includegraphics[width=0.9\linewidth]{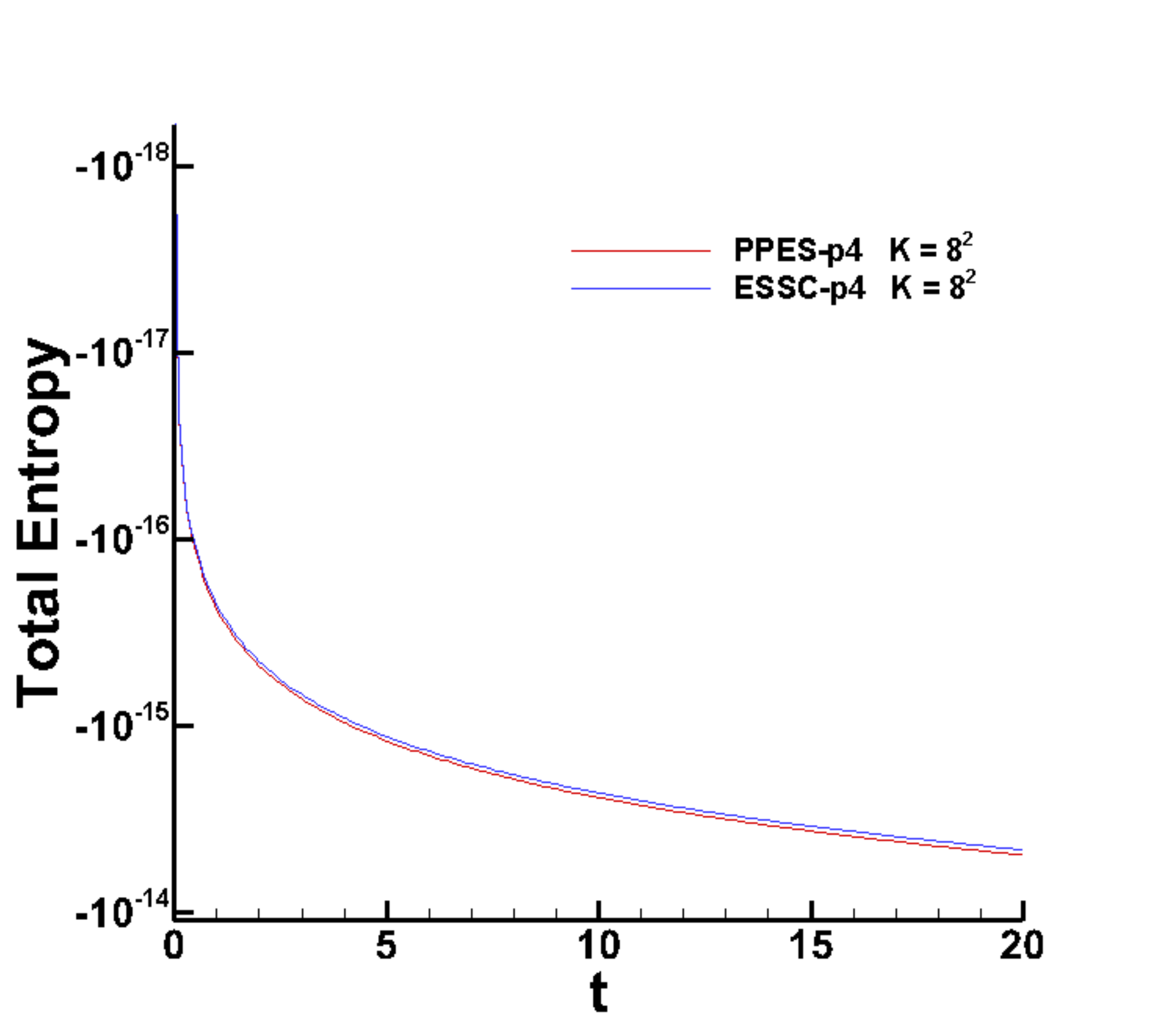} 
		\caption{}
	\end{subfigure}
	\begin{subfigure}{0.5\textwidth}
		\includegraphics[width=0.9\linewidth]{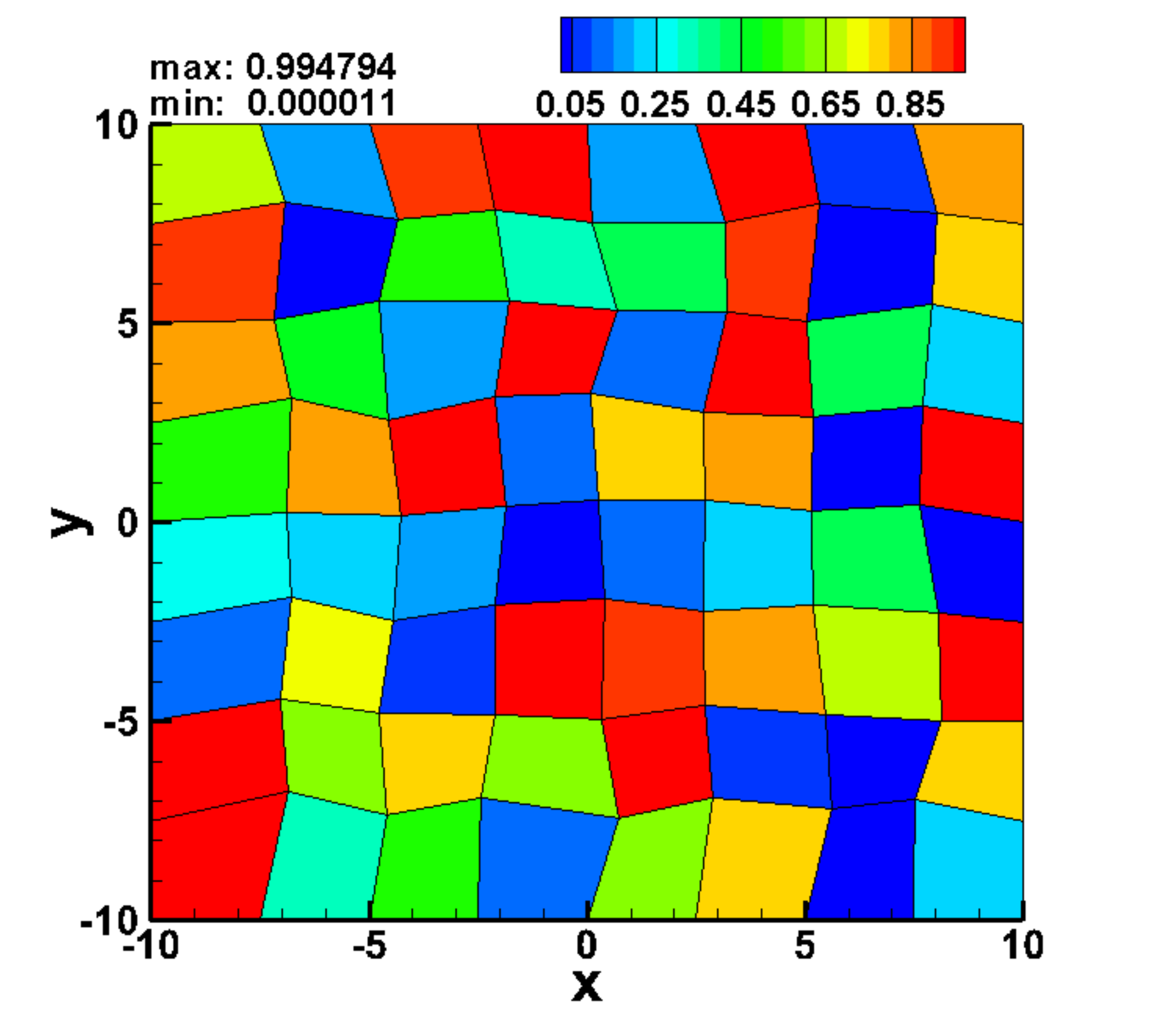}
		\caption{}
	\end{subfigure}
	 \caption{Time histories (left panel) of the total entropy computed with the ESSC-p4 and PPES-p4 schemes 
	 and the PPES-p4 limiter coefficient (right panel)  on a randomly perturbed $K=8^2$ grid 
	 for the isentropic vortex problem.}  
\label{isentropicVortexFig}
\end{figure}

\subsection{Freestream preservation}
We now corroborate our theoretical results presented in Theorem~\ref{thm:freePresMix} and show that
the new high-order positivity--preserving flux-limiting scheme given by Eq.~\eqref{EQN_LIMSCHEME} is freestream preserving on static curvilinear grids. 
To demonstrate this property, the 2D constant viscous flow with $\rho=1$, $T=1$, $\bfnc{V}=[\cos(10^0), \sin{10^0}, 0]^\top$ at $Re = 500$, $Ma = 3.5$, and $Pr = 0.7$ is solved by using the PPESAD scheme on a 864-element genuinely curvilinear grid around a cylinder. 
%
To ensure that all terms in the high-order positivity--preserving flux-limiting scheme are turned on during the simulation,  we randomly set $\bfnc{\mu}^{AD}_p$ and $\bar{\bfnc{\mu}}^{AD}_1$ (see Section~\ref{VISC_LIMSCHEME}) to values between  0 and $1/Re$, and the flux limiter $\theta_f$ (see Section~\ref{SEMIDISC_LIMSCHEME}) to a value between 0 and 1 at each Runge-Kutta stage. As evident in Figure~\ref{freeStreamFig}, all artificial dissipation and flux-limiting terms in the PPESAD scheme given by Eq.~\eqref{EQN_LIMSCHEME} are nonzero throughout the simulation.  Nonetheless, the global $L_2$ and $L_{\infty}$ errors at the final time $t_{\rm final} = 10$ are $2.84\mathrm{e}{-15}$  and $1.46\mathrm{e}{-13}$, respectively, thus corroborating our theoretical results.

\subsection{Entropy Conservation}
In the companion paper \cite{UY_3Dlow} (see Lemma~1), it has been proven that the first-order positivity-preserving entropy stable scheme is entropy conservative for inviscid smooth flows, if all artificial dissipation terms are turned off.  To demonstrate this, we solve the inviscid isentropic vortex flow with periodic boundary conditions at $Ma=0.3$ on a randomly perturbed coarse grid (see Figure~\ref{isentropicVortexFig}).  The vortex is initially located at $(0,0)$, propagates to the right, and returns to the origin by $t_{\text{final}} = 20$. This test problem has the exact solution (e.g., see \cite{GCLSS2019}). To validate that the proposed flux-limiting scheme is entropy conservative, we randomly set the limiter value in the range between 0 and 1 at each grid element, as shown in Fig.~\ref{isentropicVortexFig}.
%
%
\begin{figure}[!h] 
	\begin{subfigure}{0.5\textwidth}
		\includegraphics[width=0.9\linewidth]{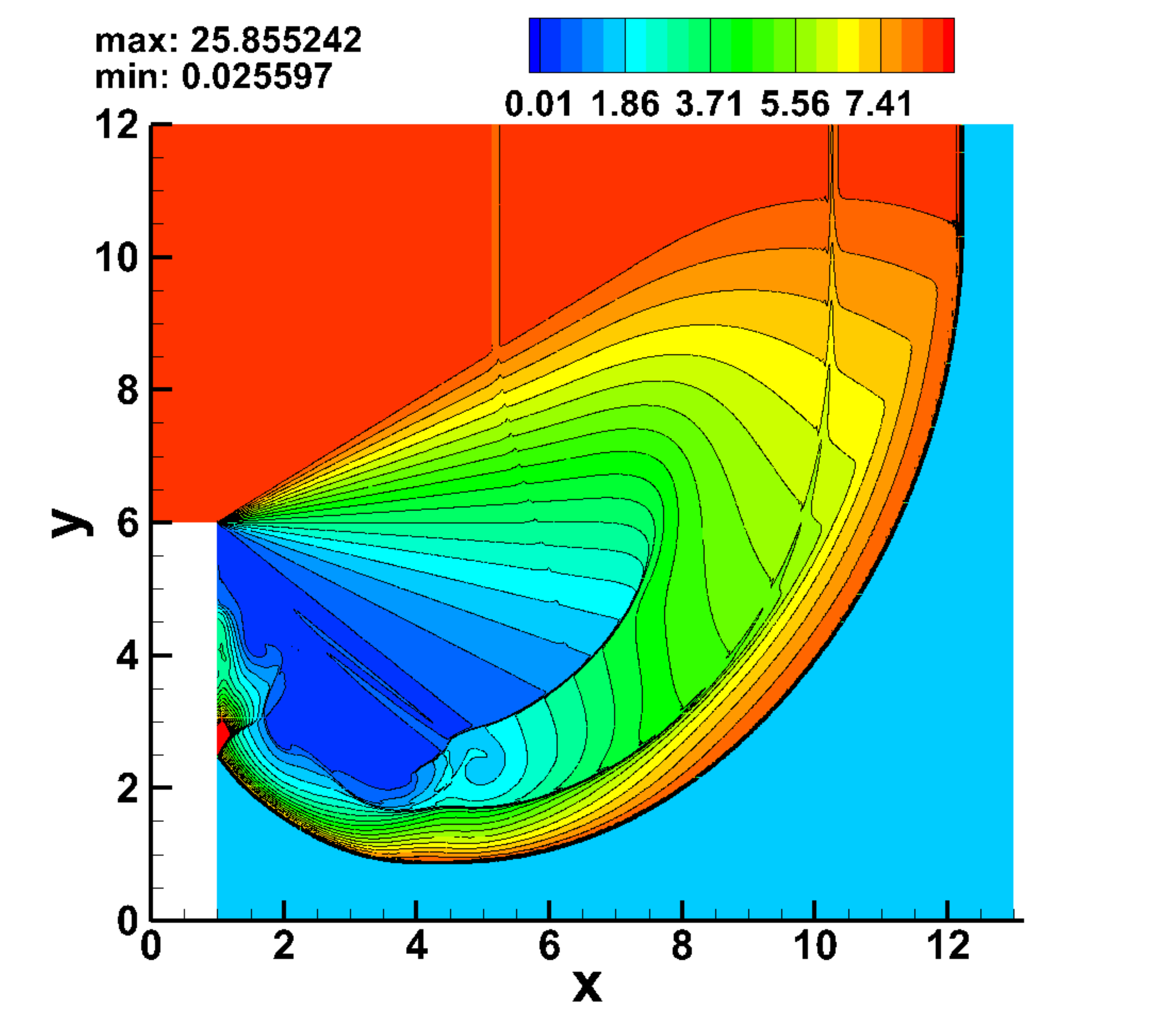}
		\caption{}
	\end{subfigure}
	\begin{subfigure}{0.5\textwidth}
		\includegraphics[width=0.9\linewidth]{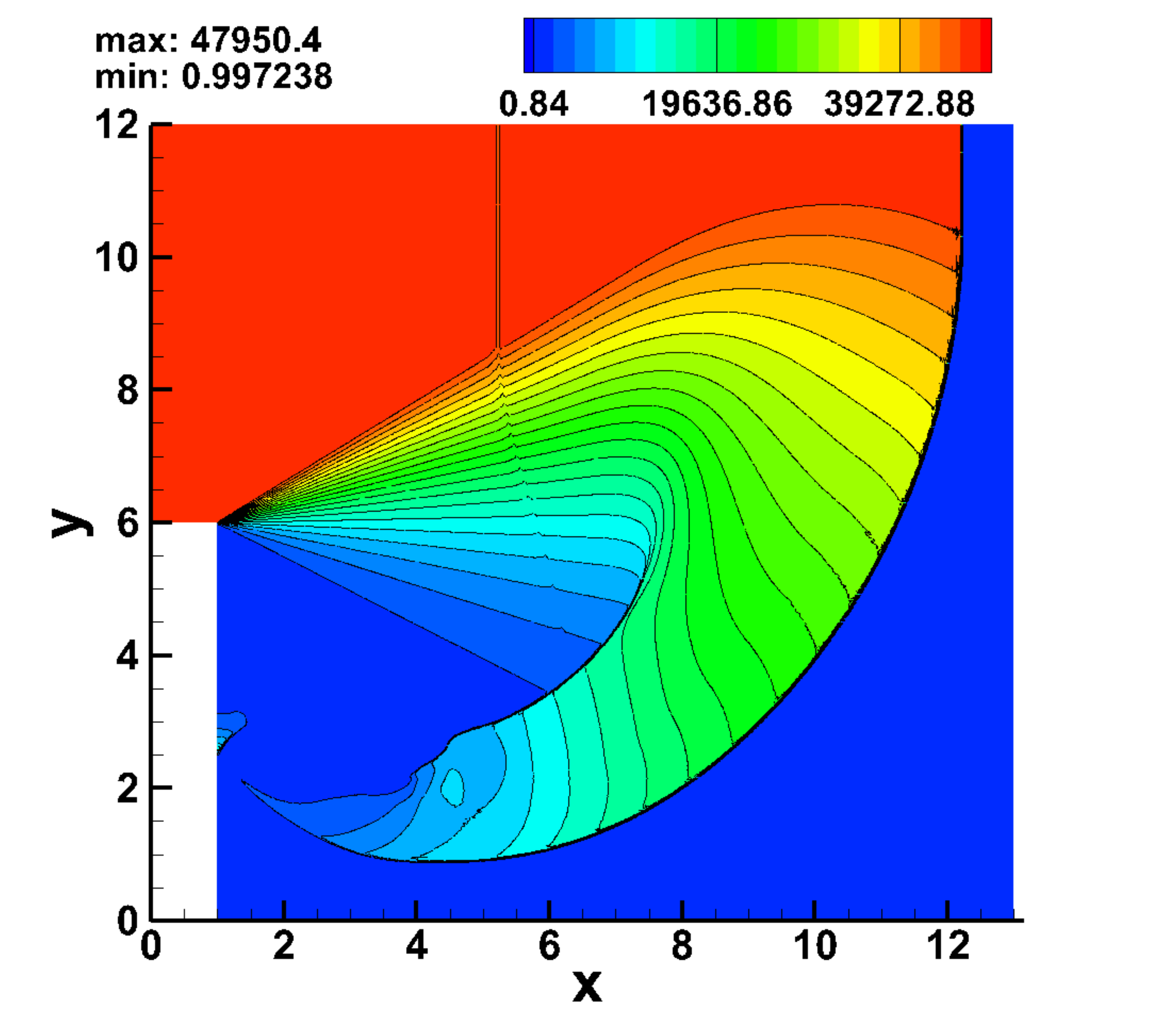}
		\caption{}
	\end{subfigure}
	\\
	\begin{subfigure}{0.5\textwidth}
		\includegraphics[width=0.9\linewidth]{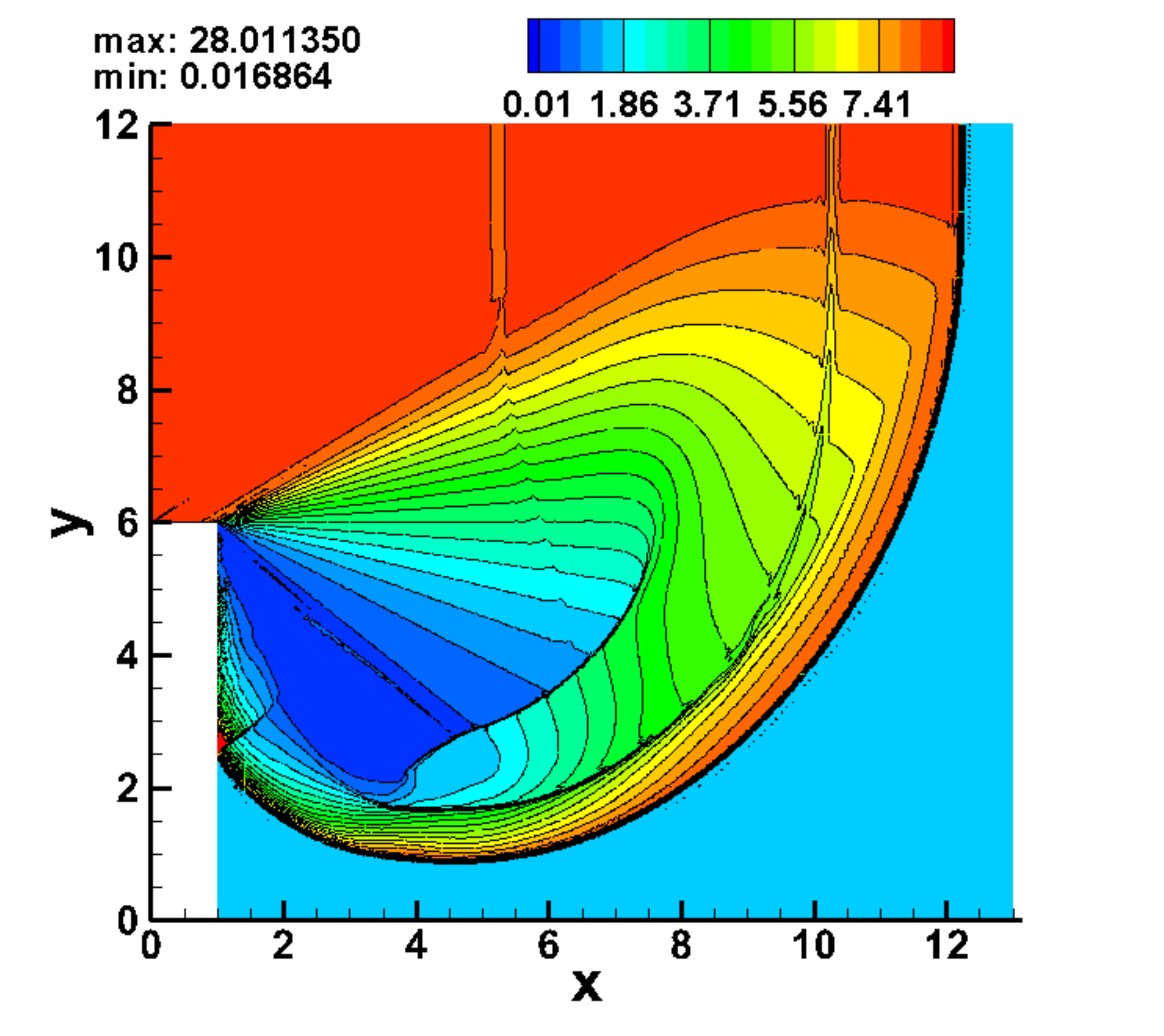} 
		\caption{}
	\end{subfigure}
	\begin{subfigure}{0.5\textwidth}
		\includegraphics[width=0.9\linewidth]{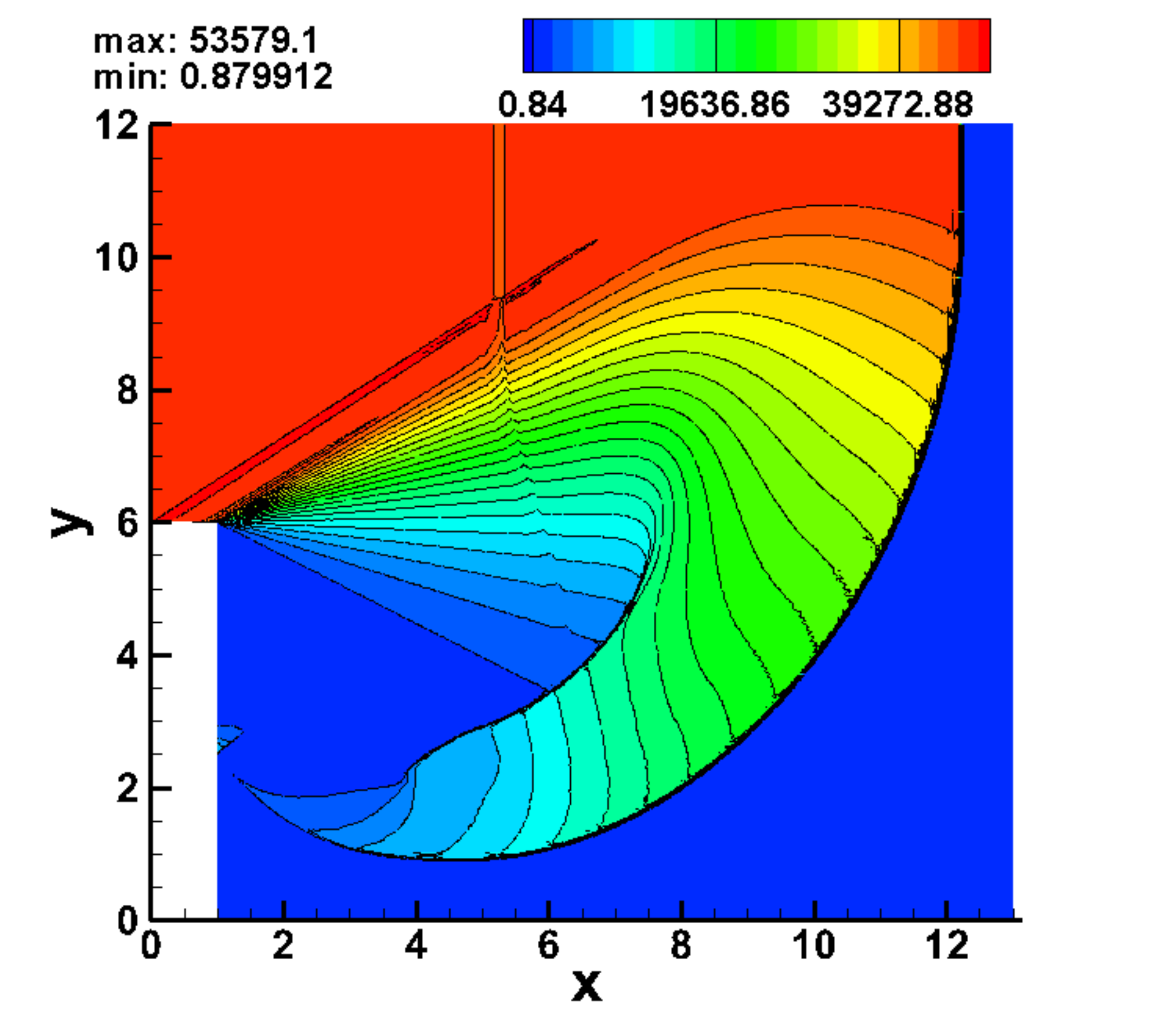} 
		\caption{}
	\end{subfigure}
	 \caption{Density and  pressure contours obtained with the PPESAD-p5 scheme for the inviscid flow (top row) and with the PPESAD-p4 scheme for the viscous shock diffraction  flow (bottom row) at $Ma=200$.} 
\label{densityPres_shockDiffraction_M200}
\end{figure}
For this smooth inviscid flow with periodic boundaries, both the ESSC and PPES schemes semi-discretely conserve the total entropy in the domain.  Note, however, 
the total entropy production obtained with the ESSC-p4 and PPES-p4 schemes with constant time step  $\Delta t = 2\mathrm{e}{-4}$ is of the order of $10^{-14}$ at the final time, because the 3rd-order SSP Runge-Kutta scheme used for approximating the time derivatives is not entropy conservative and violates this condition by the amount that is 
proportional to the truncation error of the temporal discretization.  
%
%
\begin{figure}[!ht] 
	\begin{subfigure}{0.5\textwidth}
		\includegraphics[width=0.9\linewidth]{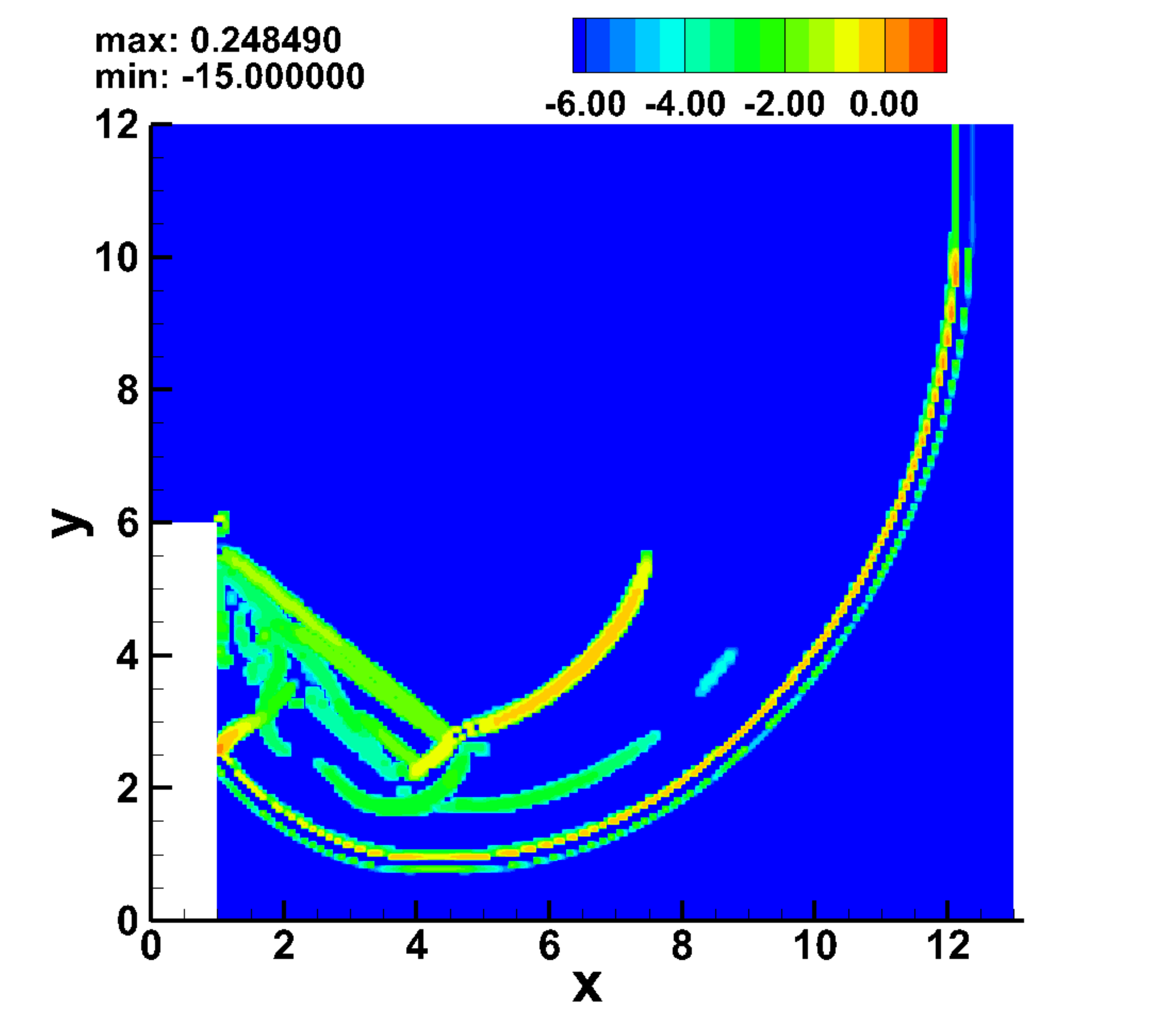} 
		\caption{}
	\end{subfigure}
	\begin{subfigure}{0.5\textwidth}
		\includegraphics[width=0.9\linewidth]{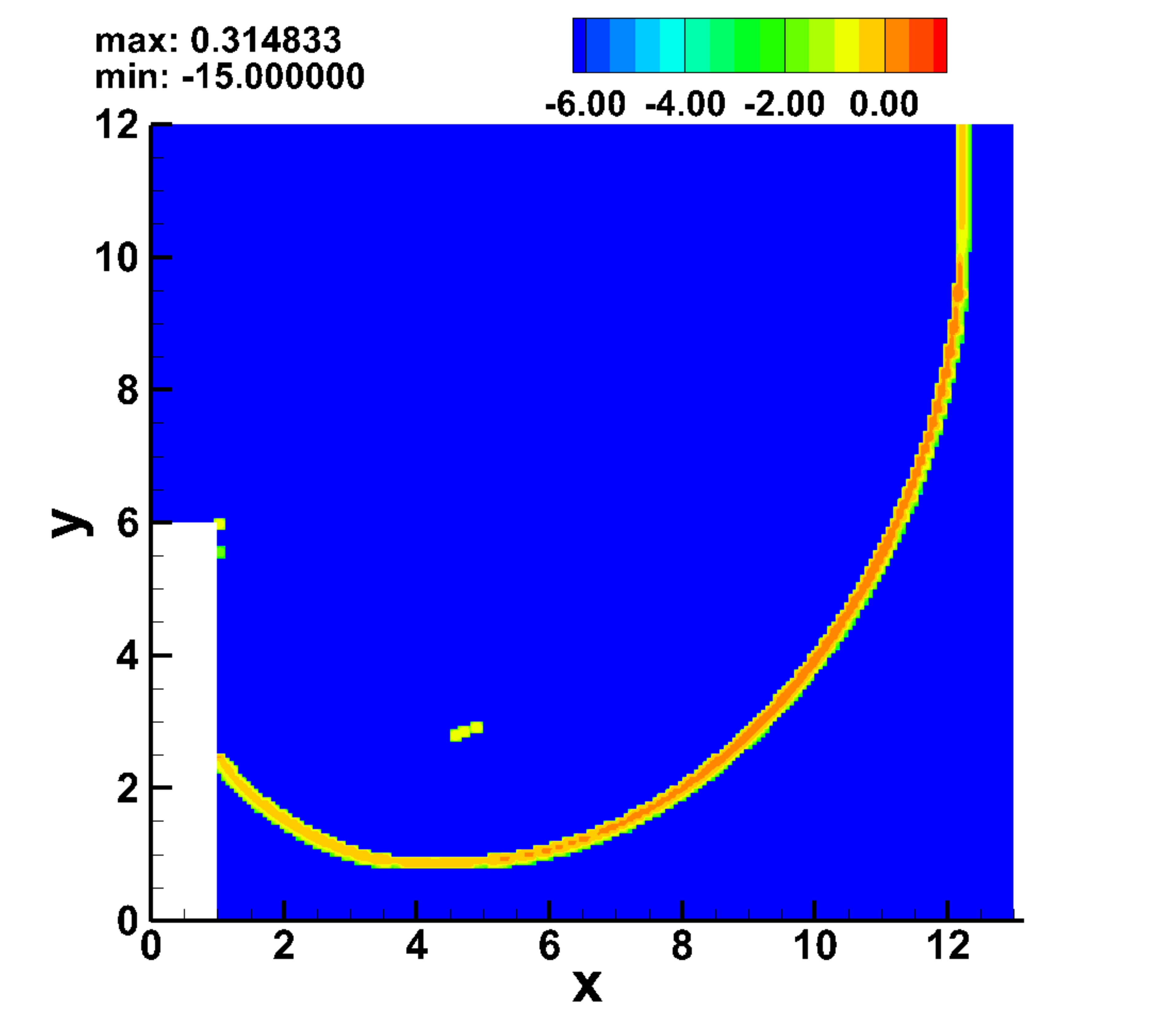}
		\caption{}
	\end{subfigure}
	\\
	\begin{subfigure}{0.5\textwidth}
		\includegraphics[width=0.9\linewidth]{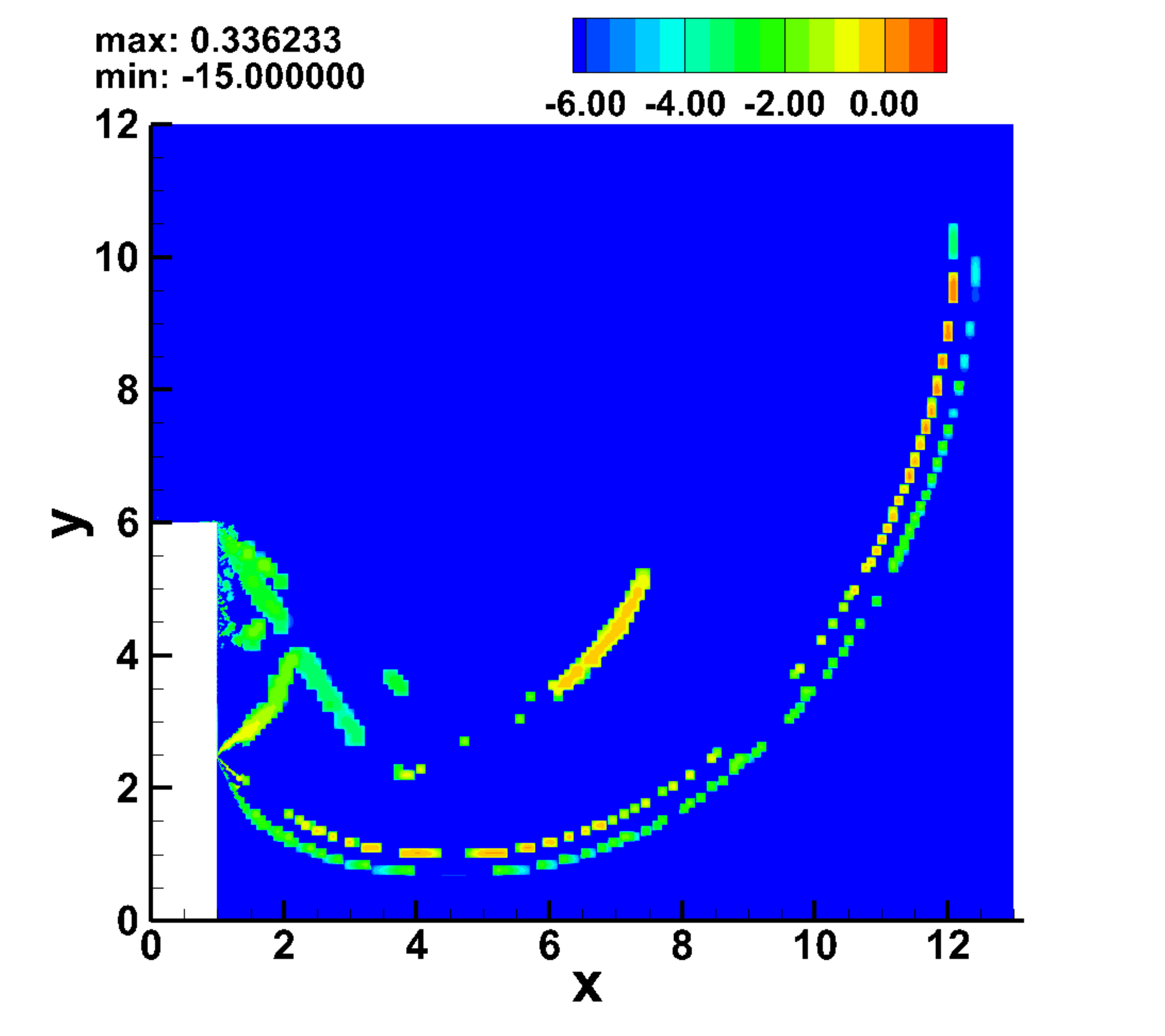} 
		\caption{}
	\end{subfigure}
	\begin{subfigure}{0.5\textwidth}
          \includegraphics[width=0.9\linewidth]{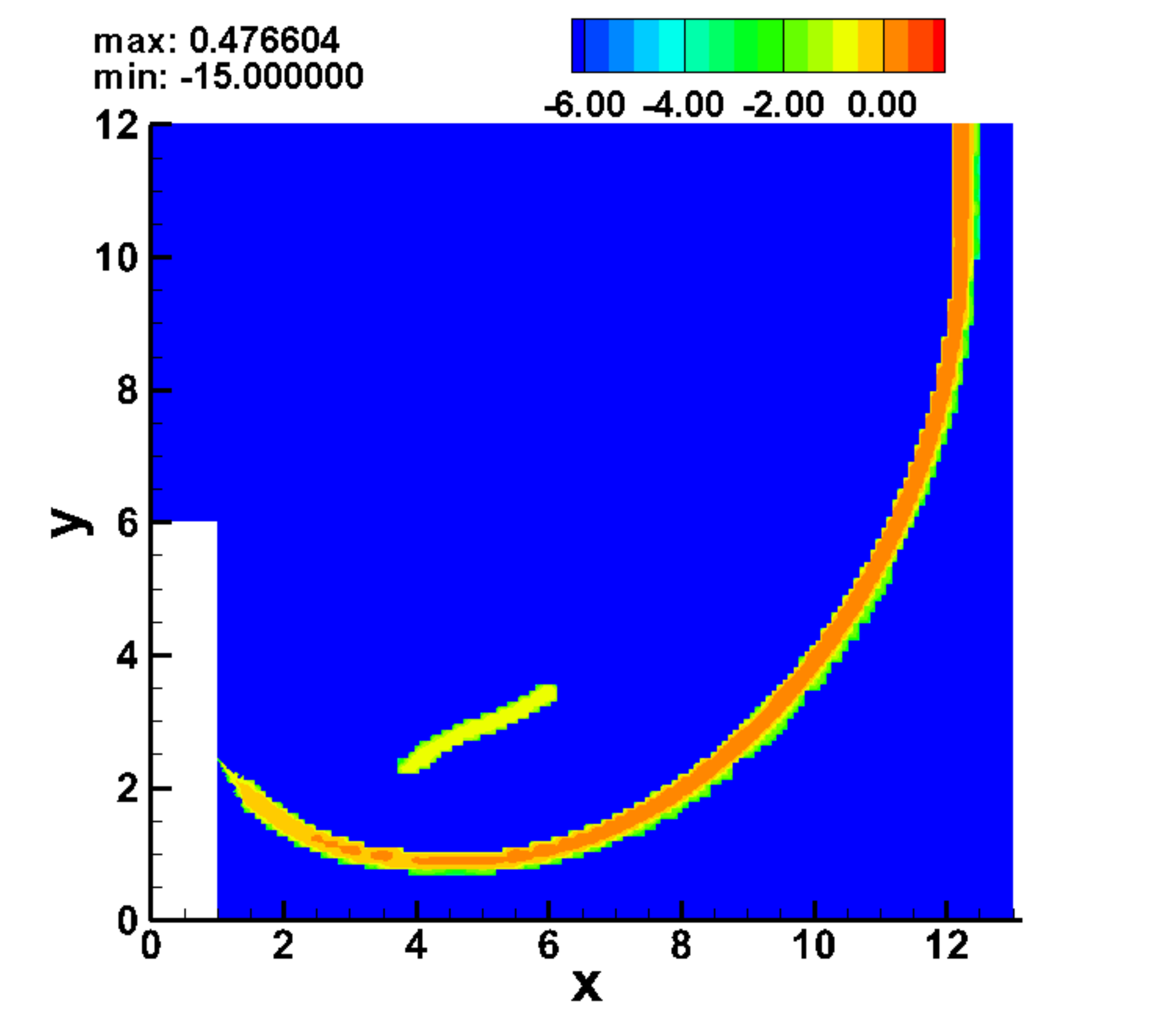}
		\caption{}
	\end{subfigure}
	 \caption{High-order (left column) and low-order (right column) artificial viscosities ($\log_{10}$) obtained with the PPESAD-p5 scheme for the inviscid flow (top row) and with the PPESAD-p4 scheme for the viscous shock diffraction flow (bottom row) at $Ma=200$.} 
\label{logmuad_shockDiffraction_M200}
\end{figure}

\subsection{$2$-D shock diffraction}
%
The next test problem is the diffraction of a rightward moving shock over a backward-facing step at the Mach number, $Ma=200$ . We consider both viscous and inviscid flow regimes.  
This is a very challenging problem that is characterized by the presence of both the strong discontinuities and regions with very low densities and pressures. 
If not dissipated properly, any high-order scheme can generate negative density and/or pressure values near the corner point and at the shock front. 
%
%
\begin{figure}[!h] 
	\begin{subfigure}{0.5\textwidth}
		\includegraphics[width=0.9\linewidth]{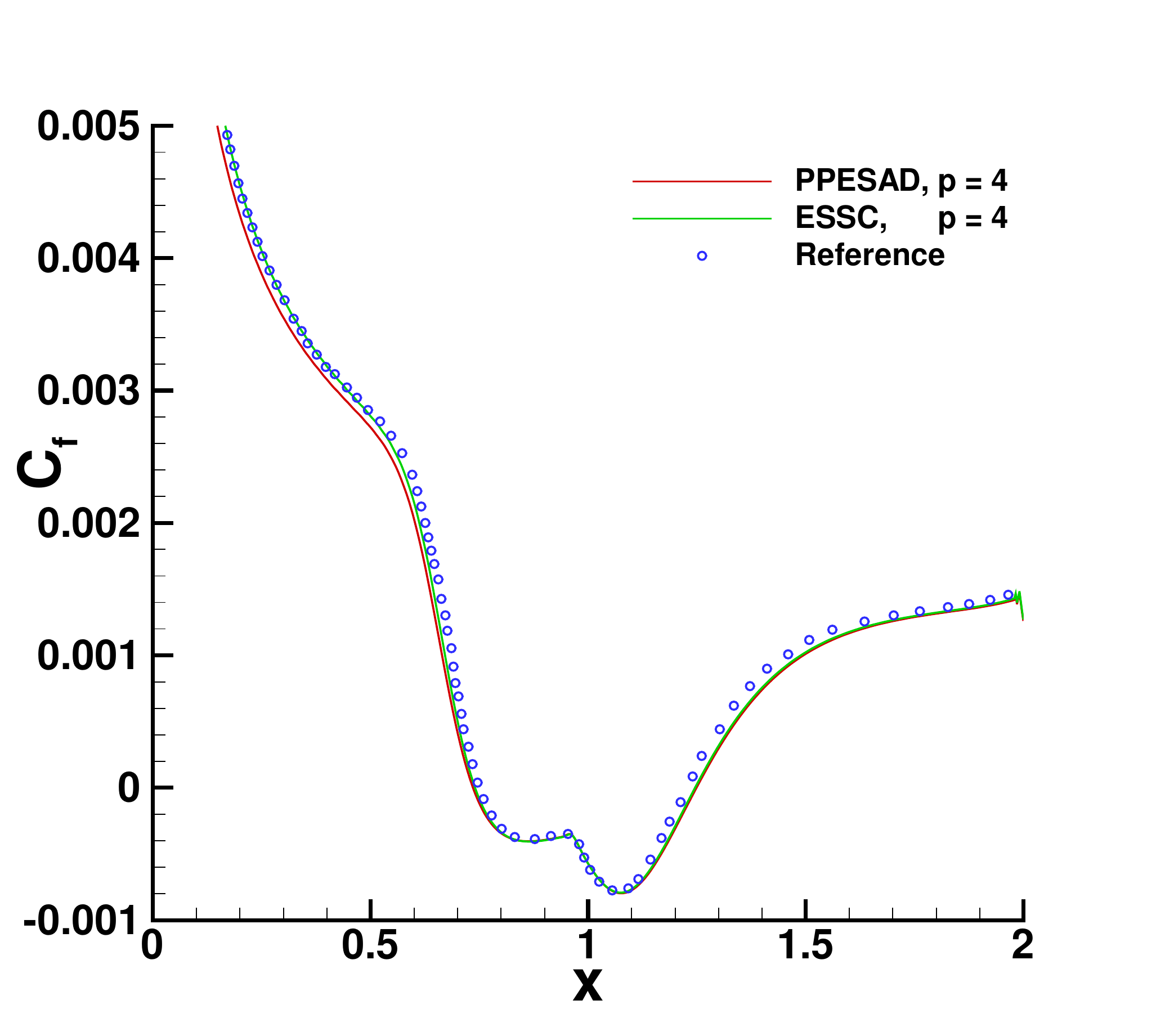} 
		\caption{}
	\end{subfigure}
	\begin{subfigure}{0.5\textwidth}
		\includegraphics[width=0.9\linewidth]{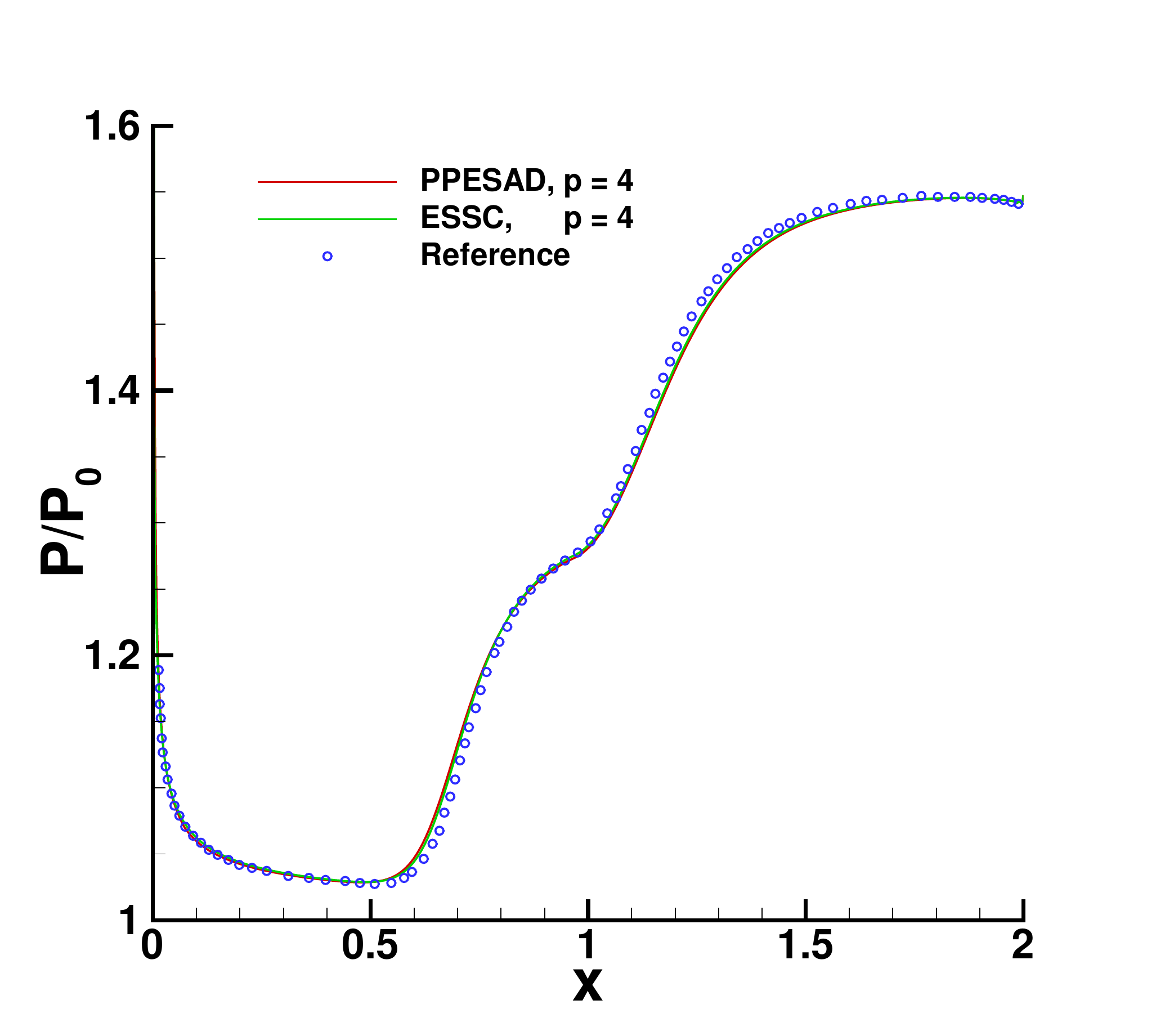}
		\caption{}
	\end{subfigure}
	 \caption{Comparison of wall skin friction (left panel) and pressure profiles obtained with the PESAD-p4 and ESSC-p4 schemes and the reference solution taken from \cite{SBLI_M2_proceedings} for the $Ma = 2.15$ SBLI problem.}
\label{skin_friction_pressure_M2pt15}
\end{figure}
In contrast to the results presented in \cite{Zhang}, we use the entropy stable adiabatic no-slip boundary conditions at the wall 
and penalize against the Blasius solution corresponding to $Ma=200$ at the inflow boundary for solving the Navier-Stokes equations. 
For the viscous flow case, the grid consists of 52944 elements and is clustered near the step surface,  so that the normal grid spacing at the wall is $2.67\times 10^{-3}$.
A uniform rectangular mesh with constant grid spacings $\Delta x = \Delta y$ and $40,000$ elements is used for the inviscid flow simulation. 
For the viscous flow, the Sutherland's law is used and the Reynolds and Prandtl numbers are set equal to $10^4$ and $0.75$, respectively.  

Unlike the ESSC scheme that fails to preserve the positivity of thermodynamic variables for both the inviscid and viscous shock diffraction flows at $Ma=200$,  the new PPESAD-p4 and PPESAD-p5 schemes captures  both the weak and strong and shocks as well as the contact discontinuity within one grid element practically without producing any spurious oscillations,
as one can see in Figure~\ref{densityPres_shockDiffraction_M200}.
Contours of the low- and high-order artificial viscosities of the PPESAD-p5 and PPESAD-p4 schemes for the inviscid and viscous shock diffraction flows at the final time are shown in Fig.~\ref{logmuad_shockDiffraction_M200}. As follows from these results, the artificial viscosity coefficient is at least 3 orders of magnitude smaller at the contact discontinuity than at the shock, thus indicating that the proposed physics-based artificial dissipation method is capable of distinguishing different types of waves.  
%
%
\begin{figure}[!h] 
	\begin{subfigure}{0.5\textwidth}
		\includegraphics[width=0.9\linewidth]{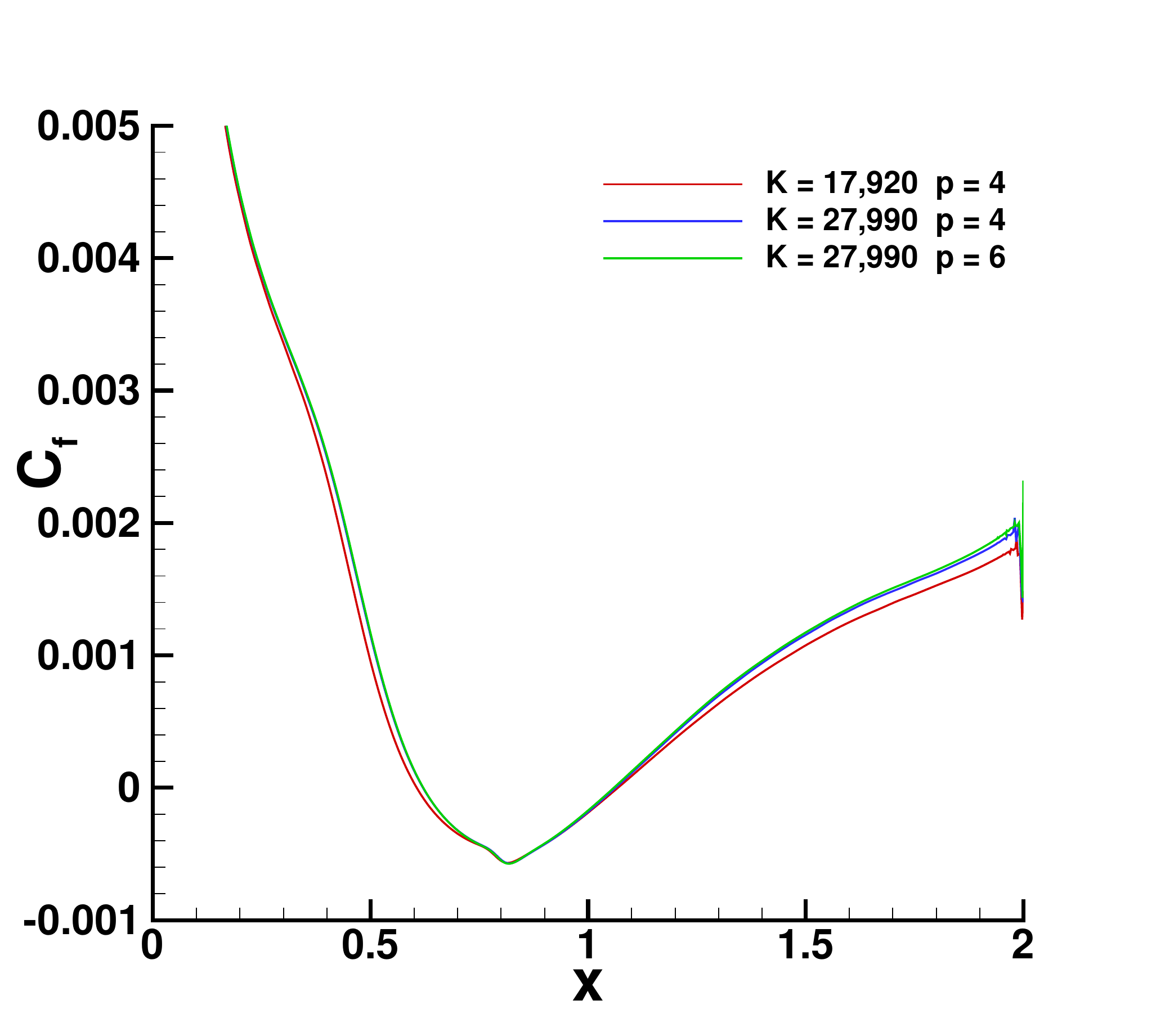} 
		\caption{}
	\end{subfigure}
	\begin{subfigure}{0.5\textwidth}
		\includegraphics[width=0.9\linewidth]{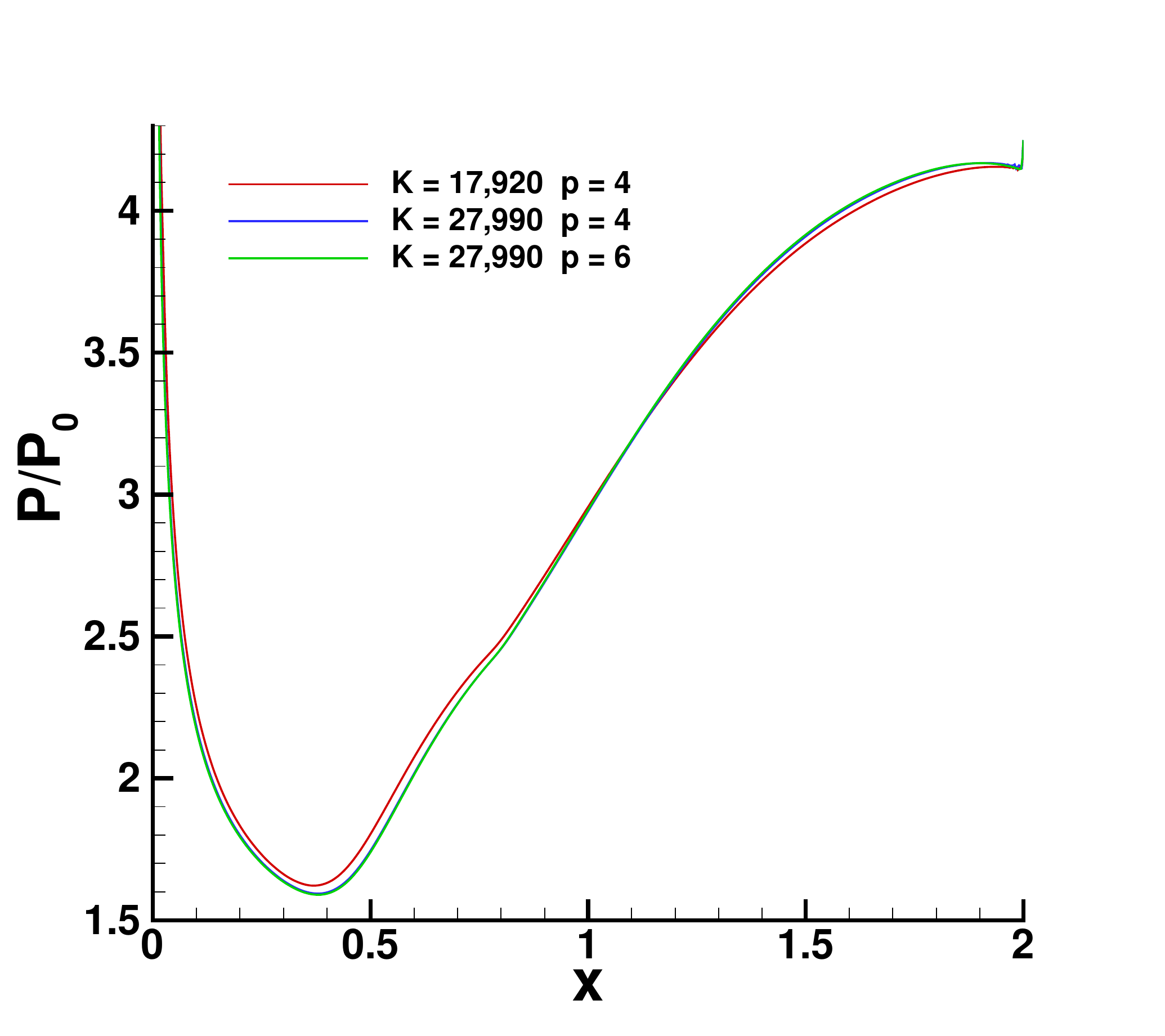}
		\caption{}
	\end{subfigure}
	 \caption{Wall skin friction (left panel) and pressure profiles computed with PPESAD-p4 and PPESAD-p6 schemes on the medium ($K = 17,920$) and fine ($K=27,990$) grids for the $Ma = 6.85$ SBLI problem.}
\label{skin_friction_pressure_M6pt85}
\end{figure}
Note that the artificial viscosity coefficient near the shock is spread out over a wider area for the viscous flow case.  This is mostly due to the fact that the viscous grid has less resolution in this region than the inviscid counterpart.

\subsection{$2$-D Shock/Boundary Layer Interaction }
To test the shock-capturing and positivity-preserving capabilities of the proposed PPESAD scheme and its ability to accurately predict boundary layers, we consider the interaction of a shock wave and a laminar boundary layer for two Mach numbers, $Ma =2.15$ and $Ma=6.85$, and two impinging angles $\theta=30.8^{\circ}$ and $\theta=11.8^{\circ}$, accordingly. For all shock/boundary layer interaction (SBLI) simulations, the physical viscosity is computed by using the Sutherland's law and the Prandtl and Reynolds numbers are set equal to $0.72$, and $10^5$. The $Ma =2.15$ test case is computed on a grid with $17,050$ elements, while the $Ma=6.85$ simulations are performed on $17,920-$ and $27,990-$element grids. 
%
%
\begin{figure}[!h] 
	\begin{subfigure}{\textwidth}
		\includegraphics[width=\linewidth]{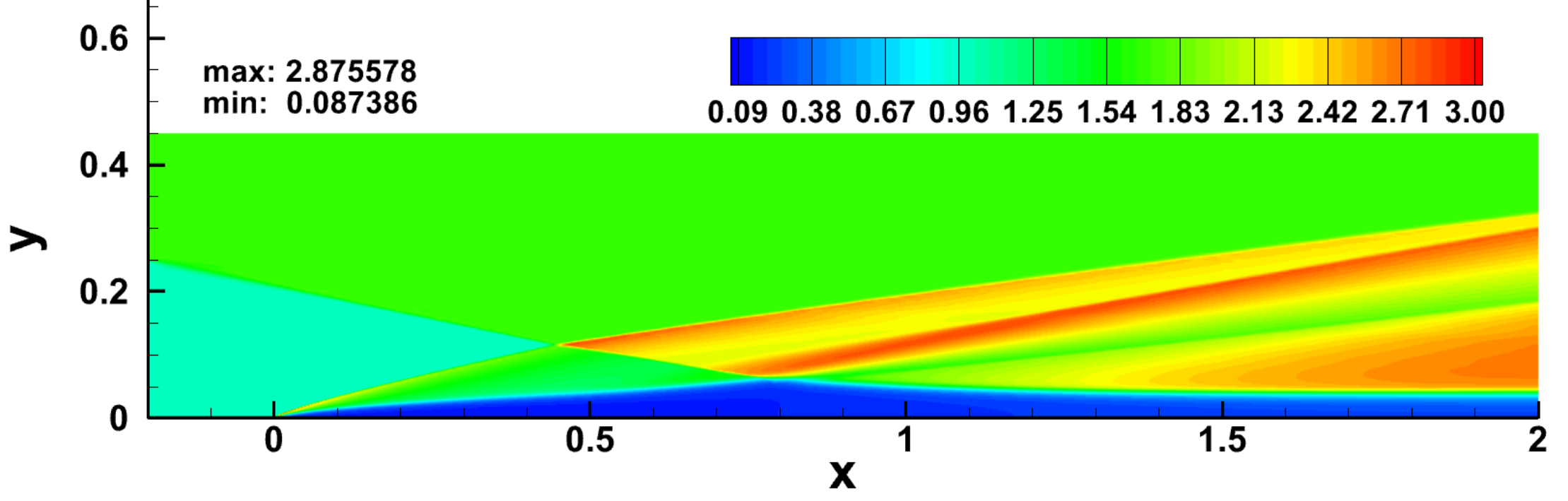} 
		\caption{}
	\end{subfigure}
	 \caption{Density 
	 contours obtained with the PPESAD-p6 scheme on the $27,990-$element grid for the $Ma = 6.85$ SBLI problem.} 
\label{densityPresMach_wholeDomain_M6pt85}
\end{figure}
%
%
%
\begin{figure}[!h] 
	\begin{subfigure}{\textwidth}
		\includegraphics[width=\linewidth]{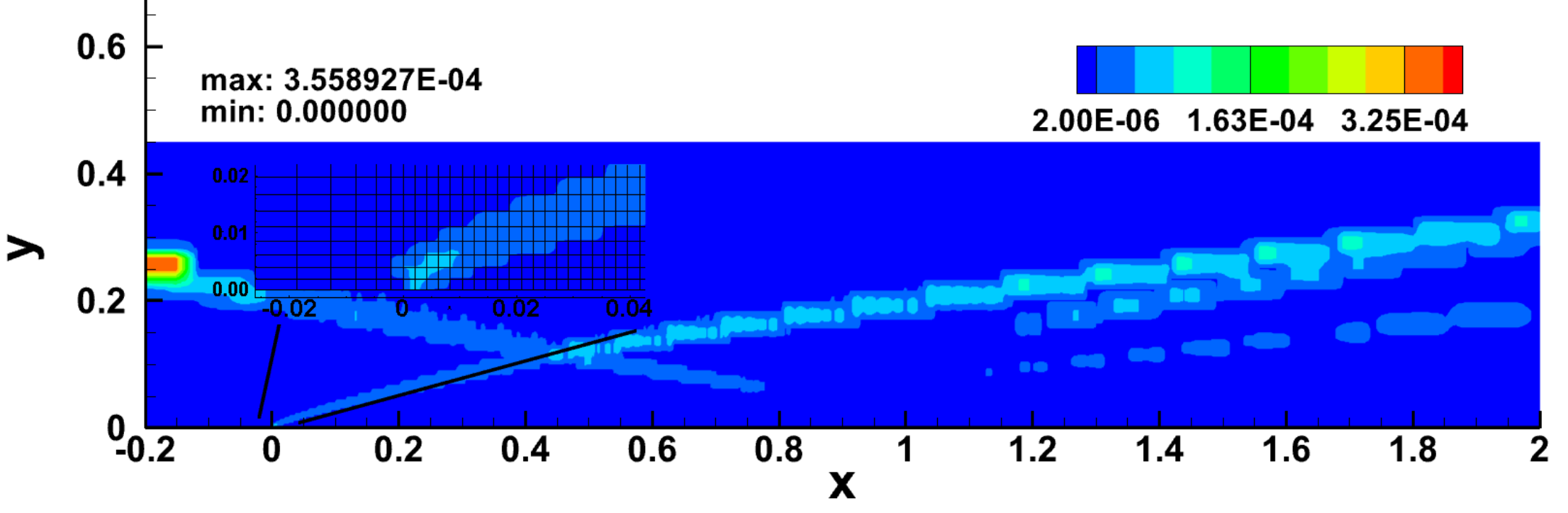} 
		\caption{}
	\end{subfigure}
	 \caption{High-order
	  artificial viscosity coefficient of the PPESAD-p6 scheme on the fine grid for the $Ma = 6.85$ SBLI problem.} 
\label{muAD_wholeDomain_M6pt85}
\end{figure}

Wall skin friction and pressure profiles computed with the PPESAD-p4 and ESSC-p4 schemes and the reference solution obtained using the $p=6$ discontinuous Galerkin method \cite{SBLI_M2_proceedings} for the $Ma = 2.15$ and $\theta=30.8^0$ SBLI problem are presented in Fig.~\ref{skin_friction_pressure_M2pt15}. As follows from this comparison,
the  ESSC-p4 and PPESAD-p4 density and pressure profiles are nearly indistinguishable from each other and demonstrate excellent agreement with the reference solution, thus indicating that the PPESAD-p4 scheme does not over-dissipate the SBLI solution.

Next, we consider the SBLI problem at $Ma=6.85$ and $\theta =11.8^{\circ}$.
For this test case, the ESSC scheme is unable to maintain positivity of thermodynamic variables for $p \ge 2$.  
Therefore, to evaluate the accuracy of the discrete solution, we compare PPESAD solutions obtained on medium ($17,920$ elements) and fine (27,990 elements) grids for polynomial orders $p=4$ and $p=6$. 
%
%
\begin{figure}[!h] 
	\begin{subfigure}{0.5\textwidth}
		\includegraphics[width=0.9\linewidth]{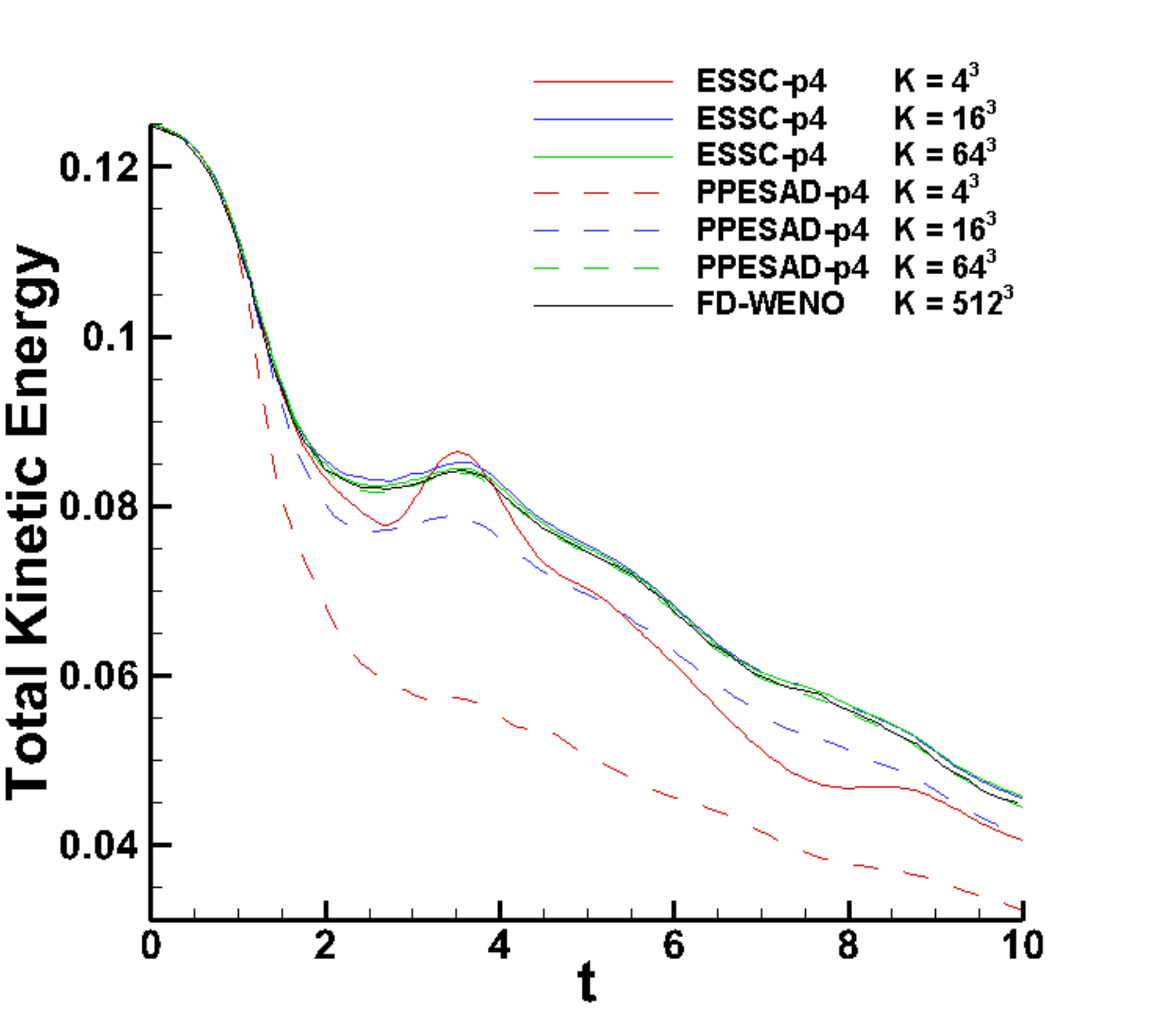} 
		\caption{}
	\end{subfigure}
	\begin{subfigure}{0.5\textwidth}
		\includegraphics[width=0.9\linewidth]{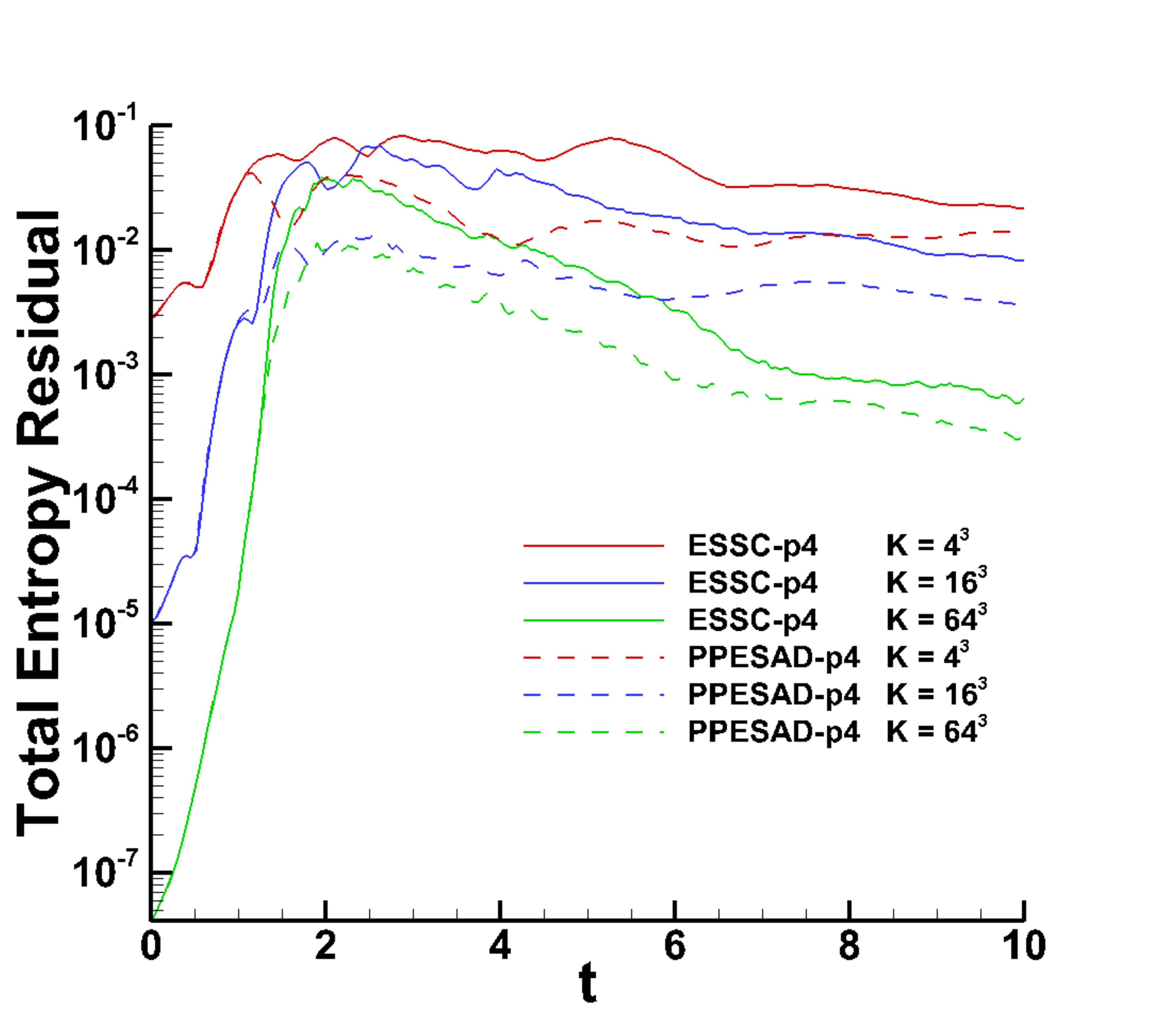}
		\caption{}
	\end{subfigure}
	 \caption{Time histories of the total kinetic energy (left) and total  entropy residual obtained with the ESSC-p4 and PPESAD-p4 schemes and the reference solution computed using FD-WENO scheme\cite{peng2018effects} for the $Ma = 2$ TGV problem. }
\label{KE_Sres_M2TGV}
\end{figure}
%
%
\begin{figure}[!h] 
	\begin{subfigure}{0.5\textwidth}
		\includegraphics[width=0.9\linewidth]{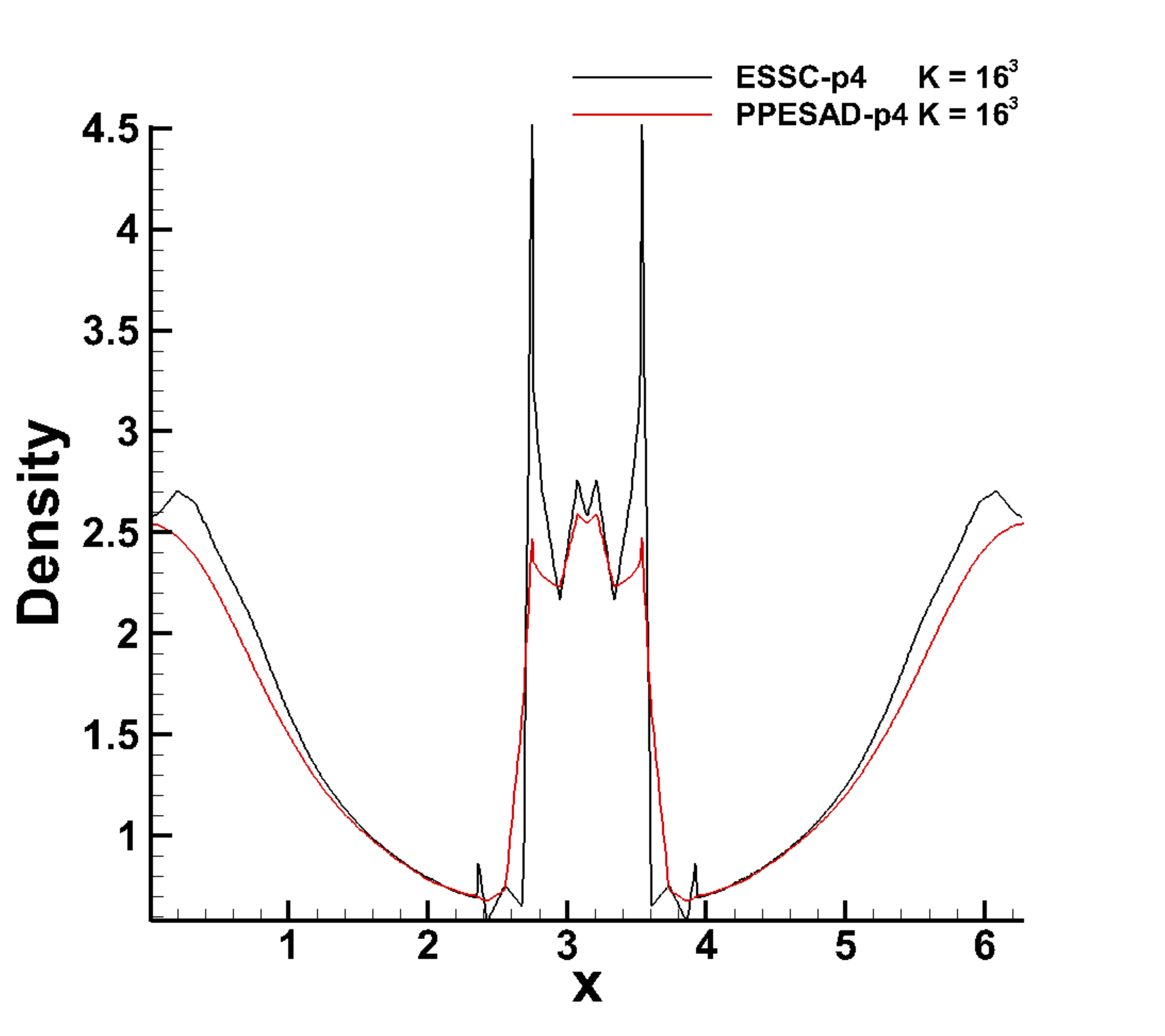} 
		\caption{}
	\end{subfigure}
	\begin{subfigure}{0.5\textwidth}
		\includegraphics[width=0.9\linewidth]{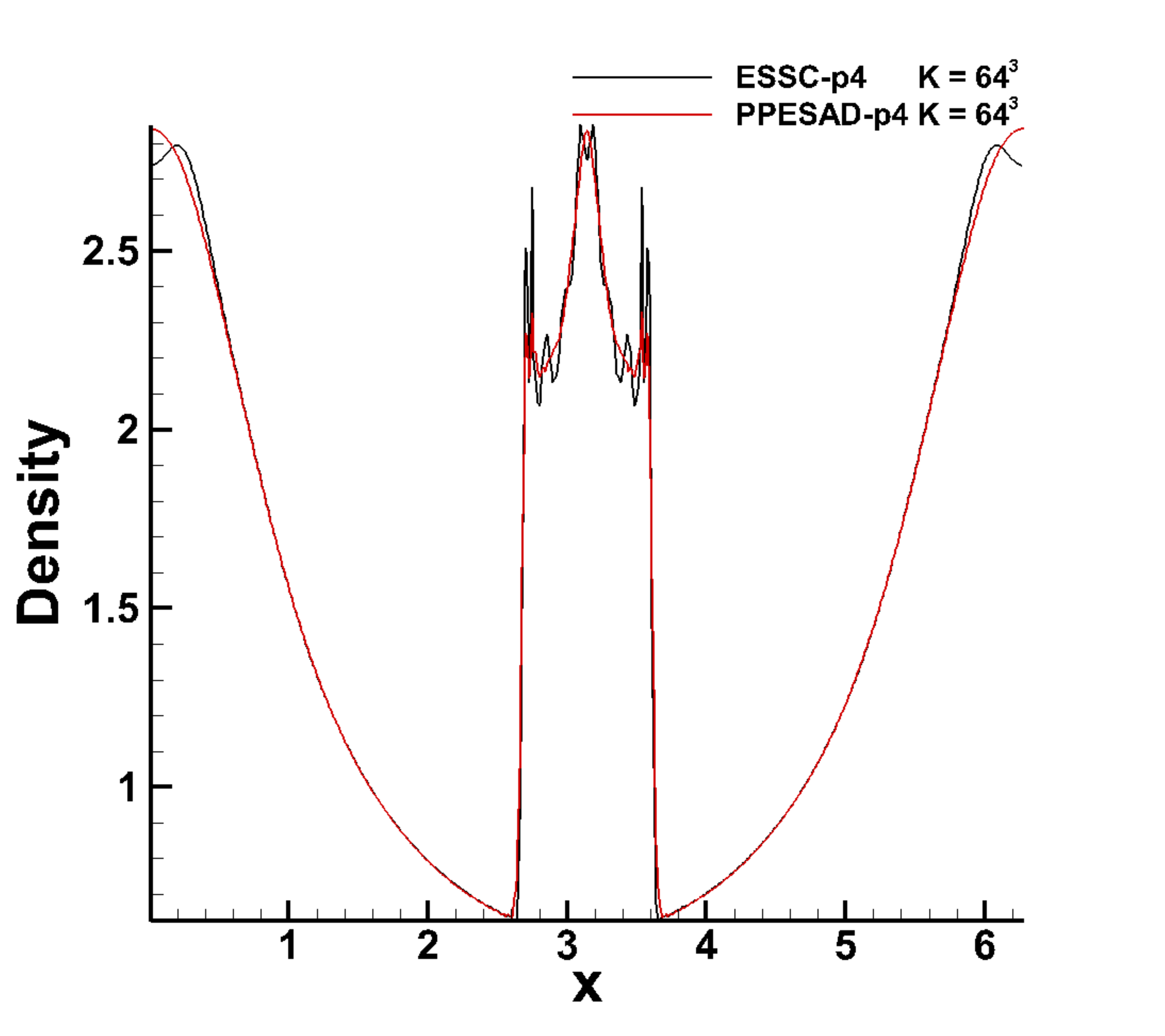}
		\caption{}
	\end{subfigure}
      \caption{Density profiles computed with the PPESAD-p4 and ESSC-p4 schemes on the $16^3$ (left panel) and $64^3$ grids for the $Ma=2$ TGV problem. }
\label{densPresVel_M2TGV}
\end{figure}
Figure~\ref{skin_friction_pressure_M6pt85} shows the comparison of skin friction and pressure profiles computed with the PPESAD-p4 and PPESAD-p6 schemes on the medium and fine grids. The forth-order ($p=4$) medium and fine grid solutions agree very well with each other, while the PPESAD-p4 and PPESAD-p6 solutions on the fine grid are 
nearly identical. Density 
contours computed with the PPESAD-p6  scheme on the fine grid are depicted in Fig.~\ref{densityPresMach_wholeDomain_M6pt85}. Similar to the previous test problems, the proposed PPESAD-p6 scheme demonstrates excellent shock-capturing capabilities and provides nearly nonoscillatory solution for this high-Mach-number SBLI flow. High-order artificial viscosity contours that correspond to the discrete solution shown in Fig.~\ref{densityPresMach_wholeDomain_M6pt85} are presented in Fig.~\ref{muAD_wholeDomain_M6pt85}. 
%
%
\begin{figure}[!h] 
	\begin{subfigure}{0.5\textwidth}
		\includegraphics[width=0.9\linewidth]{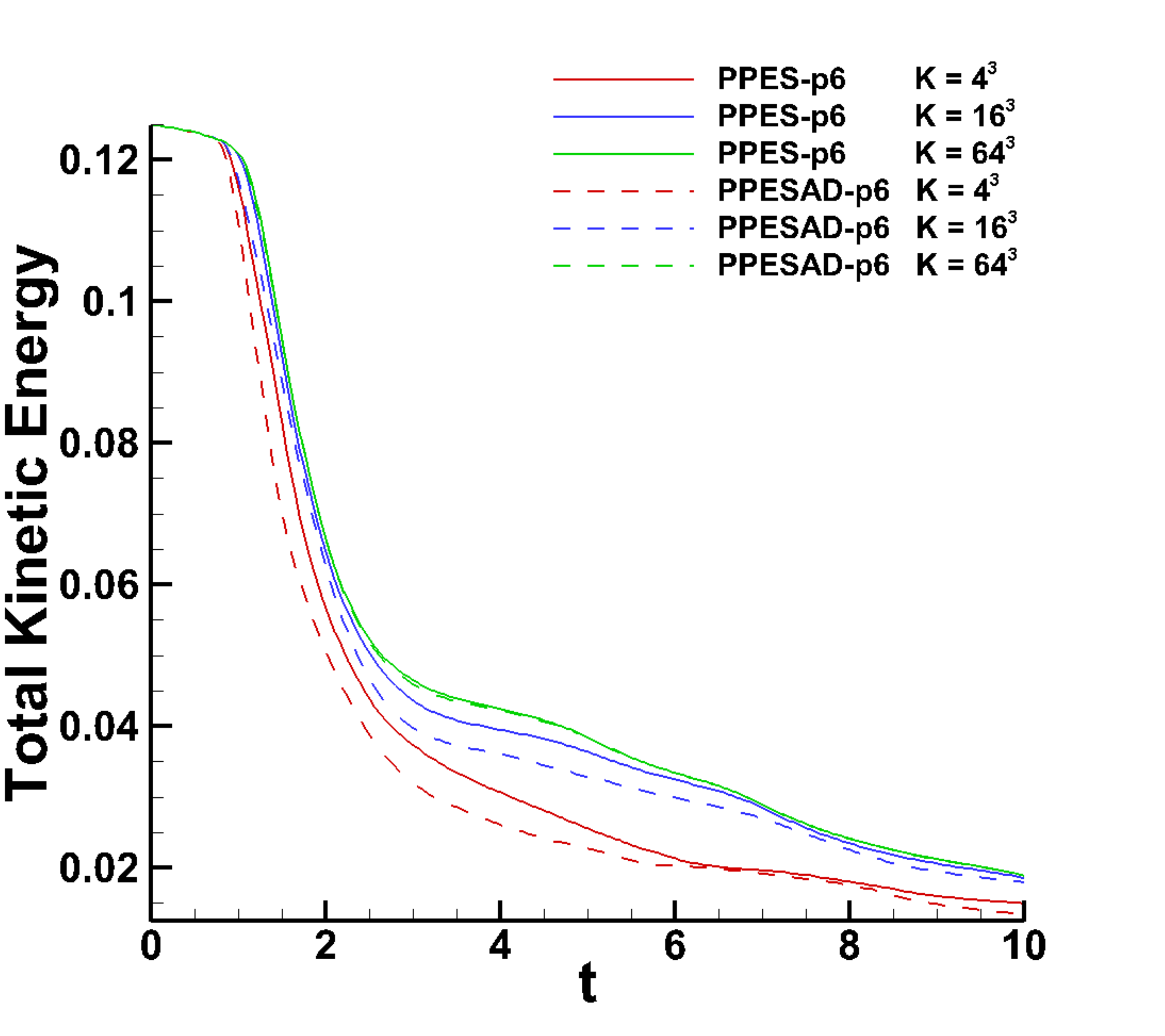} 
		\caption{}
	\end{subfigure}
	\begin{subfigure}{0.5\textwidth}
		\includegraphics[width=0.9\linewidth]{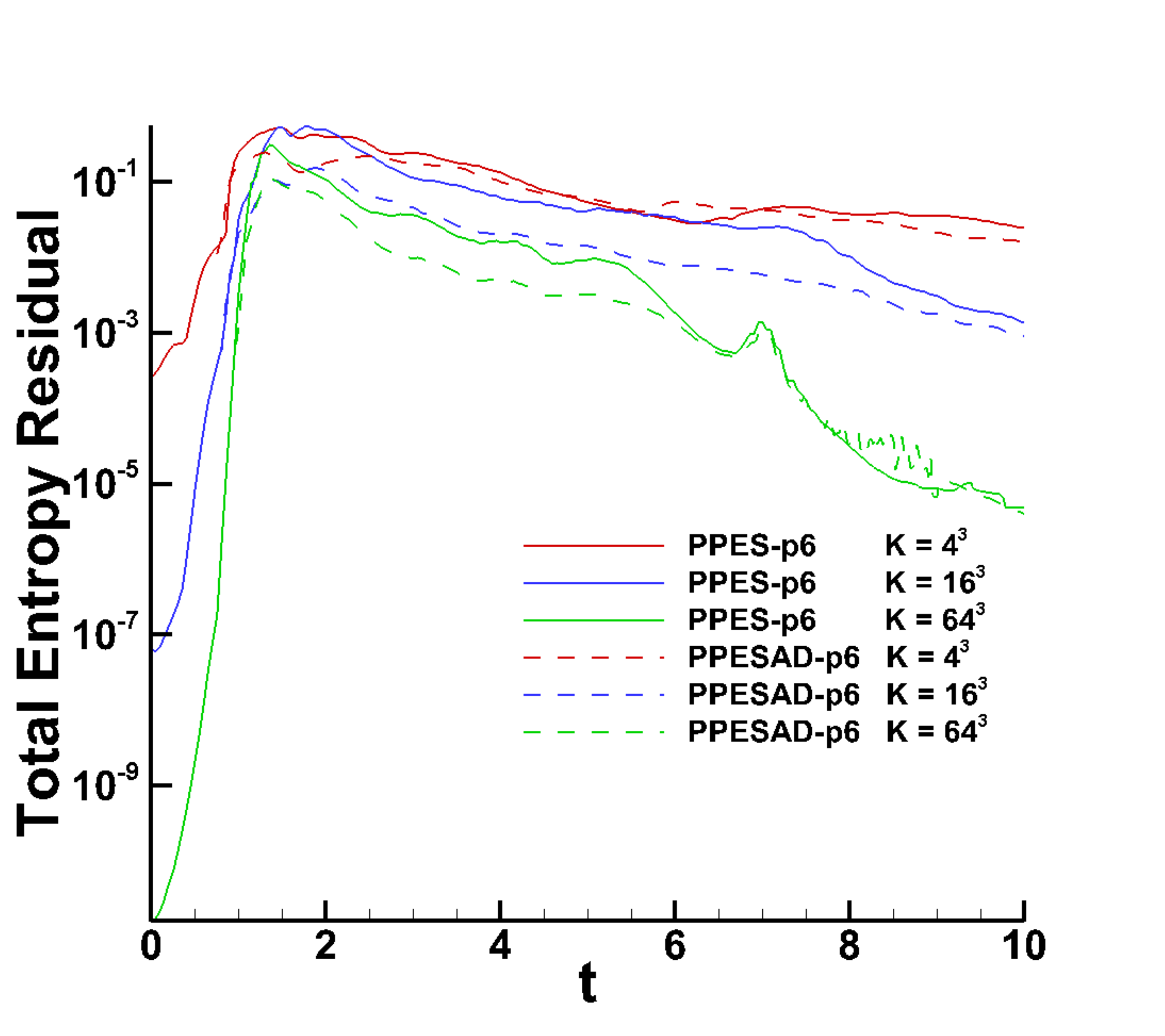}
		\caption{}
	\end{subfigure}
	 \caption{Time histories of the total kinetic energy (left panel) and total  entropy residual obtained with the PPES-p6 and PPESAD-p6 schemes for the $Ma = 10$ TGV problem on $4^3$, $16^3$ and $64^3$ grids.}
\label{KE_Sres_M10TGV}
\end{figure}
%
%
\begin{figure}[!h] 
	\begin{subfigure}{0.5\textwidth}
		\includegraphics[width=0.9\linewidth]{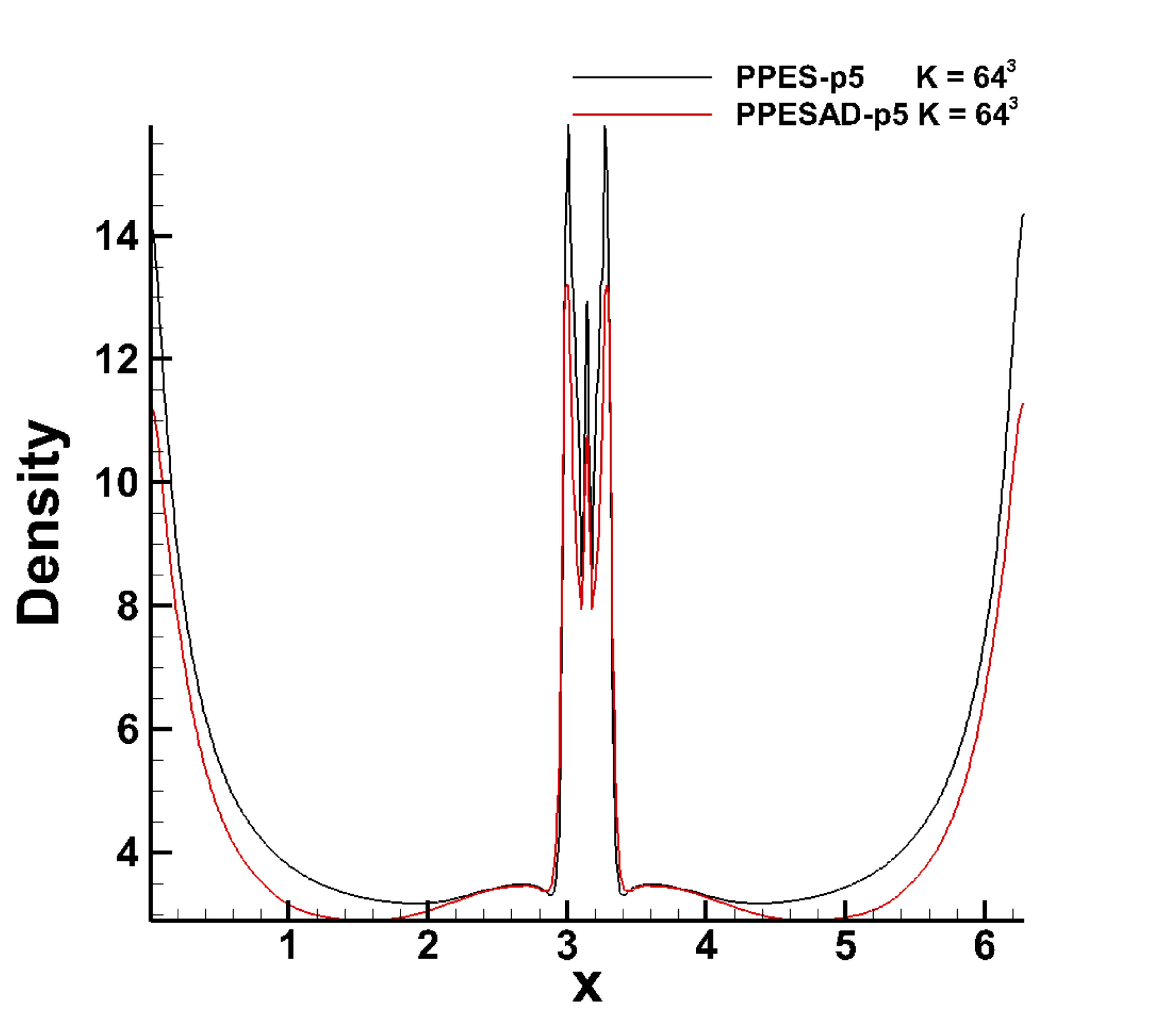} 
		\caption{}
	\end{subfigure}
	\begin{subfigure}{0.5\textwidth}
		\includegraphics[width=0.9\linewidth]{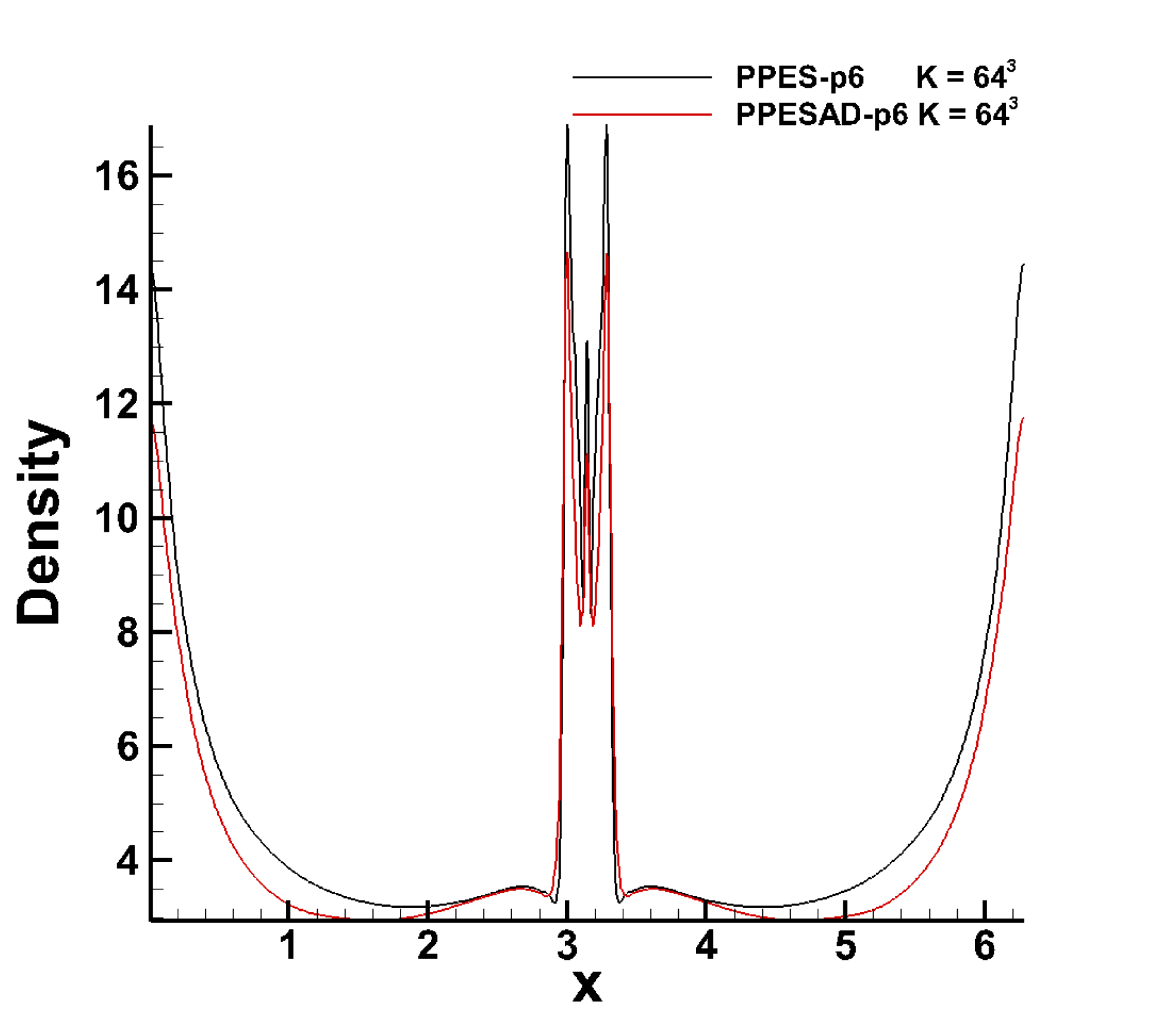}
		\caption{}
	\end{subfigure}
	 \caption{Density profiles computed with the  $p=5$ (left column) and $p=6$ PPESAD and PPES schemes for the $Ma=10$ TGV problem on the $64^3$  grid. }
\label{densPresVel_M10TGV}
\end{figure}
The low-order artificial viscosity is nonzero only in six elements near the flat plate leading edge and therefore not presented herein.
Furthermore,  the flux limiter in Eq.~\eqref{EQN_LIMSCHEME} is equal to 1 everywhere except for  the single element near the leading edge.  The high-order artificial viscosity of the PPESAD-p6 scheme is mostly zero everywhere except for at the shocks.  Note that the velocity and temperature limiters have never been turned on for this test case.

\subsection{3-D supersonic Taylor-Green vortex flow}
The last test problem is the 3-D viscous, compressible Taylor-Green vortex (TGV) flow at Mach numbers $Ma = 2$ and $Ma = 10$. This problem is considered to test how the proposed PPESAD scheme performs for under-resolved turbulent flows with strong shock waves.  We adopt the TGV flow parameters used in \cite{peng2018effects} and compare our results for the $Ma = 2$ case first.  For all TGV flows considered, the Sutherland's law is used for the physical viscosity, and the Reynolds and Prandtl numbers are $400$ and $0.7$, respectively.  
The problem is solved on the periodic cube, $0 \leq x,y,z \leq 2\pi$, with the following initial conditions: 
$[\rho, \bfnc{V}, T] = [1 + \frac{1}{16}(\cos 2x + \cos 2y) (\cos 2z + 2), \sin x \cos y \cos z, - \cos x \sin y \cos z, 0, 1]^\top$.

The comparison of kinetic energy histories obtained with the ESSC-p4 and PPESAD-p4 schemes and the hybrid 8th-order compact finite difference/ 7th-order weighted essentially nonoscillatory (WENO) scheme  \cite{peng2018effects} are presented in Fig.~\ref{KE_Sres_M2TGV}a. As follows from this comparison, the kinetic energy computed using the new spectral collocation scheme on the uniform $64^3$-element grid is practically  identical to that of the ESSC-p4 scheme on the same grid and in an excellent agreement with that computed by the 7th-order FD-WENO scheme on the $512^3$-element grid. 
On the $4^3$ and $16^3$grids, the PPESAD-p4 scheme dissipates the total kinetic energy more than the ESSC-p4 scheme.  However, this does not imply that the ESSC-p4 solution is overall more accurate, which can be observed in Figs.~\ref{KE_Sres_M2TGV}b and \ref{densPresVel_M2TGV}. Indeed,
the entropy residual obtained with ESSC-p4 scheme is significantly larger than that of the PPESAD-p4 scheme, thus indicating the larger discretization error.
The ESSC-p4 solution on the $16^3$ grid contains large spurious overshoots that are not present in the corresponding PPESAD-p4 solution. It should be emphasized that both solutions converge to each other as the grid is refined.

For $Ma=10$, the ESSC scheme fails to preserve the positivity of thermodynamic variables for the viscous TGV problem, even if $p=1$.
Therefore, we compare the PPESAD and PPES solutions for the $Ma = 10$  case.  Overall, Figs.~\ref{KE_Sres_M10TGV} and \ref{densPresVel_M10TGV} show a similar behavior
observed for the $Ma=2$ test case. On coarse meshes, PPESAD-p6 solution is less prone to spurious oscillations than the PPES-p6 counterpart, while both solutions converge to each other as the grid is refined.

\bigskip  
{\bf Acknowledgments}
The first author was supported by the Virginia Space Grant Consortium
Graduate STEM Research Fellowship and the Science, Mathematics
and Research for Transformation (SMART) Scholarship. The second author
acknowledges the support from Army Research Office through grant
W911NF-17-0443.



\begin{thebibliography}{99}

\bibitem{UY_1Dlow} J. Upperman and N. K. Yamaleev, ``Positivity-preserving entropy stable schemes for the 1-D compressible Navier-Stokes equations: First-order approximation,''  submitted to {\em J. Comput. Phys.}.

\bibitem{UY_1Dhigh} J. Upperman and N. K. Yamaleev, ``Positivity-preserving entropy stable schemes for the 1-D compressible Navier-Stokes equations: High-order flux limiting,''  submitted to {\em J. Comput. Phys.}.

\bibitem{UY_3Dlow} J. Upperman and N. K. Yamaleev, ``First-order positivity-preserving entropy--stable spectral collocation scheme for the 3-D compressible Navier-Stokes equations,''  {\em arXiv}:2111.03239v1 [math.NA], 2021.

\bibitem{Svard} M. Sv\"ard, ``A convergent numerical scheme for the compressible Navier-Stokes equations,'' {\em SIAM J. Numer. Anal.}, Vol. 54, No. 3, 2016, pp. 1484--1506.


\bibitem{GHKL} D. Grapsas, R. Herbin, W. Kheriji, and J.-C. Latch\'e, ``An unconditionally stable staggered pressure correction scheme for the compressible Navier-Stokes equations,'' 
{\em SMAI J. Comput. Math.}, vol. 2, 2016, pp.51--97.

\bibitem{GMPT}  J.-L. Guermond, M. Maier, B. Popov, I. Tomas, ``Second-order invariant domain preserving approximation of the compressible Navier-Stokes equations,'' arXiv:2009.06022v1, 2020.

\bibitem{Zhang} X. Zhang, ``On positivity-preserving high order discontinuous Galerkin schemes for compressible Navier-Stokes equations,'' {\em J. Comput. Phys.}, Vol. 328, 2017, pp. 301-343.

\bibitem{Brenner1} H. Brenner, ``Navier-Stokes revisited,'' {\em Physica A}, Vol. 349, 2005, pp. 60--132.

\bibitem{FV} E. Feireisl and A. Vasseur, ``New perspectives in fluid dynamics: Mathematical analysis of a model proposed by Howard Brenner,'' {\em Adv. Math. Fluid Mech.}, New directions in mathematical fluid dynamics, 2009, pp. 153--179.


\bibitem{GUERPOP2014} J.-L. Guermond and B. Popov, ``Viscous regularization of the Euler equations and entropy principles,'' {\em SIAM J. Appl. Math.}, Vol. 74, 2014, pp. 284--305.

\bibitem{GCLSS2019} N. K. Yamaleev, D. C. Del Rey Fernandez, J. Lou, and M. H. Carpenter, ``Entropy stable spectral collocation schemes for the 3-D Navier-Stokes equations on dynamic unstructured grids," {\em J. Comput. Phys.}, Vol. 399, 2019, 108897.

\bibitem{CFNF} M. H. Carpenter, T. C. Fisher, E. J. Nielsen, and S. H. Frankel, ``Entropy Stable Spectral Collocation Schemes for the Navier-Stokes Equations: Discontinuous Interfaces,'' {\em SIAM J. Sci. Comput.}, Vol. 36, No. 5, 2014, pp. B835--B867.

\bibitem{YC3} N. K. Yamaleev and M. H. Carpenter, ``A family of fourth-order entropy stable non-oscillatory spectral collocation schemes for the 1-D Navier-Stokes equations,'' {\em J. of Comput. Phys.}, Vol. 331, 2017, pp. 90--107.

\bibitem{FCNYS} T. C. Fisher, M. H. Carpenter, J. Nordstr\"om, N. K. Yamaleev, and R. C. Swanson, ``Discretely conservative finite-difference formulations for nonlinear conservation laws in split form: Theory and boundary conditions,'' {\em J. Comput. Phys.}, Vol. 234, 2013, pp. 353--375.

\bibitem{TAD2003} E. Tadmor,``Entropy stability theory for difference approximations of nonlinear conservation laws and related time-dependent problems," {\em Acta. Numer.}, Vol. 12, 2003, pp. 451--512.

\bibitem{Chand} P. Chandrashekar, ``Kinetic energy preserving and entropy stable finite volume schemes for compressible Euler and Navier-Stokes equations,'' {\em Commun. Comput. Phys.}, Vol. 14.5, 2013, pp. 1252-1286.


\bibitem{CFNPSY} M. H. Carpenter and T. C. Fisher and E. J. Nielsen and M. Parsani and M. Sv\"ard and N. K. Yamaleev, ``Entropy Stable Summation-by-Parts Formulations for Compressible Computational Fluid Dynamics,'' Handbook of Numerical Analysis, Vol. 17, 2016, pp. 495--524.

\bibitem{TLGCL} P. D. Thomas and C. K. Lombard, ``Geometric conservation law and its application to flow computations on moving grids," {\em AIAA J.}, Vol. 17, 1979, pp. 1030--1037.

\bibitem{CS} B. Cockburn and C.-W. Shu, ``The local discontinuous Galerkin method for time-dependent
convection-diffusion systems, {\em SIAM J. Numer. Anal.}, Vol. 35, 1998, pp. 2440–2463.

\bibitem{MER} M. L. Merriam, ``An Entropy-Based Approach to Nonlinear Stability,'' Tech. report TM 101086, NASA, 1989.

\bibitem{UY} J. Upperman and N. K. Yamaleev, ``Entropy stable artificial dissipation based on Brenner regularization of the Navier-Stokes equations,''  {\em J. Comput. Phys.}, Vol. 393, 2019, pp. 74--91.

\bibitem{JUthesis} J. K. Upperman, ``High-order Positivity-preserving $L_2$-stable Spectral Collocation Schemes for the 3-D compressible Navier-Stokes equations,'' Ph.D. thesis, Old Dominion University, 2021.

\bibitem{SSP} C.-W. Shu, ``Total-variation-diminishing time discretizations,'' {\em SIAM J. Sci. Stat. Comput.}, Vol. 9,  1988.

\bibitem{SBLI_M2_proceedings} F. Renac, ``BL2 - Laminar shock-boundary layer interaction,'' 4th International Workshop on High-Order CFD Methods, 2016.

\bibitem{peng2018effects} N. Peng and Y. Yang, ``Effects of the Mach number on the evolution of vortex-surface fields in compressible Taylor-Green flows,'' {\em Physical Review Fluids},
Vol. 3, No. 1, 2018.

\end{thebibliography}
\end{document}